\documentclass[final]{siamart1116}

\usepackage{lipsum}
\usepackage{amsfonts}
\usepackage{graphicx}
\usepackage{caption}
\usepackage{subcaption}
\usepackage{cleveref}
\usepackage{algorithmic}
\usepackage[mathscr]{eucal}
\usepackage{lipsum}
\usepackage{multirow}

\numberwithin{theorem}{section}

\newcommand{\argmin}{\operatornamewithlimits{argmin}}
\DeclareMathOperator{\diag}{diag}

\def\R{{\mathbb R}}
\def\C{{\mathbb C}}
\def\N{{\mathbb N}}

\def\Z{{\mathbb Z}}

\def\dom{\mathrm{dom}}

\def\s{\sigma}
\def\f{\frac}
\def\p{\partial}

\def\na{\nabla}
\def\la{\langle}
\def\ra{\rangle}

\def\bi{{\mathbf i}}

\def\mA{{\mathcal A}}

\def\mC{{\mathcal C}}
\def\mD{{\mathcal D}}

\def\mR{{\mathcal R}}
\def\mS{{\mathcal S}}
\def\mT{{\mathcal T}}

\def\msF{{\mathscr F}}

\def\msH{{\mathscr H}}

\def\bi{\begin{itemize}} \def\ei{\end{itemize}}
\def\be{\begin{eqnarray*}}
\def\ee{\end{eqnarray*}}

\def\0{{\mathbf 0}}

\newcommand{\beq}{\begin{equation}}
\newcommand{\eeq}{\end{equation}}

\def\wt{\widetilde}

\def\XXint#1#2#3{{\setbox0=\hbox{$#1{#2#3}{\int}$ }
\vcenter{\hbox{$#2#3$ }}\kern-.55\wd0}}

\newcommand{\TheTitle}{PET-MRI Joint Reconstruction by Joint Sparsity Based Tight Frame Regularization}
\newcommand{\TheAuthors}{Jae Kyu Choi, Chenglong Bao, and Xiaoqun Zhang}

\headers{PET-MRI Joint Reconstruction by JSTF}{\TheAuthors}

\title{{\TheTitle}\thanks{Submitted to the editors DATE.}}

\author{
  Jae Kyu Choi\footnotemark[2]
  \and
  Chenglong Bao\footnotemark[3]
  \and
  Xiaoqun Zhang\footnotemark[8]
}

\usepackage{amsopn}

\begin{document}
\maketitle
\renewcommand{\thefootnote}{\fnsymbol{footnote}}
\footnotetext[2]{Institute of Natural Sciences, Shanghai Jiao Tong University, Shanghai, 200240 China, ({\tt jaycjk@sjtu.edu.cn}). The research of this author is partially supported by the Young Top-notch Talent program of China, 973 program (No. 2015CB856004), Sino-German center grant (GZ1025), and General Financial Grant from the China Postdoctoral Science Foundation (No. 2017M611539).}
\footnotetext[3]{Corresponding Author. Yau Mathematical Sciences Center, Tsinghua University, Beijing, 100084 China, ({\tt clbao@math.tsinghua.edu.cn}).}
\footnotetext[8]{Corresponding Author. Institute of Natural Sciences, School of Mathematical Sciences, and MOE-LSC, Shanghai Jiao Tong University, Shanghai, 200240 China, ({\tt xqzhang@sjtu.edu.cn}). The research of this author is partially supported by the Young Top-notch Talent program of China, 973 program (No. 2015CB856004) and Sino-German center grant (GZ1025).}
\renewcommand{\thefootnote}{\arabic{footnote}}

\begin{abstract} Recent technical advances lead to the coupling of PET and MRI scanners, enabling to acquire functional and anatomical data simultaneously. In this paper, we propose a tight frame based PET-MRI joint reconstruction model via the joint sparsity of tight frame coefficients. In addition, a non-convex balanced approach is adopted to take the different regularities of PET and MRI images into account. To solve the nonconvex and nonsmooth model, a proximal alternating minimization algorithm is proposed, and the global convergence is present based on Kurdyka-\L{}ojasiewicz property. Finally, the numerical experiments show that the our proposed models achieve better performance over the existing PET-MRI joint reconstruction models.
\end{abstract}



\section{Introduction}\label{Introduction}

Medical imaging is the technique and process of visualizing the anatomy of a body for clinical analysis and medical intervention, as well as the function of some organs and tissues. For many decades, single medical imaging modality has been available and widely used to image either function or anatomy. For example, positron emission tomography (PET) and single positron emission computed tomography (SPECT) provide information on the distribution of radioisotopes and tracers in living tissue. These distributions allow clinicians to determine certain functions of tissue \cite{R.B.Innis2007,D.Kazantsev2014}. On the other hand, magnetic resonance imaging (MRI) and X-ray computed tomography (CT) can provide the information on anatomical structure.

Recent technical advances have allowed for the coupling of PET and MRI, leading to significant advantages over traditional PET-CT and standalone MRI \cite{C.Catana2013}. The resulting PET-MRI modalities are able to acquire functional data and anatomical data simultaneously, while PET-CT modalities acquire data sequentially \cite{C.Catana2013,Cherry2006,M.J.Ehrhardt2015,Townsend2008}. Since both PET and MRI data stem from the same underlying object, it would be more desirable to explore the relation between two modality images and develop a joint reconstruction model, rather than restore images independently. Typically, the major drawback of PET is low signal-to-noise ratio and low spatial resolution in the visualized image \cite{D.Kazantsev2014}. On the other hand, MRI provides an anatomical image with superior spatial resolution, whereas the so-called $k$-space data acquisition process is so time consuming that there has been increasing demand for methods reducing the amount of acquired data without degrading the image quality \cite{M.Lustig2007}. However, when the $k$-space data is undersampled, the Nyquist sampling criterion is violated, and this inevitably leads to the aliasing in the reconstructed image \cite{M.Lustig2007}. Nevertheless, through the joint reconstruction process, the complementary information in MRI and PET images can be borrowed to improve restoration quality, especially in the case of degraded and/or incomplete data. In this paper, we only consider the two dimensional PET-MRI joint reconstruction, as the same modeling can be easily applied to the three dimensional cases, and even to  other joint reconstruction problems such as SPECT-CT.

Let $u_1$ and $u_2$ denote the two dimensional PET and MRI images, and let $A$ and $\msF_p$ denote the acquisition process operator of PET and MRI respectively. The acquired PET data is assumed to be an instance of Poisson distributed vector with expectation $Au_1+c$ where $c$ denotes the mean number of background counts \cite{M.Burger2014,M.J.Ehrhardt2015}. Meanwhile, the noise in MRI can be modeled as a complex white Gaussian noise \cite{M.J.Ehrhardt2015}. Then, by the maximum likelihood principle, the PET-MRI joint reconstruction problem is formulated as the following minimization problem:
\begin{align}\label{GeneralPETMRIJR}
\min_{u_1\in\mC_1,u_2\in\mC_2}~\la 1,Au_1+c\ra-\la f,\ln(Au_1+c)\ra+\f{\kappa}{2}\|\msF_pu_2-g\|_2^2+\mR(u_1,u_2),
\end{align}
where $\mC_1$, $\mC_2$ are feasible sets, $\kappa$ is a positive constant related to the noise level of MRI data, and $\mR(u_1,u_2)$ is a nonseparable regularization term for $u_1$ and $u_2$.

In relation to the joint reconstruction, the above model \eqref{GeneralPETMRIJR} naturally arises the question of what is the proper choice of $\mR(u_1,u_2)$. As the function cannot be independent from structure \cite{M.J.Ehrhardt2015,A.Rangarajan2000}, it is reasonable to assume that the images to be reconstructed have \emph{structural correlation} \cite{M.J.Ehrhardt2015}. In the literature, the structure of an image is mathematically modeled as a singularity such as edge and ridge, and the idea of structural correlation based joint reconstruction was first explored in the color image processing as the color channels in general share the singularities \cite{X.Bresson2008,M.J.Ehrhardt2014,B.Goldluecke2012,R.Kimmel2000,G.Sapiro1996,N.Sochen1998,D.Tschumperle2005}. The successful examples include the vectorial total variation \cite{X.Bresson2008,B.Goldluecke2012} and the parallel level set (PLS) \cite{M.J.Ehrhardt2014}. Motivated by the works on the color image processing, the authors in \cite{M.J.Ehrhardt2015} recently applied joint total variation (JTV) \cite{M.J.Ehrhardt2015,E.Haber2013} and PLS priors to the PET-MRI joint reconstruction. Their analysis shows that the G\^ateaux derivatives of both JTV and PLS lead to the nonlinear, inhomogeneous diffusion, which promotes the edge-enhancing effect. However, even though both JTV and PLS methods showed significant improvements over the individual reconstruction methods, the regularization term $\mR(u_1,u_2)$ has to be further improved as the artifacts appear from the other modality image when the images have complex structures.

Wavelet frames are known to be effective in capturing  image singularities from degraded images. This is due to the multiscale structure of the wavelet frame systems, short supports of the framelet functions with varied vanishing moments, and the presence of both low pass and high pass filters in the wavelet frame filter banks, which are desirable in sparsely approximating images \cite{B.Dong2017}. In addition, due to the redundancy of systems, wavelet frames are more robust to errors than the (bi)orthogonal wavelet bases in \cite{Mallat2008}. Finally, the special structure of filter banks enables the construction of a so-called \emph{data driven tight frame} to provide a better sparse approximation of a given input image adaptively \cite{J.F.Cai2014,J.Wang2015,R.Zhan2016,W.Zhou2013}, with the computational efficiency over the K-SVD method in \cite{M.Aharon2006}. The successful applications in the medical image reconstruction can be found in the sparse angle X-ray CT reconstruction (wavelet frames in \cite{B.Dong2013a}, and data driven tight frames in \cite{R.Zhan2016}), the limited view CT reconstruction \cite{J.Choi2016}, the sparse MRI reconstruction \cite{Y.Liu2015}, and the data driven tight frame based CT reconstruction in \cite{W.Zhou2013}, etc.

In this paper, we propose a new tight frame (wavelet frame and data driven tight frame) based PET-MRI joint reconstruction model. Our proposed model is also based on the observation that two modality images have a strong correlation in the image singularities. As the singularities are sparsely captured by the (data driven) wavelet tight frame, our model enforces the \emph{joint sparsity} \cite{M.Fornasier2008} of frame coefficients to exploit such a structural correlation. Related examples can be found in the previous work in \cite{J.Wang2015}, where the authors use the \emph{vector $\ell_0$ norm} of frame coefficients to construct a multi channel data driven tight frame enforcing the structural correlation between image channels. Indeed, their applications to color image denoising and joint color-depth image reconstruction show the improvements over the independent reconstruction. Motivated by their success in color image processing, we further explore the application in the PET-MRI joint reconstruction. More precisely, using the vector $\ell_0$ norm of frame coefficients, we propose a balanced approach \cite{J.F.Cai2009,J.F.Cai2010,R.H.Chan2003} for the joint reconstruction from the data with Poisson noise (PET) and the undersampled $k$-space data (MRI). Finally, as it can be observed that the two modality images have different regularity, we expect the use of balanced approach will help suppress the artifacts which the existing models in \cite{M.J.Ehrhardt2015} may introduce.

The rest of this paper is organized as follows. In \cref{Preliminaries}, we review the basic concepts of wavelet frame and data driven tight frame. In \cref{ModelandAlgorithm}, we introduce our tight frame based PET-MRI joint reconstruction model, followed by the proximal alternating minimization algorithm, and the convergence analysis is given at the end of this section. In \cref{Results}, we present some experimental results, and the concluding remarks are given in \cref{Conclusion}.

\section{Preliminaries and Related Works}\label{Preliminaries}

\subsection{Wavelet Tight Frames}\label{WaveletFrame}

In this subsection, we briefly introduce the concept of tight frames and wavelet tight frames. Interested readers may consult \cite{J.F.Cai2008a,B.Dong2013,B.Dong2015,A.Ron1997,Shen2010} for details. Let $\msH$ be a Hilbert space equipped with the inner product $\la\cdot,\cdot\ra$ and the norm $\|\cdot\|$. A sequence $\{\varphi_n:n\in\Z\}\subseteq\msH$ is called a tight frame on $\msH$ if
\begin{align}\label{TightFrame}
\|u\|^2=\sum_{n\in\Z}|\la u,\varphi_n\ra|^2~~~~~\text{for all}~~~u\in\msH.
\end{align}
Given $\{\varphi_n:n\in\Z\}\subseteq\msH$, the analysis operator $W:\msH\to\ell_2(\Z)$ is defined as
\begin{align*}
u\in\msH\mapsto Wu=\{\la u,\varphi_n\ra:n\in\Z\}\in\ell_2(\Z),
\end{align*}
and the synthesis operator $W^T:\ell_2(\Z)\to\msH$ is defined as
\begin{align*}
v\in\ell_2(\Z)\mapsto W^Tv=\sum_{n\in\Z}v[n]\varphi_n\in\msH.
\end{align*}
Thus, $\{\varphi_n:n\in\Z\}$ is a tight frame on $\msH$ if and only if $W^TW=I$ with $I:\msH\to\msH$ being the identity. This means that for a given tight frame $\{\varphi_n:n\in\Z\}$, we have the following canonical expression:
\begin{align*}
u=\sum_{n\in\Z}\la u,\varphi_n\ra\varphi_n,
\end{align*}
where $Wu=\{\la u,\varphi_n\ra:n\in\Z\}$ is called the canonical tight frame coefficients. Hence, the tight frames are generalizations of orthonormal bases to the redundant systems. In fact, a tight frame is an orthonormal basis if and only if $\|\varphi_n\|=1$ for all $n\in\Z$.

One of the most widely used class of tight frames is discrete wavelet frame generated by a set of finitely supported filters $\{q_1,\cdots,q_m\}$. In this paper, we only discuss the undecimated wavelet frames, which is also known as the translation invariant wavelet frame transform. Given $q\in\ell_1(\Z)$, define a convolution operator $\mS_q:\ell_2(\Z)\to\ell_2(\Z)$ by
\begin{align*}
(\mS_qu)[n]=(q\ast u)[n]=\sum_{k\in\Z}q[n-k]u[k]~~~~~\text{for}~~~u\in\ell_2(\Z).
\end{align*}
Given a set of finitely supported filters $\{q_1,\cdots,q_m\}$, define the analysis operator $W$ and the synthesis operator $W^T$ respectively by
\begin{align}
W&=\left[\mS_{q_1[-\cdot]}^T,\mS_{q_2[-\cdot]}^T,\cdots,\mS_{q_m[-\cdot]}^T\right]^T\label{AnalConv}\\
W^T&=\left[\mS_{q_1},\mS_{q_2},\cdots,\mS_{q_m}\right].\label{SyntConv}
\end{align}
Then, the rows of $W$ form a tight frame on $\ell_2(\Z)$ if and only if $W^TW=I$, i.e. the filters $\{q_1,\cdots,q_m\}$ satisfy one of the so-called {\emph{unitary extension principle}} (UEP) condition \cite{B.Han2011}:
\begin{align}\label{SomeUEP}
\sum_{l=1}^m\sum_{k\in\Z}q_l[n+k]q_l[k]=\delta_n=\left\{\begin{array}{cl}
1~&\text{if}~n=0,\vspace{0.4em}\\
0~&\text{if}~n\neq0.
\end{array}\right.
\end{align}
Once the one dimensional filters generate a wavelet tight frame on $\ell_2(\Z)$, the higher dimensional wavelet tight frame could be obtained via the tensor product of the one dimensional filters.

\subsection{Data Driven Tight Frames}\label{DataDrivenTightFrame}

Even though tensor product wavelet frame based approach is simple to implement and able to achieve sparse representation of a piecewise smooth image, the major disadvantage is that these framelets mostly focus on singularities along the horizontal and vertical directions \cite{J.F.Cai2014,J.Liang2014}. For example, when an image has complex geometries such as directional textures, the tensor product wavelet frame coefficients may not be sparse enough \cite{J.F.Cai2014}.

In order to explore the adaptive sparse approximation of images, various data driven methods have been proposed (e.g. \cite{M.Aharon2006,J.F.Cai2014}). Among these methods, one promising transformation is the so-called data driven tight frame \cite{J.F.Cai2014}. It aims to construct a tight frame which sparsely approximates a given image $u$ adaptively. The tight frame $W=W(q_1,\cdots,q_{r^2})$ is generated by finitely supported real valued filters $\{q_1,\cdots,q_{r^2}\}$ satisfying \eqref{SomeUEP}. More concretely, given an image $u$, the data driven tight frame is obtained via solving the following minimization:
	\begin{align}\label{DDTFModel}
	\min_{v,W}~\|v-Wu\|_2^2+\lambda^2\|v\|_0~~~~~\text{subject to}~~~~W^TW=I,
	\end{align}
	where the $\ell_0$ norm $\|v\|_0$ encodes the number of nonzero entries in the coefficient vector $v$. To solve \eqref{DDTFModel}, we first reformulate it in the following way. Reshape all $r\times r$ patches of $u$ into $G\in\R^{r^2\times p}$ where $p$ denotes the number of total patches. Let $D\in\R^{r^2\times r^2}$ be the matrix generated by concatenating filters $\{q_1,\cdots,q_{r^2}\}$ into column vectors $\{\vec{q}_1,\cdots,\vec{q}_{r^2}\}$. Denote $V\in\R^{r^2\times p}$ as the frame coefficients. Hence, we have
	\begin{align*}
	u&\Leftrightarrow G=(\vec{g}_1,\cdots\,\vec{g}_p)\in\R^{r^2\times p}\\
	W&\Leftrightarrow D=\left(\vec{q}_1,\cdots,\vec{q}_{r^2}\right)\in\R^{r^2\times r^2}\\
	v&\Leftrightarrow V=(\vec{v}_1,\cdots,\vec{v}_p)\in\R^{r^2\times p}.
	\end{align*}
	Under this setting, \eqref{DDTFModel} is equivalent to
	\begin{align}\label{DDTFModelModi}
	\min_{V,D}~\|V-D^TG\|_F^2+\lambda^2\|V\|_0~~~~~\text{subject to}~~~~DD^T=I
	\end{align}
where $\|\cdot\|_F$ denotes the Frobenius norm of a matrix. For solving \eqref{DDTFModelModi}, the alternating minimization with closed form solutions is presented in \cite{J.F.Cai2014}, and the proximal alternating minimization (PAM) scheme with global convergence property is proposed in \cite{C.Bao2015}.

Recently, such a patch-based adaptive construction of tight frame filters has been extended to the following data driven tight frame for multi channel image in \cite{J.Wang2015}:
\begin{align}\label{DDTFMultiChannelModi}
\begin{split}
&\min_{v,\{W_i\}_{i=1}^c}~\sum_{i=1}^cw_i\|W_iu_i-v_i\|_2^2+\lambda\|v\|_{2,0}\\
&\text{subject to}~~W_i^TW_i=I,~~~i=1,\cdots,c.
\end{split}
\end{align}
Here, $W_i$, $u_i$, and $v_i$ for $i=1,2,\ldots,c$ respectively denote the tight frame filters, $i$th channel image and frame coefficients, and $\|v\|_{2,0}=\left|\{j:|v_1[j]|^2+\cdots+|v_c[j]|^2\neq0\}\right|$ with $v=(v_1,\cdots,v_c)$ is the vector $\ell_0$ norm to promote the joint sparsity. This vector $\ell_0$ norm $\|\cdot\|_{2,0}$ encodes the structural correlation among the channel images, which demonstrates the advantages in multi channel image restoration \cite{J.Wang2015}.

\section{Model and Algorithm}\label{ModelandAlgorithm}

\subsection{PET-MRI Joint Reconstruction Model}\label{Model}

Let $u_1$ and $u_2$ respectively denote the PET and MRI images to be reconstructed. In the literature, the observed PET data $f\in\R_+^{M}$ is corrupted by the Poisson noise while the MRI data (or $k$-space data) $g\in\C^{L}$ is corrupted by white complex Gaussian noise. Assuming that the noise is independent for each pixel, we have
\begin{align*}
P(f|u_1)=\prod_{j=1}^{M}\f{(Au_1+c)[j]^{f[j]}e^{-(Au_1+c)[j]}}{f[j]!}
\end{align*}
where the PET forward operator $A:\R^N\to\R^{M}$ is in general modelled as the discrete attenuated Radon transform \cite{E.M.Bardsley2010,M.Burger2014}, and $c$ denotes the mean number of background counts. Meanwhile, for MRI, we have
\begin{align*}
P(g|u_2)\propto\prod_{j=1}^L\exp\left\{-\f{|(\msF_pu_2)[j]-g[j]|^2}{2\s^2}\right\}
\end{align*}
where the MRI forward operator $\msF_p:\R^N\to\C^L$ is modelled as the unitary discrete Fourier transform followed by a projection onto the measured frequencies, and $\s$ is the standard deviation of noise in $k$-space data. Given that $f$ and $g$ are conditionally independent on $u_1$ and $u_2$, we have
\begin{align}\label{JointProb}
P(f,g|u_1,u_2)=P(f|u_1,u_2)P(g|u_1,u_2)=P(f|u_1)P(g|u_2).
\end{align}
Then, together with $P(u_1,u_2)\propto\exp(-\mR(u_1,u_2))$ and the Bayes' rule, the maximum a posteriori is equivalent to the minimization problem \eqref{GeneralPETMRIJR}.

Let $v_1$ and $v_2$ denote the wavelet frame coefficients of $u_1$ and $u_2$ under a given wavelet frame transform $W$ respectively. As the singularities of $u_1$ and $u_2$ are mostly correlated, we assume that the sparsity of $v_1$ and $v_2$ is correlated. Exploiting the idea of joint sparsity in \cite{J.Wang2015}, we propose our joint sparsity tight frame (JSTF) PET-MRI joint reconstruction model as follows:
\begin{align}\label{Proposed}
\begin{split}
&\min_{u_1\in\mC_1,u_2\in\mC_2,v}~\Phi_1(u_1)+\Phi_2(u_2)+\f{\mu_1}{2}\|Wu_1-v_1\|_2^2+\f{\mu_2}{2}\|Wu_2-v_2\|_2^2+\lambda\|v\|_{2,0}
\end{split}
\end{align}
where
\begin{align*}
\Phi_1(u_1)&=\la 1,Au_1+c\ra-\la f,\ln(Au_1+c)\ra\\
\Phi_2(u_2)&=\f{\kappa}{2}\|\msF_pu_2-g\|_2^2,
\end{align*}
and $\kappa=1/\s^2$. Here, $\mC_1=[0,a_1]^N$ and $\mC_2=[0,a_2]^N$ respectively denote feasible sets which reflect the physical properties of PET and MRI images.

As each modality image may contain information on the different features despite the structural correlation, we may have to actively learn the joint sparse approximation of $u_1$ and $u_2$ \cite{J.Wang2015}. Hence, we also propose the following joint sparsity based data driven tight frame (JSDDTF) joint reconstruction model
\begin{align}\label{DDTFPETMRI}
\begin{split}
&\min_{\substack{u_1\in\mC_1,u_2\in\mC_2,\\
W_1,W_2,v}}~\Phi_1(u_1)+\Phi_2(u_2)+\f{\mu_1}{2}\|W_1u_1-v_1\|_2^2+\f{\mu_2}{2}\|W_2u_2-v_2\|_2^2+\lambda\|v\|_{2,0}\\
&~\text{subject to}~~W_i^TW_i=I,~~~~~i=1,2,
\end{split}
\end{align}
where the tight frames $W_1$ and $W_2$ are learned from the two modality images $u_1$ and $u_2$ respectively.

In the literature, the joint sparsity of tight frame coefficients can also be achieved via the following joint analysis (JAnal) based approach
\begin{align}\label{JAnalPETMRI}
\min_{u_1\in\mC_1,u_2\in\mC_2}~\Phi_1(u_1)+\Phi_2(u_2)+\lambda\|Wu\|_{2,1}
\end{align}
with
\begin{align*}
\|Wu\|_{2,1}=\left\|\left(\|Wu_1\|_2^2+\|Wu_2\|_2^2\right)^{1/2}\right\|_1,~~~~~u=(u_1,u_2).
\end{align*}
Indeed, under properly chosen parameters, the above JAnal model \eqref{JAnalPETMRI} can be viewed as a discretization of a variational model including the following JTV model:
\begin{align}\label{JTVPETMRI}
\min_{u_1\in\mC_1,u_2\in\mC_2}~\Phi_1(u_1)+\Phi_2(u_2)+\lambda\left\|\left(\|\na u_1\|_2^2+\|\na u_2\|_2^2\right)^{1/2}\right\|_1,
\end{align}
while achieving the restoration results superior to the standard finite difference discretization of JTV (See e.g. \cite{J.F.Cai2012} for the related details). Meanwhile, our model takes the balanced approach (e.g. \cite{J.F.Cai2009,J.F.Cai2010,R.H.Chan2003}) as a special case; by fixing $u_1=W^Tv_1$ and $u_2=W^Tv_2$, the JSTF model \eqref{Proposed} reduces to
\begin{align*}
\begin{split}
\min_{v,W^Tv_1\in\mC_1,W^Tv_2\in\mC_2}~&\Phi_1(W^Tv_1)+\Phi_2(W^Tv_2)\\
&+\f{\mu_1}{2}\|(I-WW^T)v_1\|_2^2+\f{\mu_2}{2}\|(I-WW^T)v_2\|_2^2+\lambda\|v\|_{2,0},
\end{split}
\end{align*}
and the similar reasoning can be applied to the JSDDTF model \eqref{DDTFPETMRI} as well, by fixing $u_i=W_i^Tv_i$ for $i=1,2$.

When $\mu_1$, $\mu_2=\infty$, the JSTF model \eqref{Proposed} can be rewritten as the following nonconvex variant of JAnal model:
\begin{align}\label{NonconvexJAnal}
\min_{u_1\in\mC_1,u_2\in\mC_2}~\Phi_1(u_1)+\Phi_2(u_2)+\lambda\|Wu\|_{2,0}
\end{align}
as $v_i=Wu_i$ for $i=1,2$. Hence, the main difference between \eqref{Proposed} and the JAnal model \eqref{JAnalPETMRI} lies in the contribution of $\|Wu_i-v_i\|_2^2$, which measures the distance between $v_i$ (frame coefficients) and $Wu_i$ (canonical coefficients of $u_i$). Since the canonical coefficients in general may not be sparse enough due to the complex geometry of human body, we use the term $\|Wu_i-v_i\|_2^2$ to provide the flexibility in sparse approximation of two modality images. Most importantly, under some mild conditions on the tight frame $W$, the canonical coefficients not only provide the information on the image singularities, but their magnitude also reflects the regularity of an image (See \cite{L.Borup2004} for details). This means that, when the JAnal model \eqref{JAnalPETMRI} is used, we implicitly assume that the two modality images should have the same regularity, as well as the exact coincidence of singularities. However, even though the singularities of two modality images have a strong correlation, they may not exactly coincide with each other as different modality images reflect different physical properties of a human body. In addition, when two images have different regularity, the JAnal model can cause the singularities in one modality image to affect the smooth region in the other even if their singularities coincide. As a consequence, the JAnal model may introduce the artifacts in the reconstructed images. In contrast, $\mu_1$ and $\mu_2$ in our proposed model play a role of balancing the structural correlation of two modality images and their different regularity. As we can observe that MRI images are less noisy compared to PET image, in our proposed model, $\mu_2$ is chosen to be larger than $\mu_1$, and we expect that this different choice of $\mu_i$ will help to suppress such artifacts. Finally, the same arguments are applied to the JSDDTF model \eqref{DDTFPETMRI} as it has the same modeling philosophy as the JSTF model \eqref{Proposed} except that we learn two tight frames adaptively from $u_1$ and $u_2$.

\subsection{Alternating Minimization Algorithm}\label{Algorithm}
We propose the proximal alternating minimization (PAM) algorithms \cite{H.Attouch2010} to solve both \eqref{Proposed} and \eqref{DDTFPETMRI}. As the two schemes are similar  except that  $W_1$ and $W_2$ are additionally  updated in \eqref{DDTFPETMRI}, we only consider the algorithm for solving \eqref{DDTFPETMRI}. Given initializations $u_1^0$ and $u_2^0$, the initializations of $v$, $W_1$, and $W_2$ are obtained via
\begin{align}\label{FraCoeffIni}
\begin{split}
&\min_{v,W_1,W_2}~\mu_1\|W_1u_1^0-v_1\|_2^2+\mu_2\|W_2u_2^0-v_2\|_2^2+\wt{\lambda}\|v\|_{2,0}\\
&\text{subject to}~~W_i^TW_i=I,~~~i=1,2,
\end{split}
\end{align}
using the algorithm in \cite{J.F.Cai2014}. After the initializations, we optimize $\{u_1,u_2,v,W_1,W_2\}$ by solving the model \eqref{DDTFPETMRI} alternatively. The full details are described in \cref{Alg1}. Note that we consider the subproblems separately whenever they are separable.

\begin{algorithm}[htp!]
\caption{Proximal Alternating Minimization Algorithm for \eqref{DDTFPETMRI}}\label{Alg1}
\begin{algorithmic}
\STATE{\textbf{Initialization:} $u_1^0$, $u_2^0$, $v^0$, $W_1^0$, $W_2^0$}
\FOR{$k=0$, $1$, $2$, $\cdots$}
\STATE{\textbf{(1)} Optimize $u_1$ and $u_2$:
\begin{align}\label{u1u2update}
\begin{split}
u_1^{k+1}&=\argmin_{u_1\in\mC_1}~\Phi_1(u_1)+\f{\mu_1}{2}\|W_1^ku_1-v_1^k\|_2^2+\f{\alpha_1^k}{2}\|u_1-u_1^k\|_2^2\\
u_2^{k+1}&=\argmin_{u_2\in\mC_2}~\Phi_2(u_2)+\f{\mu_2}{2}\|W_2^ku_2-v_2^k\|_2^2+\f{\alpha_2^k}{2}\|u_2-u_2^k\|_2^2.
\end{split}
\end{align}
\textbf{(2)} Optimize $W_1$ and $W_2$:
\begin{align}\label{Wupdate}
\begin{split}
W_1^{k+1}&=\argmin_{W_1^TW_1=I}~\f{\mu_1}{2}\|W_1u_1^{k+1}-v_1^k\|_2^2+\f{\beta_1^k}{2}\|W_1-W_1^k\|_2^2\\
W_2^{k+1}&=\argmin_{W_2^TW_2=I}~\f{\mu_2}{2}\|W_2u_2^{k+1}-v_2^k\|_2^2+\f{\beta_2^k}{2}\|W_2-W_2^k\|_2^2.
\end{split}
\end{align}
\textbf{(3)} Optimize $v=(v_1,v_2)$:
\begin{align}\label{vupdate}
\begin{split}
v^{k+1}=\argmin_{v}~\lambda\|v\|_{2,0}&+\f{\mu_1}{2}\|v_1-W_1^{k+1}u_1^{k+1}\|_2^2\\
&+\f{\mu_2}{2}\|v_2-W_2^{k+1}u_2^{k+1}\|_2^2+\f{\gamma^k}{2}\|v-v^k\|_2^2.
\end{split}
\end{align}}
\ENDFOR
\end{algorithmic}
\end{algorithm}

It is easy to see that each problem in \eqref{u1u2update} is strongly convex and smooth, and there are numerous algorithms to solve it. In our algorithm, $u_1^{k+1}$ is updated using the projected scaled gradient method \cite{W.Jin2016} which is a special case of preconditioned alternating projection algorithm \cite{A.Krol2012}; Let $u_1^{k,0}=u_1^k$. For $j=0$, $1$, $2$, $\cdots$
\begin{align*}
M^j&=\diag(u^{j}/A^T1)\\
u_1^{j+1/2}&=u_1^{j}-\rho_1^jM^j\left[A^T\left(1-\f{f}{Au_1^{j}+c}\right)+\mu_1\left(u_1^{j}-(W_1^k)^Tv_1^k\right)+\alpha_1^k\left(u_1^{j}-u_1^k\right)\right]\\
u_1^{j+1}&=\min\left\{\max(u_1^{j+1/2},0),a_1\right\}.
\end{align*}
Similarly, $u_2^{k+1}$ is updated using the projected gradient method \cite{D.G.Luenberger2016}; Let $u_2^{k,0}=u_2^k$. For $j=0$, $1$, $2$, $\cdots$
\begin{align*}
u_2^{j+1/2}&=u_2^j-\rho_2^j\left[\kappa\msF_p^*(\msF_pu_2^j-g)+\mu_2\left(u_2^j-(W_2^k)^Tv_2^k\right)+\alpha_2^k\left(u_2^j-u_2^k\right)\right]\\
u_2^{j+1}&=\min\left\{\max(u_2^{j+1/2},0),a_2\right\},
\end{align*}
where $\msF_p^*$ is the conjugate transpose of $\msF_p$. In any case, we omit the outer iteration $k$ on $M$, $u_1$, and $u_2$. Since the constraints $\mC_1$ and $\mC_2$ are convex, the above iterations converge to the global minimizers for appropriately chosen $\rho_1^j$ and $\rho_2^j$ \cite{W.Jin2016,D.G.Luenberger2016}.

To solve \eqref{Wupdate} and \eqref{vupdate}, we introduce
\begin{align*}
\{u_1,W_1,v_1\}&\Leftrightarrow\{G_1,D_1,V_1\}\\
\{u_2,W_2,v_2\}&\Leftrightarrow\{G_2,D_2,V_2\}
\end{align*}
using $r\times r$ patches of $u_1$ and $u_2$. Under the above reformulation, \eqref{Wupdate} and \eqref{vupdate} become
\begin{align}
\begin{split}\label{Aupdate}
D_1^{k+1}&=\argmin_{D_1D_1^T=I}~\f{\mu_1}{2}\|D_1^TG_1^{k+1}-V_1^k\|_F^2+\f{\beta_1^k}{2}\|D_1-D_1^k\|_F^2\\
D_2^{k+1}&=\argmin_{D_2D_2^T=I}~\f{\mu_2}{2}\|D_2^TG_2^{k+1}-V_2^k\|_F^2+\f{\beta_2^k}{2}\|D_2-D_2^k\|_F^2,\\
\end{split}
\end{align}
and
\begin{align}
\begin{split}\label{Vupdate}
V^{k+1}&=\argmin_{V}~\lambda\|V\|_{2,0}+\f{\mu_1}{2}\|V_1-(D_1^{k+1})^TG_1^{k+1}\|_F^2\\
&\hspace{8.35em}+\f{\mu_2}{2}\|V_2-(D_2^{k+1})^TG_2^{k+1}\|_F^2+\f{\gamma^k}{2}\|V-V^k\|_F^2
\end{split}
\end{align}
where $V=(V_1,V_2)$. Hence, to solve \eqref{Aupdate}, we use the following closed form formulae:
\begin{align}\label{WupdateClosed}
\begin{split}
D_1^{k+1}&=X_1Y_1^T~~~\text{where}~~X_1,\Sigma_1,Y_1~~\text{is the SVD of}~~G_1^{k+1}(V_1^k)^T+\f{\beta_1^k}{\mu_1}D_1^k\\
D_2^{k+1}&=X_2Y_2^T~~~\text{where}~~X_2,\Sigma_2,Y_2~~\text{is the SVD of}~~G_2^{k+1}(V_2^k)^T+\f{\beta_2^k}{\mu_2}D_2^k.
\end{split}
\end{align}
The closed form solution to \eqref{Vupdate} is given by
\begin{align}\label{VupdateClosed}
V^{k+1}=\mT_{2\lambda,\mu+\gamma^k}\left[\left(\f{\mu_1(D_1^{k+1})^TG_1^{k+1}+\gamma^kV_1^k}{\mu_1+\gamma^k},\f{\mu_2(D_2^{k+1})^TG_2^{k+1}+\gamma^kV_2^k}{\mu_2+\gamma^k}\right)\right]
\end{align}
where $\mT_{\lambda,\mu}$ with $\mu=(\mu_1,\mu_2)$ is the generalized hard thresholding formula \cite{J.Wang2015} defined as
\begin{align}\label{GeneralHardThresh}
(\mT_{\lambda,w}[(U_1,U_2)])_j=\left\{\begin{array}{cl}
(U_{1,j},U_{2,j})~&\text{if}~~\sum_{i=1}^2w_i|U_{i,j}|^2\geq\lambda,\vspace{0.4em}\\
0~&\text{otherwise}.
\end{array}\right.
\end{align}
Here, $(\mT_{\lambda,w}[(U_1,U_2)])_j$ and $U_{i,j}$ denote the $j$th row vector of $\mT_{\lambda,w}[(U_1,U_2)]$ and $U_i$ respectively.

\subsection{Convergence Analysis}\label{Convergence}

This section is devoted to the convergence analysis of our alternating minimization algorithm. Due to the same reason to the PAM algorithm, we focus on the convergence analysis of sequence $\{(u_1^k,u_2^k,W_1^k,W_2^k,v^k):k\in\N\}$ in \eqref{DDTFPETMRI} generated by \cref{Alg1}. To do this, we assume the followings:
\begin{enumerate}
\item[A1.] $f>0$ and $A$ satisfies
\begin{align*}
A[i,j]\geq0,~~~~A1>0,~~~\&~~~A^T1>0.
\end{align*}
where $A[i,j]$ denotes the $(i,j)$th entry of $A$.
\item[A2.] $c>0$ is a constant vector.
\item[A3.] There exist $0<L<U$ such that $L\leq\alpha_1^k,\alpha_2^k,\beta_1^k,\beta_2^k,\gamma^k\leq U$ for all $k\in\N$.
\end{enumerate}

We first introduce some basic notation and definitions.

\begin{definition} Let $f:\R^n\to\R\cup\{\infty\}$ be a proper and lower semicontinuous (lsc) function.
\begin{enumerate}
\item The domain of $f$, denoted as $\dom(f)$ is defined as
\begin{align*}
\dom(f)=\{x\in\R^n:f(x)<\infty\}.
\end{align*}
\item For each $x\in\dom(f)$, the Fr\'echet subdifferential of $f$ at $x$ is defined as
\begin{align*}
\p_Ff(x)=\left\{s\in\R^n:\liminf_{y\to x}\f{f(y)-f(x)-\la s,y-x\ra}{\|y-x\|}\geq0\right\}.
\end{align*}
If $x\notin\dom(f)$, then we set $\p_Ff(x)=\emptyset$.
\item The (limiting-) subdifferential of $f$ at $x$ is defined as
\begin{align*}
\p f(x)=\left\{s\in\R^n:\exists x_n~\text{s.t.}~f(x_n)\to f(x)~\&~s_n\in\p_Ff(x_n)\to s\right\}.
\end{align*}
The domain of $\p f$ is defined as $\dom(\p f)=\{x\in\R^n:\p f(x)\neq\emptyset\big\}$.
\item $x\in\dom(f)$ is a critical point of $f$ if $0\in\p f(x)$.
\end{enumerate}
\end{definition}

In the literature, the global convergence of the alternating schemes are established on the framework of the Kurdyka-Lojasiewicz (KL) property in \cite{Kurdyk1998,Lojasiewicz1995}. Even though it is challenging to verify whether a given function $f$ satisfies the KL property, there are some functions of special cases which have proven to satisfy the KL property, such as analytic functions and semi-algebraic functions (e.g. \cite{H.Attouch2010,C.Bao2016,J.Bolte2014,Y.Xu2013} ).

To apply the framework of KL property to our model, we let
\begin{align*}
x=(x_1,x_2,x_3,x_4,x_5)=(u_1,u_2,W_1,W_2,v)
\end{align*}
for notational simplicity. Recall $\mC_1=[0,a_1]^N$ and $\mC_2=[0,a_2]^N$, and let $\mD=\{W\in\R^{r^2N\times N}:W^TW=I_{r^2\times r^2}\}$. Then we define
\begin{align*}
P(x)&=\Phi_1(u_1)+\Phi_2(u_2)+\f{\mu_1}{2}\|W_1u_1-v_1\|_2^2+\f{\mu_2}{2}\|W_2u_2-v_2\|_2^2\\
r_i(x_i)&=\imath_{\mC_i}(x_i)~~~~~i=1,2\\
r_i(x_i)&=\imath_{\mD}(x_i)~~~~~i=3,4\\
r_5(x_5)&=\lambda\|x_5\|_{2,0}
\end{align*}
where $\imath_{\mA}$ is the indicator function of a set $\mA$: $\imath_{\mA}(x)=0$ if $x\in\mA$, $\imath_{\mA}(x)=\infty$ otherwise.

Using the above notation, we reformulate the model \eqref{DDTFPETMRI} as
\begin{align}\label{ObjectReform}
\min_x~H(x):=P(x)+\sum_{i=1}^5r_i(x_i).
\end{align}
Note that it is not hard to verify that $\Phi_1(u_1)$ and the indicator functions are analytic, and $r_5(x_5)$ is a semi-algebraic function. Together with that the remaining terms are all polynomials, we can see that the functional $H$ defined as in \eqref{ObjectReform} satisfies the KL property.

Under this reformulation, we can present the following result on the global convergence.

\begin{theorem}\label{Th1} Assume that A1-A3 hold. Let $H(x)$ be the objective function defined as in \eqref{ObjectReform}. Then the sequence $\{x^k:k\in\N\}$ generated by \cref{Alg1} converges to a critical point of $H$. Moreover, $\{x^k\}_{k\in\N}$ has the following finite length property:
\begin{align}\label{FiniteLength}
\sum_{k=0}^{\infty}\|x^{k+1}-x^k\|_2<\infty.
\end{align}
\end{theorem}

The proof of \cref{Th1} is similar to the framework given in \cite[Theorem 3.7.]{C.Bao2016}. In fact, $H(x)$ defined as in \eqref{ObjectReform} satisfies KL property, and noting that $\Phi_1(x_1)$ is analytic, proper, lsc, strictly convex and coercive \cite{M.A.T.Figueiredo2010}, it is not hard to see that $H(x)$ satisfies the first condition in \cite[Theorem 3.7.]{C.Bao2016}. (See \cite{C.Bao2016} for the details). Hence, the proof is completed provided that we show that the sequence $\{x^k:k\in\N\}$ generated by \cref{Alg1} is bounded. In fact, since the constraints $\mC_1$, $\mC_2$ and $\mD$ are compact sets, it suffices to show that $v^k$ is bounded.

\begin{lemma}\label{Lemma1} Assume that A1-A3 hold. Let $H(x)$ be the objective function defined as in \eqref{ObjectReform}. For  $\{x^k:k\in\N\}$ generated by \cref{Alg1}, there exist $R_1$, $R_2>0$ such that $\|v_1^k\|_2\leq R_1$ and $\|v_2^k\|_2\leq R_2$ for all $k\geq0$.
\end{lemma}

\begin{proof} Note that the constraint on $u_1$ and $u_2$ leads to
\begin{align*}
\|u_1^k\|_2\leq a_1\sqrt{N}~~~~\text{and}~~~~\|u_2^k\|_2\leq a_2\sqrt{N}
\end{align*}
for all $k\geq0$. If we choose $R_1=a_1\sqrt{N}$ and $R_2=a_2\sqrt{N}$, then the proof is completed by mathematical induction. For $k=0$, \eqref{GeneralHardThresh} implies
\begin{align*}
\|v_1^0\|_2&\leq\|W_1^0u_1^0\|_2=\|u_1^0\|_2\leq R_1\\
\|v_2^0\|_2&\leq\|W_2^0u_2^0\|_2=\|u_2^0\|_2\leq R_2
\end{align*}
where the equality comes from the fact that $W_1^0$ and $W_2^0$ are tight frames.

For the mathematical induction, we assume that $\|v_1^k\|_2\leq R$ and $\|v_2^k\|_2\leq R$ for $k\geq 0$. Again, by \eqref{GeneralHardThresh}, we have
\begin{align*}
\|v_1^{k+1}\|_2\leq\left\|\f{\mu_1 W_1^{k+1}u_1^{k+1}+\gamma^kv_1^k}{\mu_1+\gamma^k}\right\|_2&\leq\f{\mu_1}{\mu_1+\gamma^k}\|W_1^{k+1}u_1^{k+1}\|_2+\f{\gamma^k}{\mu_1+\gamma^k}\|v_1^k\|_2\leq R_1\\
\|v_2^{k+1}\|_2\leq\left\|\f{\mu_2 W_2^{k+1}u_2^{k+1}+\gamma^kv_2^k}{\mu_2+\gamma^k}\right\|_2&\leq\f{\mu_2}{\mu_2+\gamma^k}\|W_2^{k+1}u_2^{k+1}\|_2+\f{\gamma^k}{\mu_2+\gamma^k}\|v_2^k\|_2\leq R_2
\end{align*}
from the fact that $W_1^k$, $W_2^k\in\mD$ for all $k\in\N$. This completes the proof.
\end{proof}

\section{Numerical Results}\label{Results}

In this section, we present some experimental results to compare our JSTF model \eqref{Proposed} and JSDDFT model \eqref{DDTFPETMRI} with the several existing methods. We choose to compare with the the following analysis based individual reconstruction models:
 \begin{align}
&\min_{u_1\in\mC_1}~\Phi_1(u_1)+\lambda_1\|Wu_1\|_1, \label{AnaPET}\\
&\min_{u_2\in\mC_2}~\Phi_2(u_2)+\lambda_2\|Wu_2\|_1, \label{AnaMRI}
\end{align}
both of which are solved by the split Bregman method \cite{J.F.Cai2009/10}, the data driven tight frame (DDTF) individual reconstruction models:
\begin{align}
&\min_{u_1\in\mC_1,v_1,W_1}~\Phi_1(u_1)+\f{\mu_1}{2}\|W_1u_1-v_1\|_2^2+\lambda_1\|v_1\|_0~~\text{subject to}~~W_1^TW_1=I, \label{DDTFPET}\\
&\min_{u_2\in\mC_2,v_2,W_2}~\Phi_2(u_2)+\f{\mu_2}{2}\|W_2u_2-v_2\|_2^2+\lambda_2\|v_2\|_0~~\text{subject to}~~W_2^TW_2=I, \label{DDTFMRI}
\end{align}
solved by the PAM algorithm, and the quadratic parallel level set (QPLS) method in \cite{M.J.Ehrhardt2015}. We also compare with the JAnal model \eqref{JAnalPETMRI} solved by the split Bregman method, as a replacement of the JTV model \eqref{JTVPETMRI}.

The experiments are conducted with $256\times 256$ PET and MRI images taking values in $[0,1]$, which are available at Harvard Whole Brain Atlas webpage\footnote{http://www.med.harvard.edu/AANLIB/home.html}. Throughout this paper, we set both $\mC_1$ and $\mC_2$ are $[0,1]^N$. Note that MRI images are classified into the positron density (PD) weighted image, the T1 weighted image, and the T2 weighted image according to the pulse sequence design \cite{E.M.Haacke1999} (See \cref{OriginalImages}). In this sequel, we perform experiments using different MRI images for different $u_2$, which we shall refer to PET-PD, PET-T1, and PET-T2 respectively. To measure the quality of restored images, we compute the relative error, the PSNR, and the correlation between the reconstructed image $\wt{u}_i$ and the true image $u_i$ respectively defined as
\begin{align*}
\mathrm{RelErr}(u_i,\wt{u}_i)&=\f{\|u_i-\wt{u}_i\|_2}{\|u_i\|_2},\\
\mathrm{PSNR}(u_i,\wt{u}_i)&=-10\log_{10}\f{\|u_i-\wt{u}_i\|_2^2}{N},\\
\mathrm{Corr}(u_i,\wt{u}_i)&=\f{\la u_i-\overline{u}_i,\wt{u}_i-\overline{\wt{u}}_i\ra}{\|u_i-\overline{u}_i\|_2\|\wt{u}_i-\overline{\wt{u}}_i\|_2}
\end{align*}
where $\overline{u}_i$ and $\overline{\wt{u}}_i$ denote the mean of $u_i$ and $\wt{u}_i$ respectively.

\begin{figure}[ht]
	\centering
	\begin{tabular}{cccc}
		\begin{minipage}{3cm}
			\includegraphics[width=3cm]{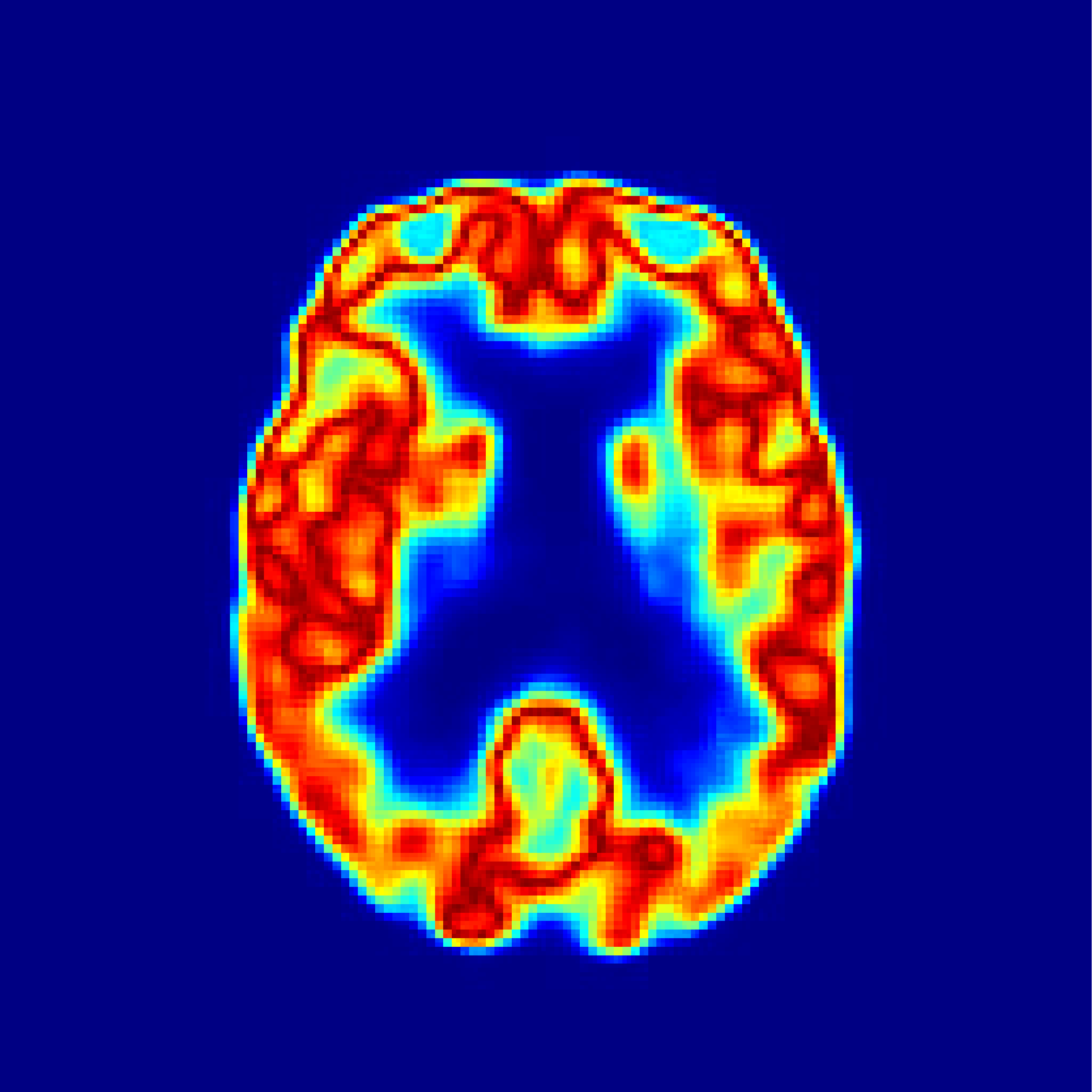}
		\end{minipage}&\hspace{-0.45cm}
		\begin{minipage}{3cm}
			\includegraphics[width=3cm]{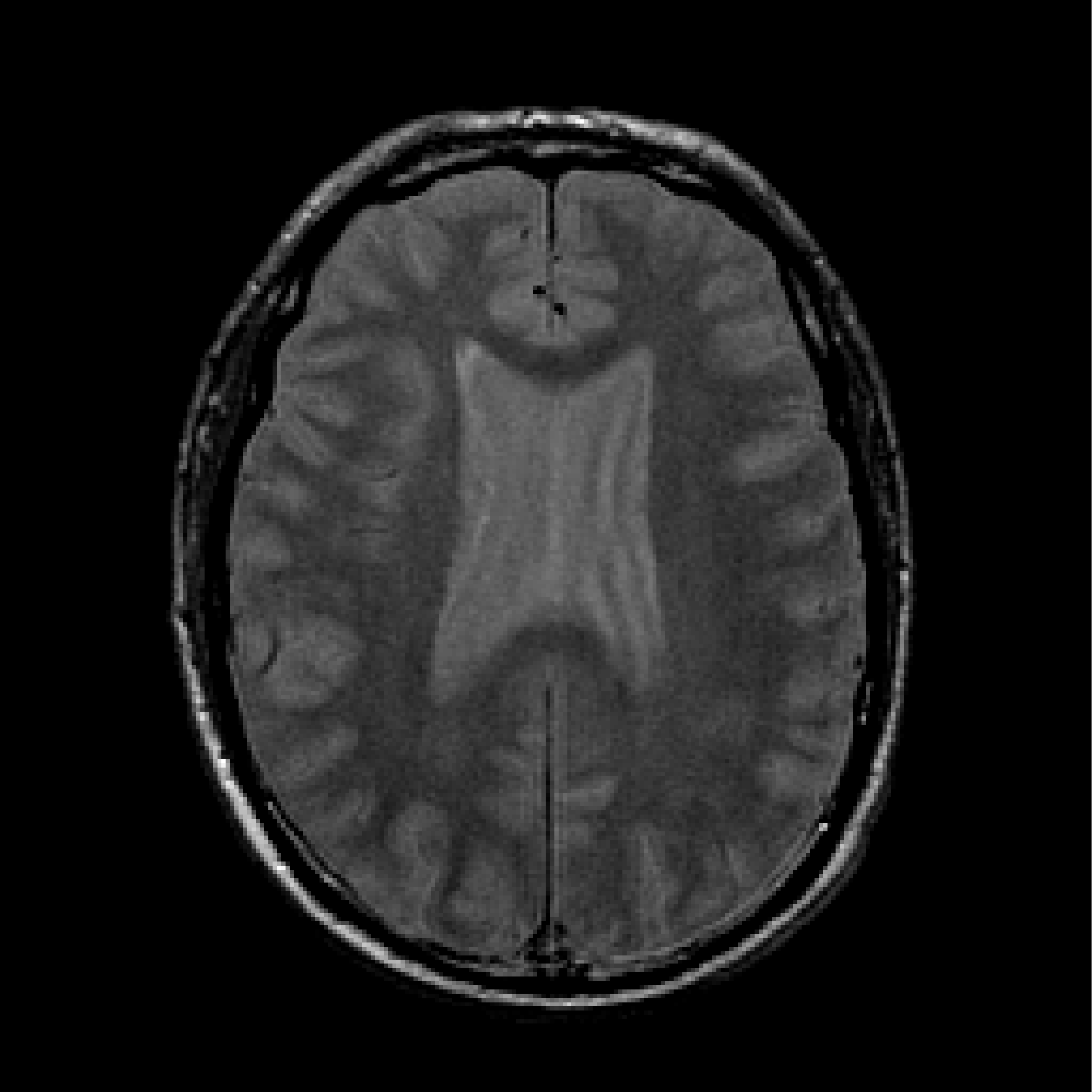}
		\end{minipage}&\hspace{-0.45cm}
		\begin{minipage}{3cm}
			\includegraphics[width=3cm]{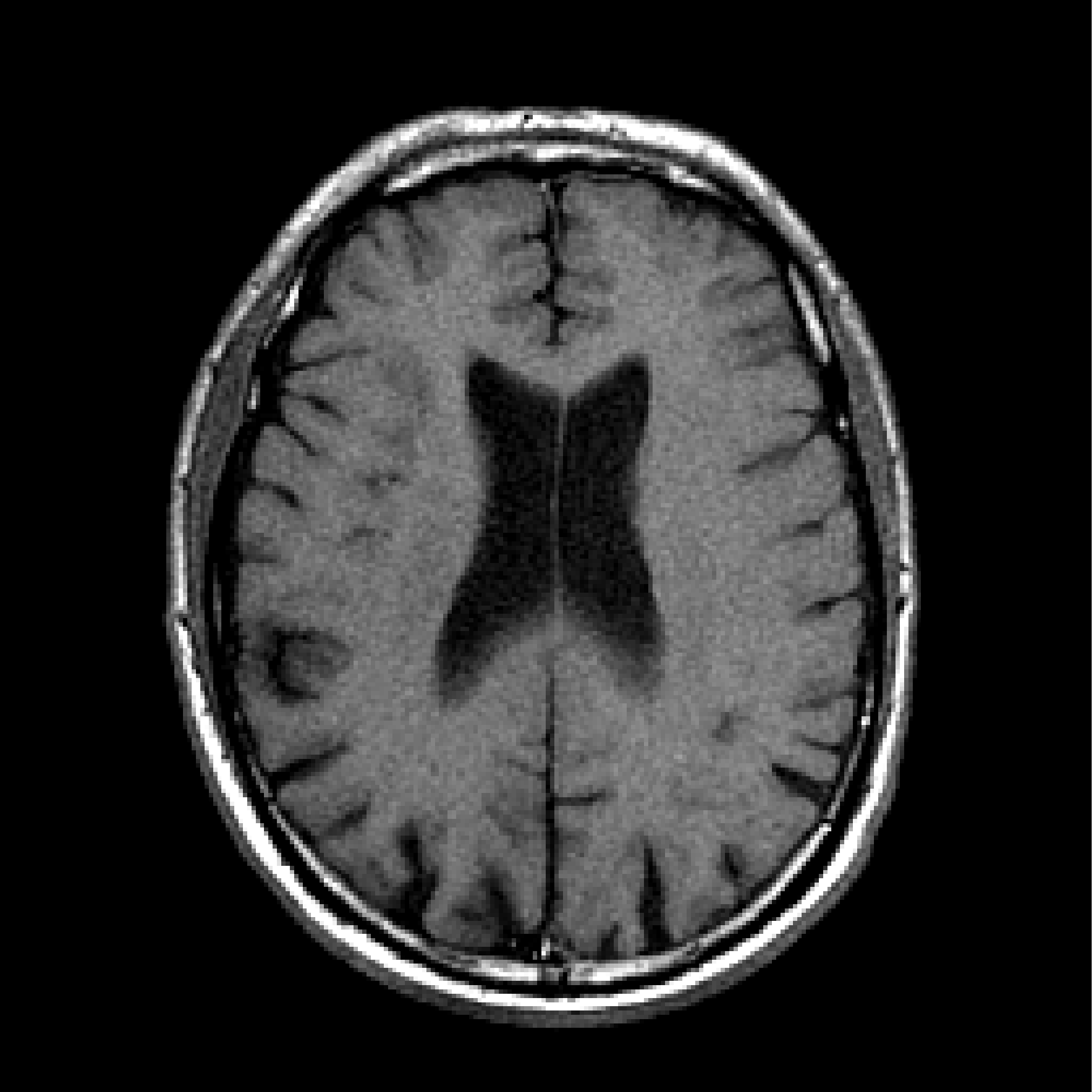}
		\end{minipage}&\hspace{-0.45cm}
		\begin{minipage}{3cm}
			\includegraphics[width=3cm]{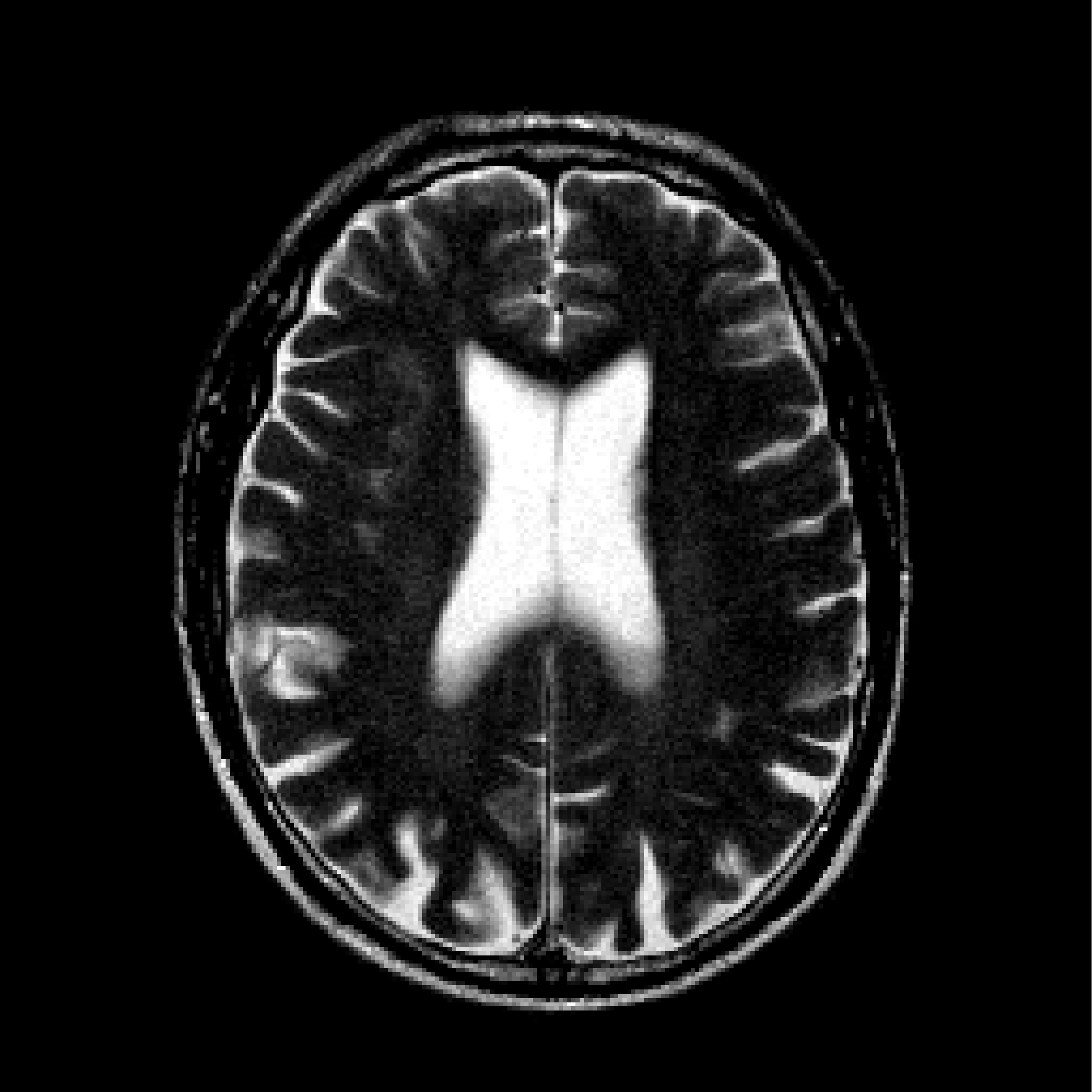}
		\end{minipage}\\
		{\small{PET Image}}&\hspace{-0.45cm}
		{\small{PD Image}}&\hspace{-0.45cm}
		{\small{T1 Image}}&\hspace{-0.45cm}
		{\small{T2 Image}}
	\end{tabular}
	\caption{True PET image (first column) and true MRI images (second to last columns). We use different MRI images for $u_2$.}\label{OriginalImages}
\end{figure}

To synthesize the data, we generate the forward PET operator $A$ as described in \cite{E.M.Bardsley2010}, and we set $c=1$ in all cases. To generate $256\times256$ PET data $f$ with Poisson noise, we use the MATLAB built-in function ``imnoise($\cdot$,`poisson')''by scaling $Au_1+c$ with a suitable factor before applying ``imnoise'', and then scaling it back with the same factor. More precisely,
\begin{align*}
f=\text{factor}*\text{imnoise}((Au_1+c)/\text{factor},\text{`poisson'})
\end{align*}
and we set factor to be $2*10^8$ in all cases. Note that the larger ``factor'' is, the larger noise level is. To generate $g$, we first generate $\msF_p=R_{\Lambda}\msF$, where $R_{\Lambda}$ is the projection on the known frequency region $\Lambda$, and $\msF$ is the unitary discrete Fourier transform. Here, we choose $R_\Lambda$ to be the sampling along the $30$ radial lines and the $10\%$ random sampling described in \cite{B.Deka2015}. The gaussian noise with standard deviation $0.05$ is also added to both real and imaginary part of $\msF_pu_2$.

In all experiments, we choose the same initializations for the fair comparison; $u_1^0$ is obtained by the Expectation-Maximization algorithm \cite{L.A.Shepp1982}, and $u_2^0$ is obtained by the inverse discrete Fourier transform with zero filling (See \cref{Initializations1}). For the individual reconstruction models \eqref{AnaPET} and \eqref{AnaMRI}, the JAnal Model \eqref{JAnalPETMRI} as well as our JSTF model \eqref{Proposed}, we use one level piecewise cubic B-spline wavelet frame transformation. For JSDDTF model \eqref{DDTFPETMRI} as well as DDTF individual reconstruction models \eqref{DDTFPET} and \eqref{DDTFMRI}, we use $8\times 8$ undecimated discrete cosine transform filters \cite{Strang1999} for the initial guess of \eqref{FraCoeffIni}. In all cases, we set $\kappa=1$. In \eqref{Proposed}, \eqref{DDTFPETMRI}, \eqref{DDTFPET}, and \eqref{DDTFMRI}, we set $\mu_1=0.05$, $\mu_2=1$, $\alpha_1^k=\alpha_2^k=0.001$, $\beta_1^k=\beta_2^k=\gamma^k=0.00005$, and $\rho_1^j=\rho_2^j=0.5$ in all cases. All parameters related to the (joint) sparsity are manually chosen so that we obtain the optimal restoration results, especially the optimal results in PET image for the joint reconstruction models. In addition, as the wavelet frame systems consist of low pass filter and high pass filters and the low pass filter coefficients are not sparse in general \cite{J.F.Cai2012}, we follow the convention that we do not penalize the frame coefficients corresponding to the low pass filter.

\begin{figure}[htp!]
\centering
\begin{tabular}{cccc}
&\hspace{-0.45cm}
\begin{minipage}{3cm}
\includegraphics[width=3cm]{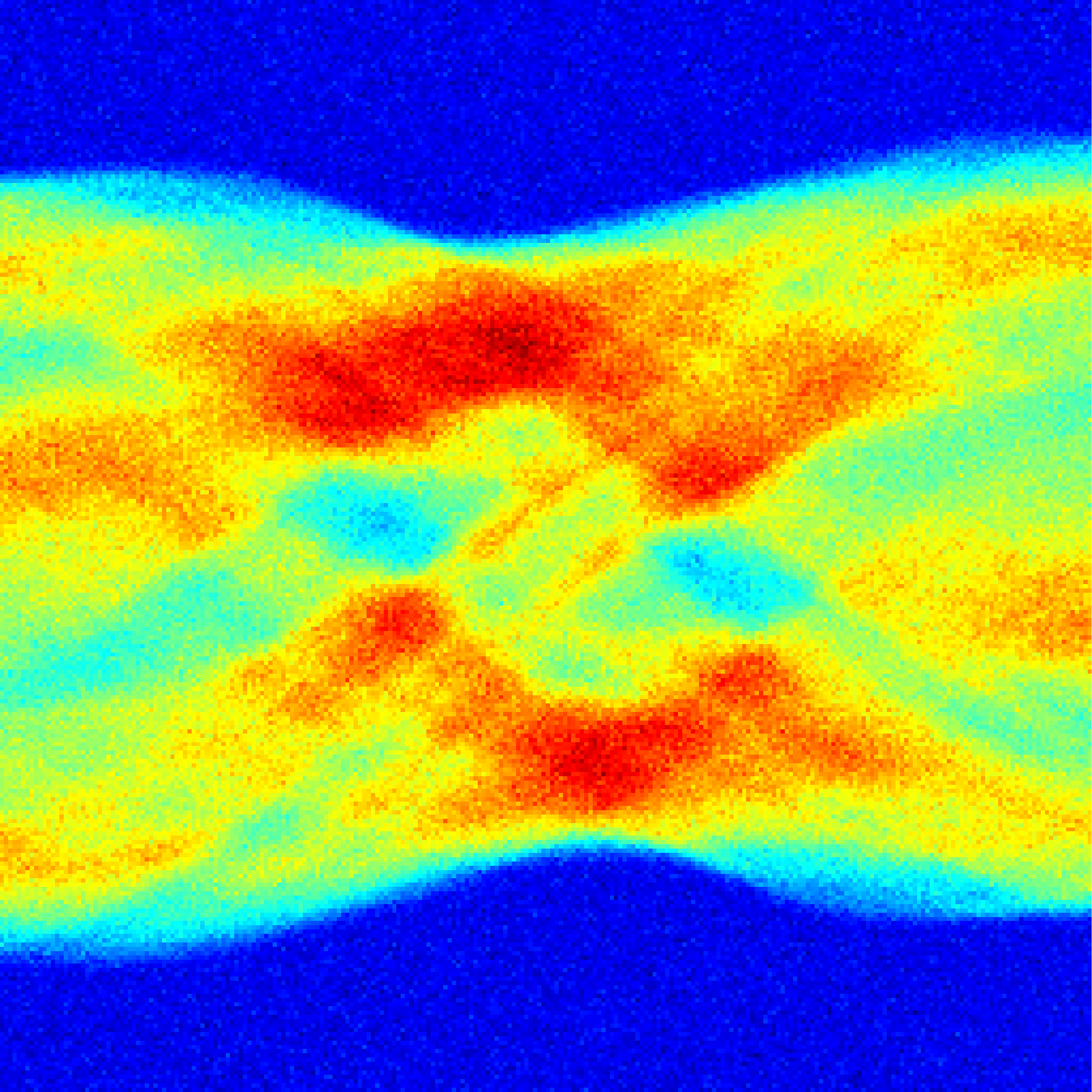}
\end{minipage}&\hspace{-0.45cm}
\begin{minipage}{3cm}
\includegraphics[width=3cm]{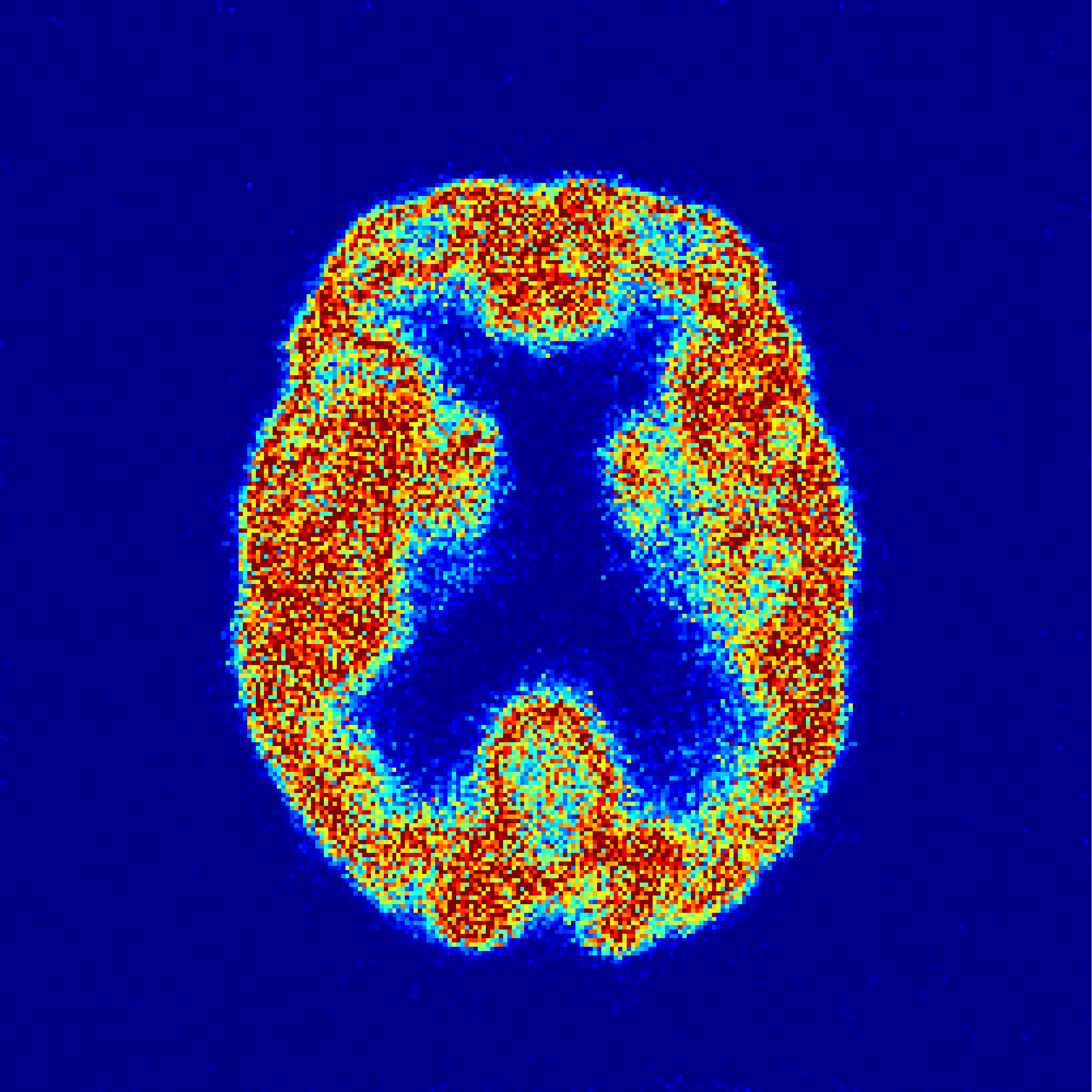}
\end{minipage}&\hspace{-0.45cm}
\vspace{0.25em}\\
\begin{minipage}{3cm}
\includegraphics[width=3cm]{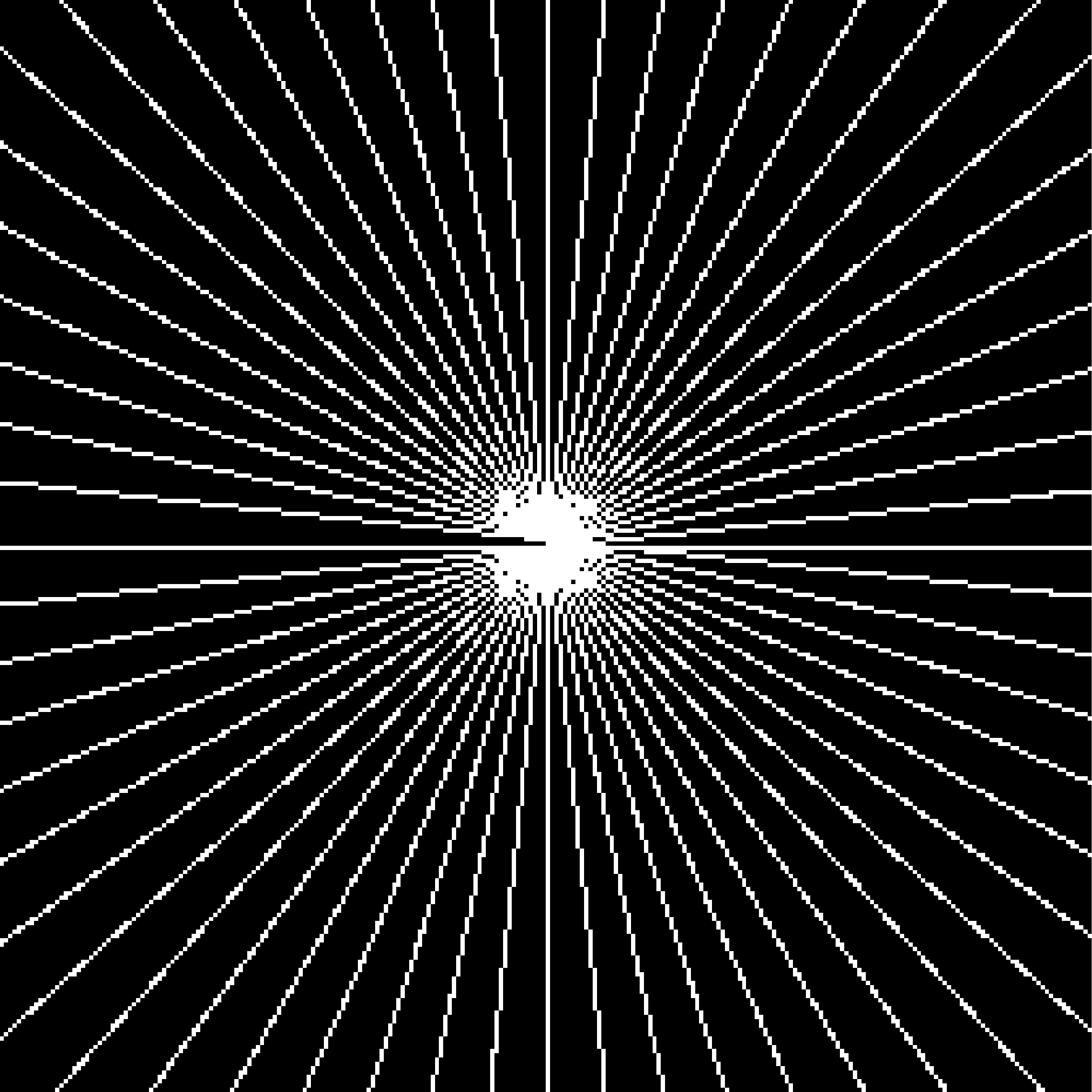}
\end{minipage}&\hspace{-0.45cm}
\begin{minipage}{3cm}
\includegraphics[width=3cm]{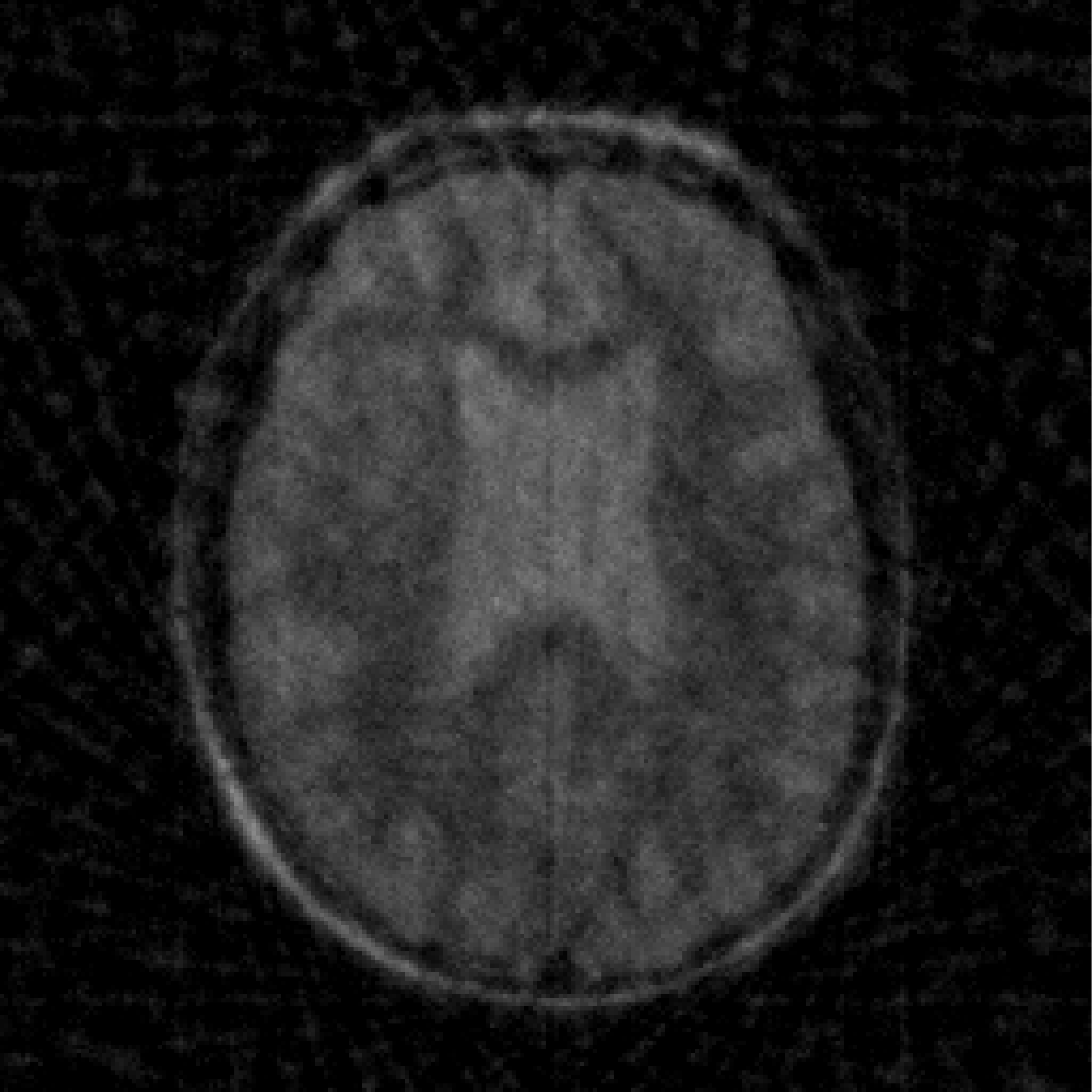}
\end{minipage}&\hspace{-0.45cm}
\begin{minipage}{3cm}
\includegraphics[width=3cm]{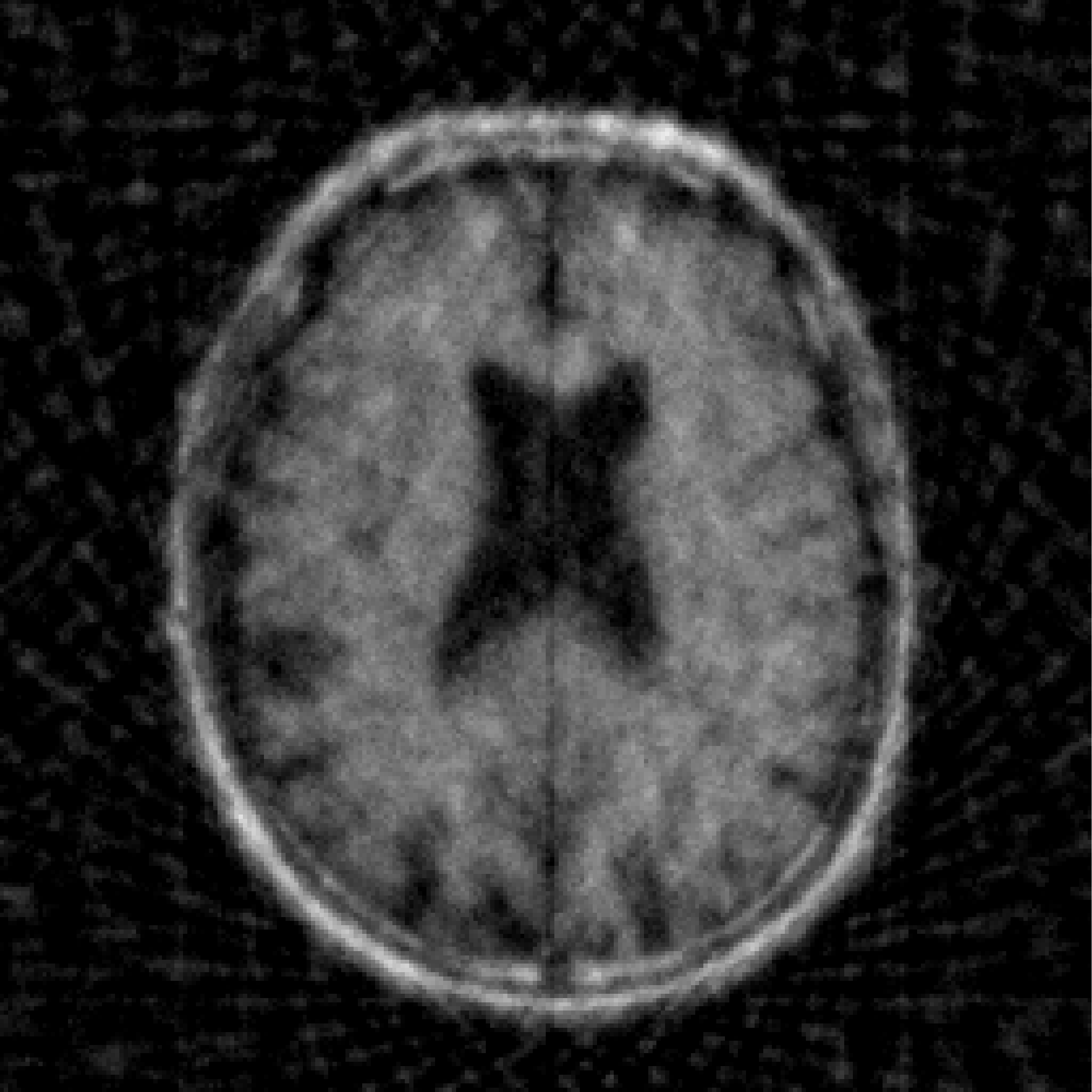}
\end{minipage}&\hspace{-0.45cm}
\begin{minipage}{3cm}
\includegraphics[width=3cm]{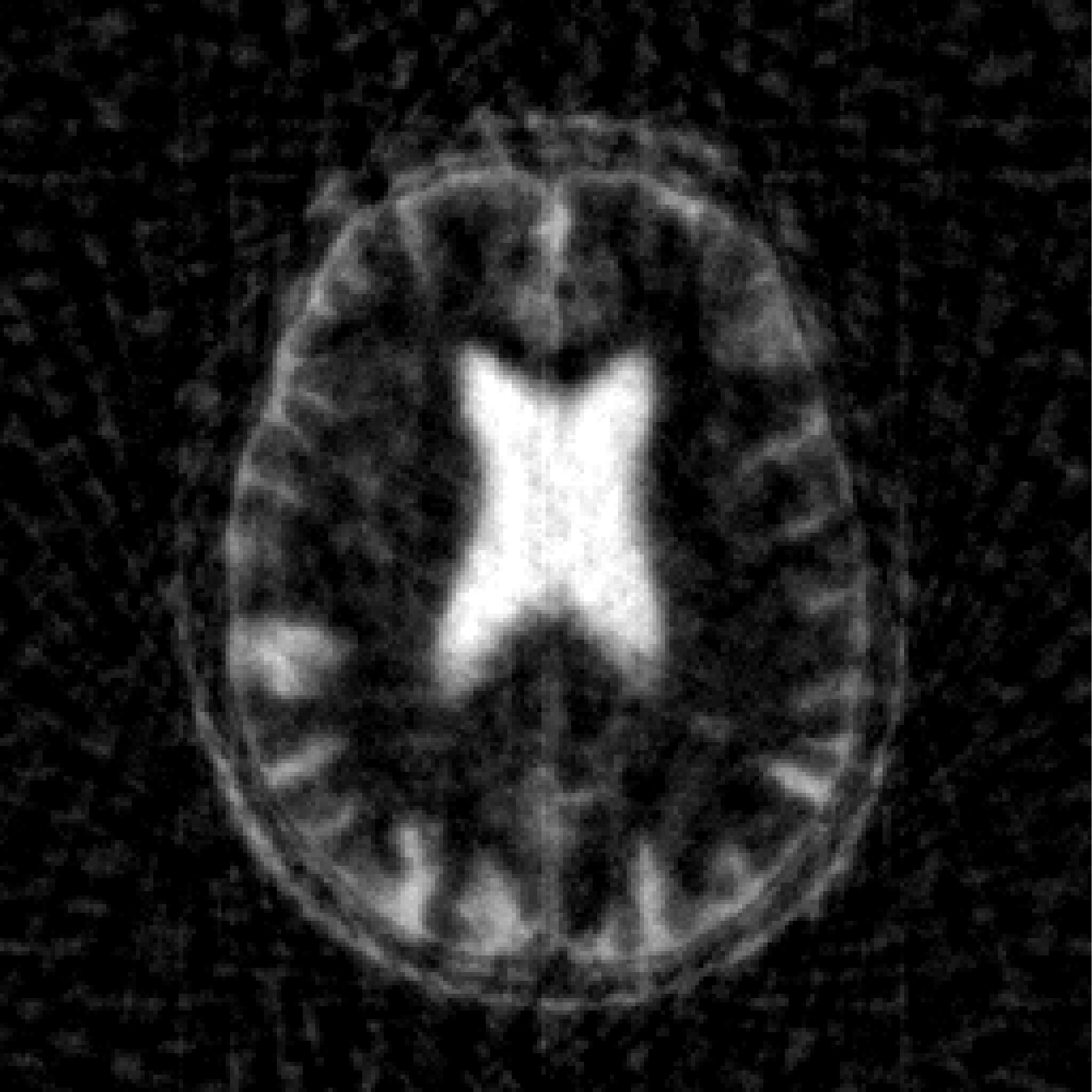}
\end{minipage}\vspace{0.25em}\\
\begin{minipage}{3cm}
\includegraphics[width=3cm]{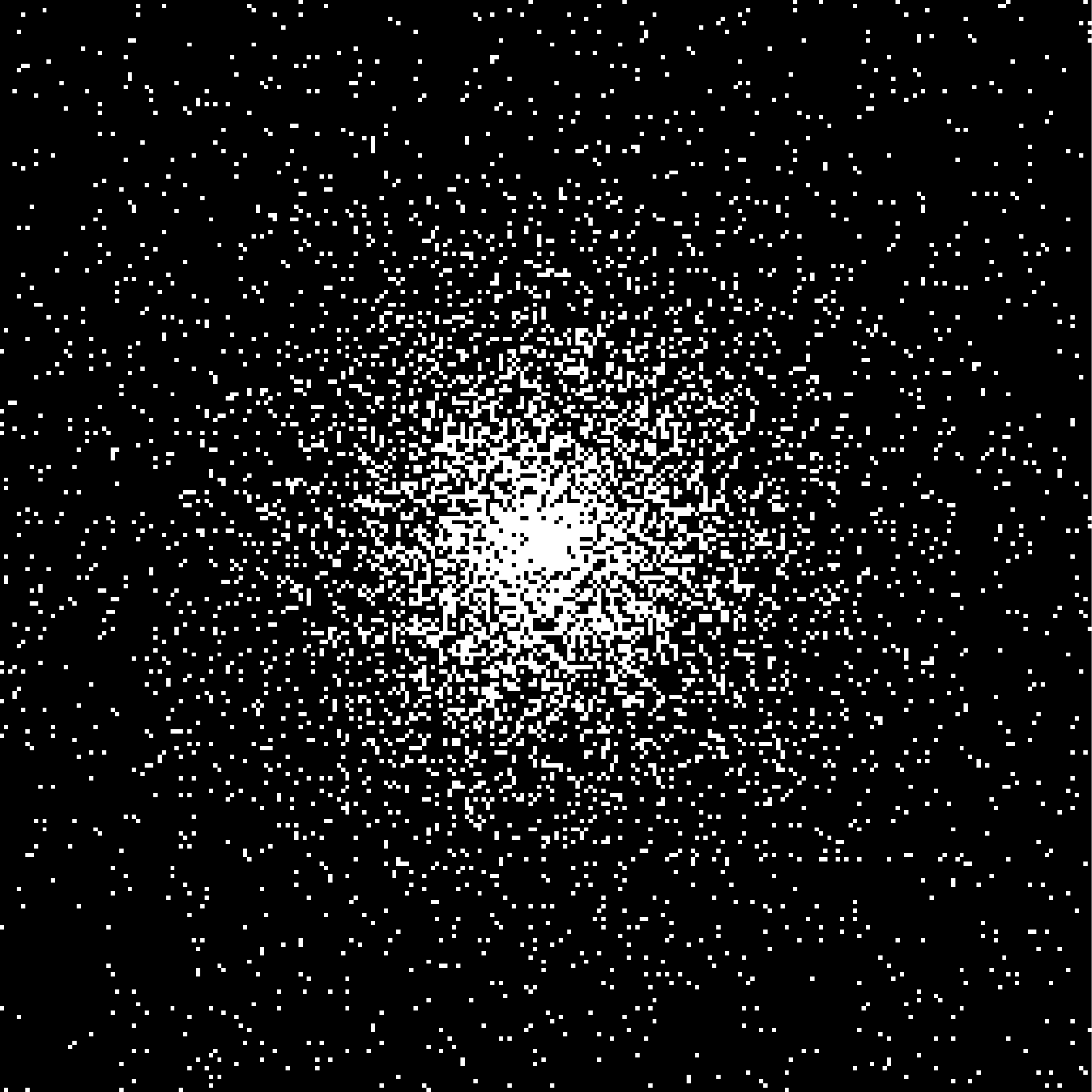}
\end{minipage}&\hspace{-0.45cm}
\begin{minipage}{3cm}
\includegraphics[width=3cm]{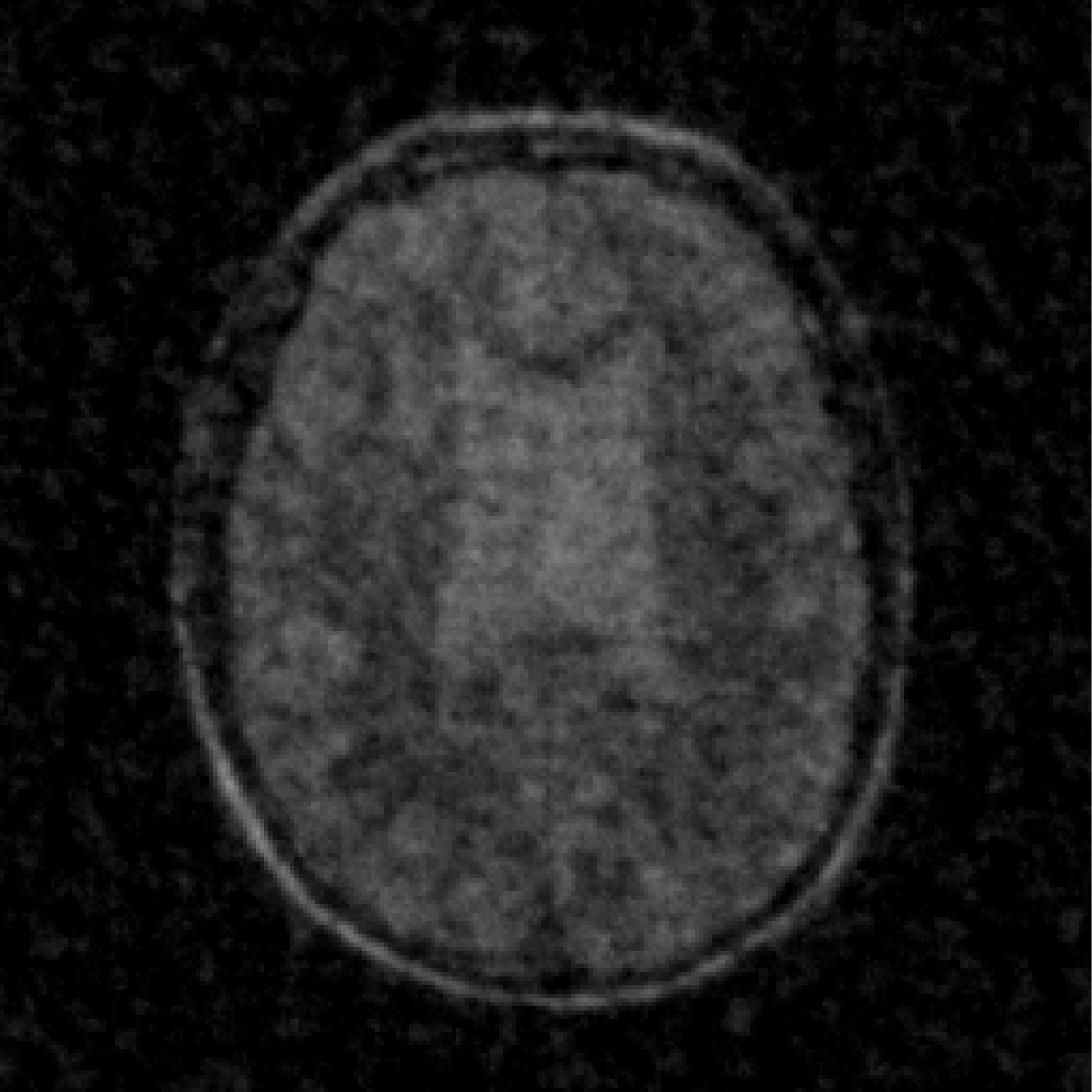}
\end{minipage}&\hspace{-0.45cm}
\begin{minipage}{3cm}
\includegraphics[width=3cm]{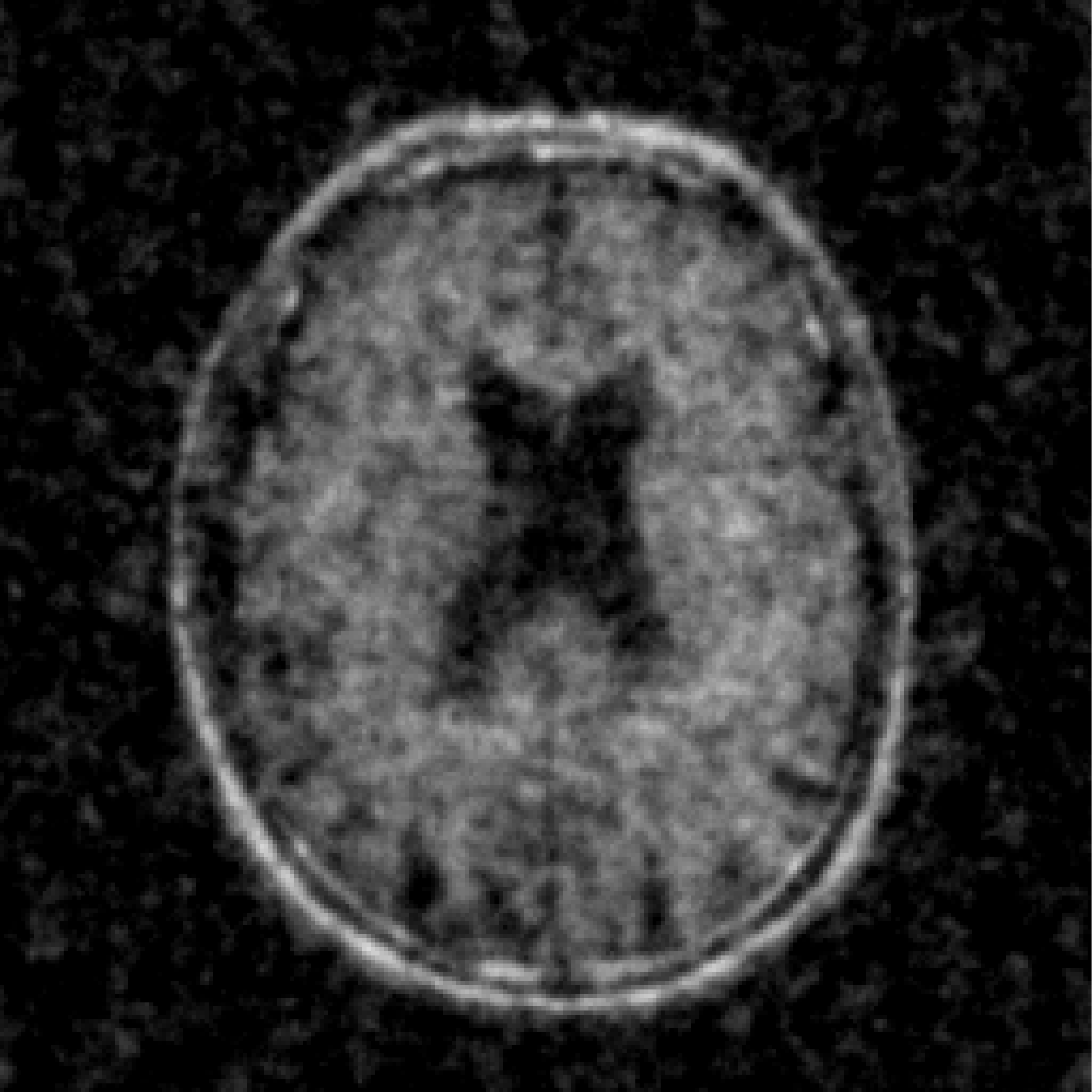}
\end{minipage}&\hspace{-0.45cm}
\begin{minipage}{3cm}
\includegraphics[width=3cm]{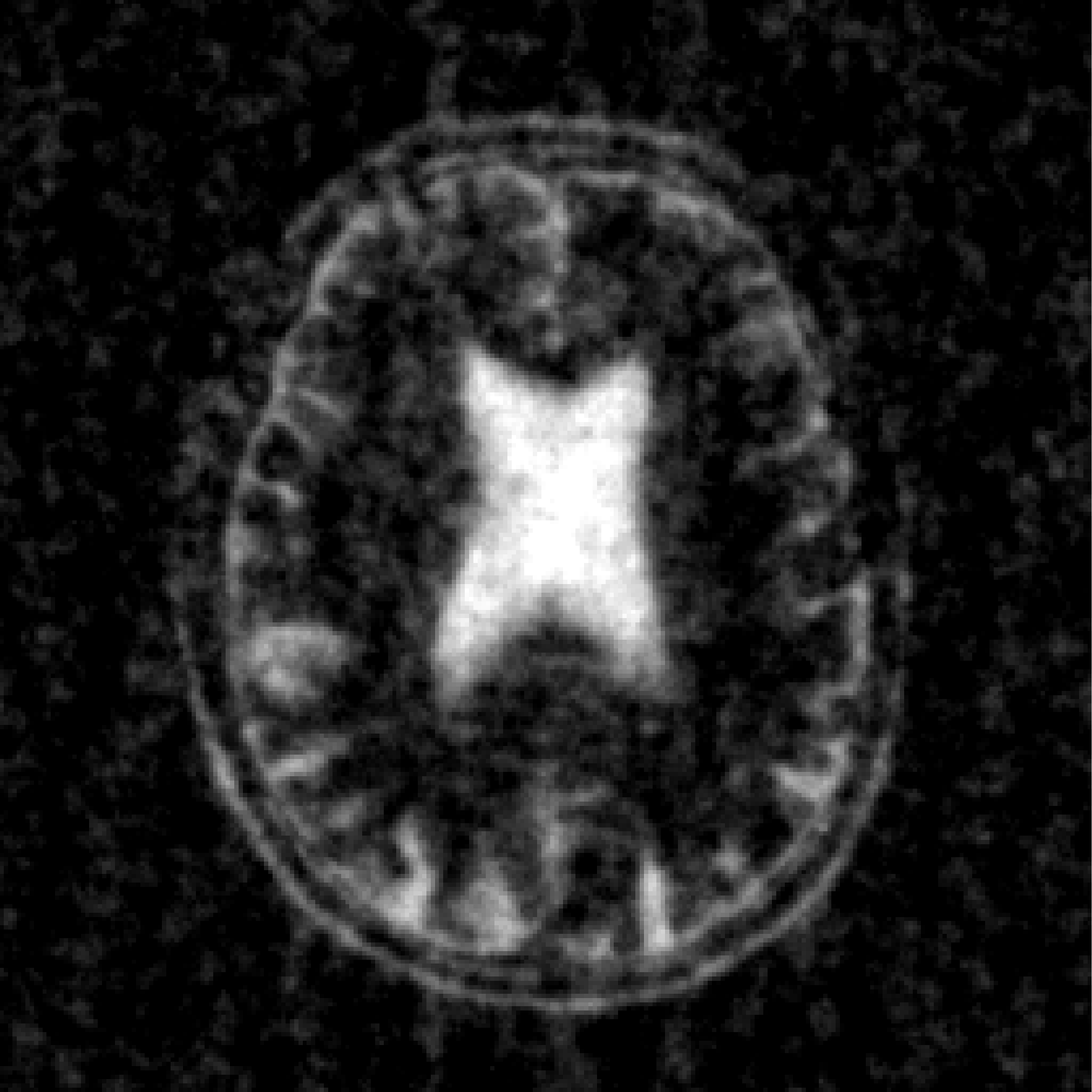}
\end{minipage}
\end{tabular}
\caption{Synthesized data and initializations. The first row describes the synthesized PET data $f$ (left) and $u_1^0$ (right). The second row depicts the radial sampling projection $R_\Lambda$, followed by $u_2^0$ for PET-PD, PET-T1 and PET-T2 respectively. The third row shows the random sampling projection $R_\Lambda$, followed by the initializations of $u_2^0$ for PET-PD, PET-T1 and PET-T2 respectively.}\label{Initializations1}
\end{figure}

\begin{table}[htp!]\tiny
\centering
	\begin{tabular}{|c|c|c|c|c|c|c|c|c|c|c|c|}
		\hline
		\multicolumn{10}{|c|}{Radial Sampling}\\ \hline
		\multicolumn{2}{|c|}{Images}&Indices&Initial&Analysis&DDTF&QPLS&JAnal&JSTF&JSDDTF\\ \hline
		\multirow{6}{*}{PET-PD}&\multirow{3}{*}{PET}&RelErr&$0.3156$&$0.1112$&$0.0939$&$0.1040$&$0.0945$&$0.0937$&$\textbf{0.0905}$\\ 
		&&PSNR&$18.8555$&$27.9170$&$29.3820$&$28.5011$&$29.3297$&$29.4054$&$\textbf{29.7059}$\\ 
		&&Corr&$0.9384$&$0.9918$&$0.9942$&$0.9928$&$0.9942$&$0.9943$&$\textbf{0.9946}$\\ \cline{2-10}
		&\multirow{3}{*}{MRI}&RelErr&$0.3448$&$0.2148$&$0.2129$&$0.2878$&$0.2732$&$0.2022$&$\textbf{0.1929}$\\ 
		&&PSNR&$24.0252$&$28.1347$&$28.2127$&$25.5958$&$26.0482$&$28.6616$&$\textbf{29.0692}$\\ 
		&&Corr&$0.9018$&$0.9633$&$0.9648$&$0.9350$&$0.9410$&$0.9678$&$\textbf{0.9707}$\\ \hline
		\multirow{6}{*}{PET-T1}&\multirow{3}{*}{PET}&RelErr&$0.3156$&$0.1112$&$0.0939$&$0.1034$&$0.0952$&$0.0934$&$\textbf{0.0893}$\\ 
		&&PSNR&$18.8555$&$27.9170$&$29.3820$&$28.5453$&$29.2649$&$29.4358$&$\textbf{29.8250}$\\ 
		&&Corr&$0.9384$&$0.9918$&$0.9942$&$0.9929$&$0.9941$&$0.9943$&$\textbf{0.9948}$\\ \cline{2-10}
		&\multirow{3}{*}{MRI}&RelErr&$0.3571$&$0.1992$&$0.1809$&$0.2266$&$0.2146$&$0.1753$&$\textbf{0.1663}$\\ 
		&&PSNR&$19.9312$&$25.0007$&$25.8395$&$23.8804$&$24.3556$&$26.1104$&$\textbf{26.5667}$\\ 
		&&Corr&$0.8967$&$0.9695$&$0.9752$&$0.9601$&$0.9641$&$0.9764$&$\textbf{0.9787}$\\ \hline
		\multirow{6}{*}{PET-T2}&\multirow{3}{*}{PET}&RelErr&$0.3156$&$0.1112$&$0.0939$&$0.1069$&$0.0964$&$0.0923$&$\textbf{0.0881}$\\ 
		&&PSNR&$18.8555$&$27.9170$&$29.3820$&$28.2628$&$29.1616$&$29.5351$&$\textbf{29.9369}$\\ 
		&&Corr&$0.9384$&$0.9918$&$0.9942$&$0.9924$&$0.9940$&$0.9944$&$\textbf{0.9949}$\\ \cline{2-10}
		&\multirow{3}{*}{MRI}&RelErr&$0.3907$&$0.2415$&$0.2285$&$0.2678$&$0.2497$&$0.2108$&$\textbf{0.2023}$\\ 
		&&PSNR&$19.9442$&$24.1204$&$24.6029$&$23.2245$&$23.8311$&$25.3016$&$\textbf{25.6593}$\\ 
		&&Corr&$0.9005$&0.9635&$0.9676$&$0.9551$&$0.9608$&$0.9723$&$\textbf{0.9745}$\\ \hline\hline
		\multicolumn{10}{|c|}{Random Sampling}\\ \hline
		\multicolumn{2}{|c|}{Images}&Indices&Initial&Analysis&DDTF&QPLS&JAnal&JSTF&JSDDTF\\ \hline
		\multirow{6}{*}{PET-PD}&\multirow{3}{*}{PET}&RelErr&$0.3156$&$0.1112$&$0.0939$&$0.1027$&$0.0943$&$0.0933$&$\textbf{0.0906}$\\ 
		&&PSNR&$18.8555$&$27.9170$&$29.3820$&$28.6088$&$29.3456$&$29.4376$&$\textbf{29.6987}$\\ 
		&&Corr&$0.9384$&$0.9918$&$0.9942$&$0.9930$&$0.9942$&$0.9943$&$\textbf{0.9946}$\\ \cline{2-10}
		&\multirow{3}{*}{MRI}&RelErr&$0.3366$&$0.2081$&$0.2082$&$0.2729$&$0.2618$&$0.2003$&$\textbf{0.1919}$\\ 
		&&PSNR&$24.2355$&$28.4134$&$28.4054$&$26.0581$&$26.4169$&$28.7440$&$\textbf{29.1145}$\\ 
		&&Corr&$0.9064$&$0.9657$&$0.9663$&$0.9408$&$0.9452$&$0.9683$&$\textbf{0.9709}$\\ \hline
		\multirow{6}{*}{PET-T1}&\multirow{3}{*}{PET}&RelErr&$0.3156$&$0.1112$&$0.0939$&$0.1036$&$0.0950$&$0.0929$&$\textbf{0.0889}$\\ 
		&&PSNR&$18.8555$&$27.9170$&$29.3820$&$28.5293$&$29.2815$&$29.4772$&$\textbf{29.8588}$\\ 
		&&Corr&$0.9384$&$0.9918$&$0.9942$&$0.9929$&$0.9941$&$0.9943$&$\textbf{0.9948}$\\ \cline{2-10}
		&\multirow{3}{*}{MRI}&RelErr&$0.3712$&$0.1952$&$0.1785$&$0.2180$&$0.2093$&$0.1707$&$\textbf{0.1627}$\\ 
		&&PSNR&$19.5948$&$25.1765$&$25.9556$&$24.2185$&$24.5722$&$26.3422$&$\textbf{26.7584}$\\ 
		&&Corr&$0.8878$&$0.9709$&$0.9759$&$0.9629$&$0.9658$&$0.9777$&$\textbf{0.9797}$\\ \hline
		\multirow{6}{*}{PET-T2}&\multirow{3}{*}{PET}&RelErr&$0.3156$&$0.1112$&$0.0939$&$0.1072$&$0.0964$&$0.0928$&$\textbf{0.0882}$\\ 
		&&PSNR&$18.8555$&$27.9170$&$29.3820$&$28.2321$&$29.1553$&$29.4865$&$\textbf{29.9256}$\\ 
		&&Corr&$0.9384$&$0.9918$&$0.9942$&$0.9924$&$0.9939$&$0.9943$&$\textbf{0.9949}$\\ \cline{2-10}
		&\multirow{3}{*}{MRI}&RelErr&$0.4052$&$0.2357$&$0.2270$&$0.2551$&$0.2397$&$0.2089$&$\textbf{0.2016}$\\ 
		&&PSNR&$19.6260$&$24.3331$&$24.6594$&$23.6451$&$24.1873$&$25.3805$&$\textbf{25.6890}$\\ 
		&&Corr&$0.8925$&0.9654&$0.9680$&$0.9591$&$0.9638$&$0.9728$&$\textbf{0.9746}$\\ \hline
	\end{tabular}
	\caption{Comparison of relative errors, PSNR, and correlations.}\label{Table1}
\end{table}

\cref{Table1} summarizes relative errors, PSNR values, and correlations of the aforementioned five restoration models, and \cref{PETMRPD15RadialResults}-\cref{PETMRT215RandomResults} present visual comparisons of the results. We can see that both JSTF \eqref{Proposed} and JSDDTF model \eqref{DDTFPETMRI} consistently outperforms both the individual reconstruction models and the existing joint reconstruction models in \cite{M.J.Ehrhardt2015}. This verifies that there exists a correlation on image singularities between the two modality images, and exploiting this correlation results in the better reconstruction results. In addition, compared to the existing joint reconstruction models, we can see that our proposed models introduce less artifacts in both modality images, leading to the visual improvements over the existing joint reconstruction methods which are consistent with the improvements in indices. It is worth noting that our proposed model can improve the restoration qualities even using a static wavelet tight frame. This demonstrates that the improvements mainly come from simultaneously considering the different regularity of the different modality images and the joint sparsity. Meanwhile, we can see that the QPLS model and the JAnal model \eqref{JAnalPETMRI}, which only take the structural correlation into account, show the degradations in MRI restoration results compared to the independent reconstruction methods due to the artifacts. Finally, we list some zoom-in views in \cref{PETMRT215RadialResultsZoom} to illustrate that our models \eqref{Proposed} and \eqref{DDTFPETMRI} restore structures better than the existing methods.

\begin{figure}[htp!]
	\centering
	\begin{tabular}{cccc}
		\begin{minipage}{3cm}
			\includegraphics[width=3cm]{PET15Original.pdf}
		\end{minipage}&\hspace{-0.45cm}
\begin{minipage}{3cm}
			\includegraphics[width=3cm]{PET15EM.pdf}
		\end{minipage}&\hspace{-0.45cm}
		\begin{minipage}{3cm}
			\includegraphics[width=3cm]{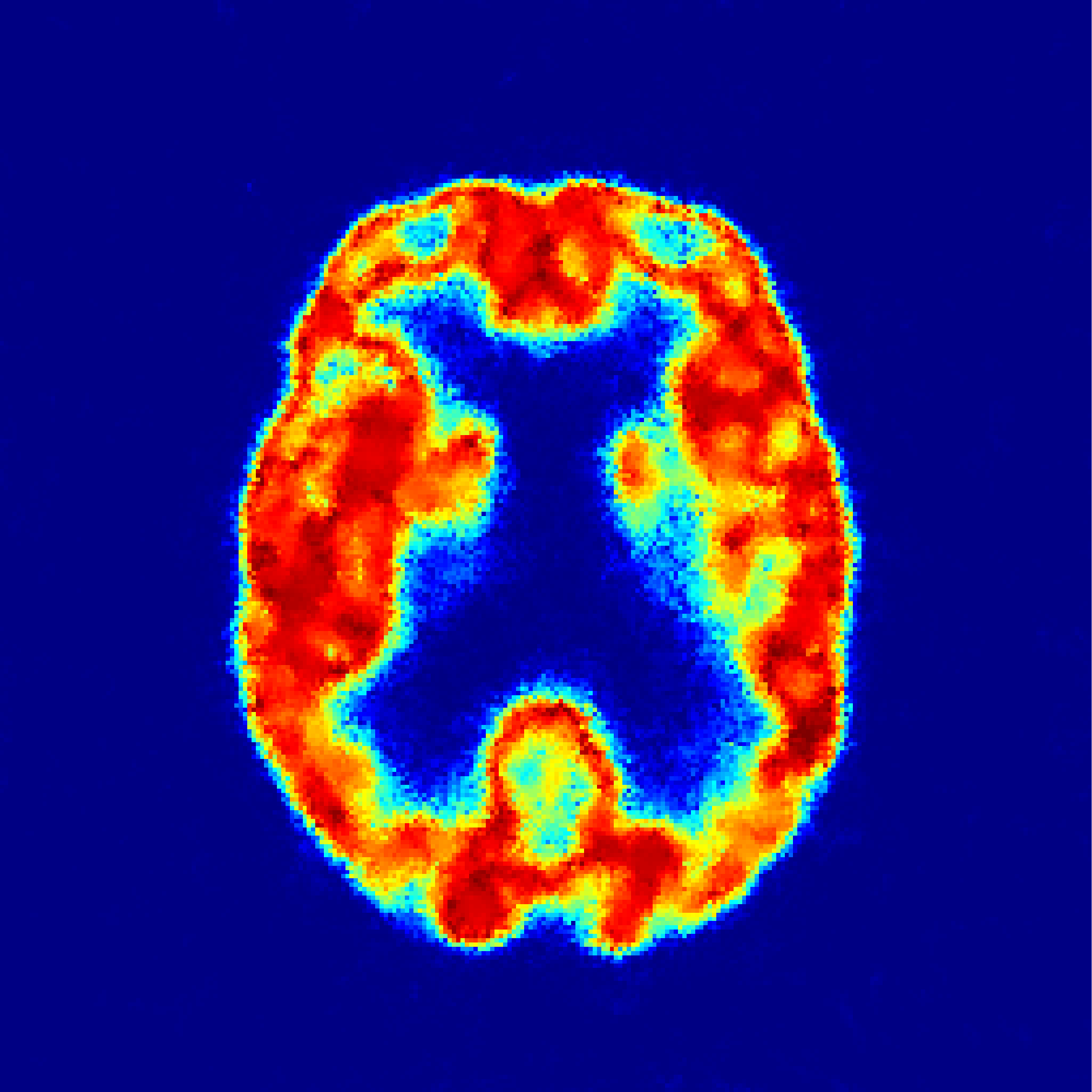}
		\end{minipage}&\hspace{-0.45cm}
		\begin{minipage}{3cm}
			\includegraphics[width=3cm]{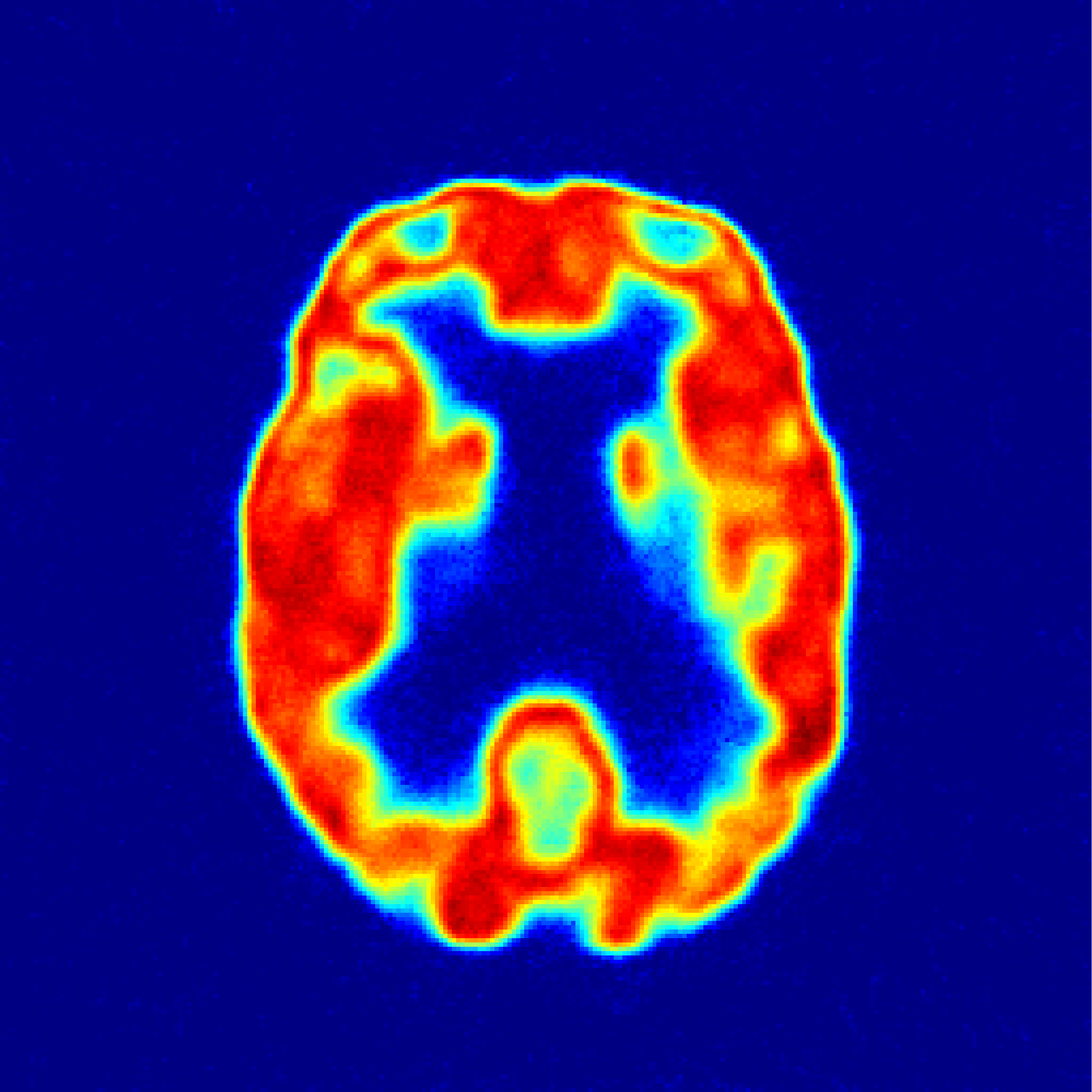}
		\end{minipage}\\
		{\small{Original}}&\hspace{-0.45cm}
{\small{Initial}}&\hspace{-0.45cm}
		{\small{Analysis \eqref{AnaPET}}}&\hspace{-0.45cm}
		{\small{DDTF \eqref{DDTFPET}}}\\
		\begin{minipage}{3cm}
			\includegraphics[width=3cm]{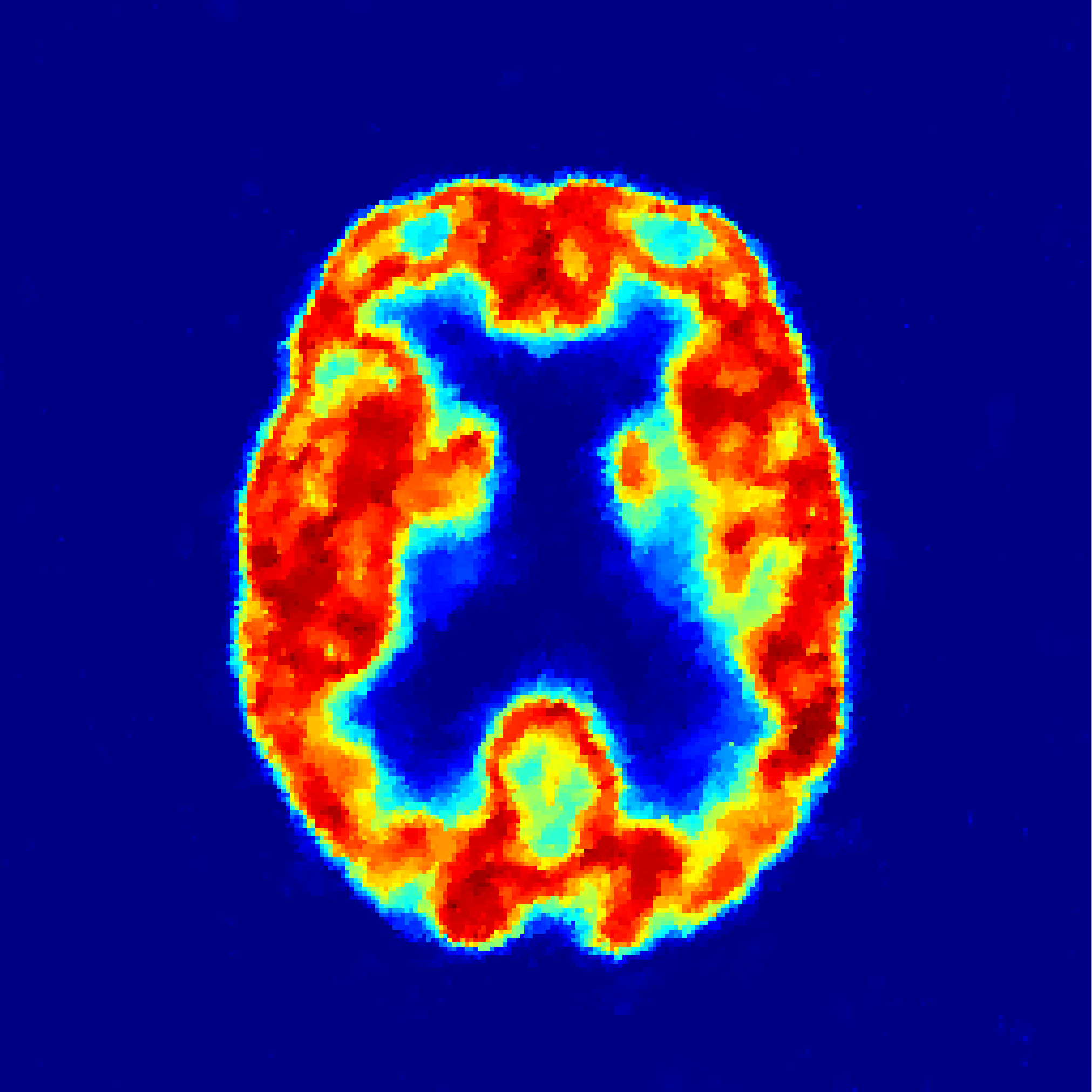}
		\end{minipage}&\hspace{-0.45cm}
		\begin{minipage}{3cm}
			\includegraphics[width=3cm]{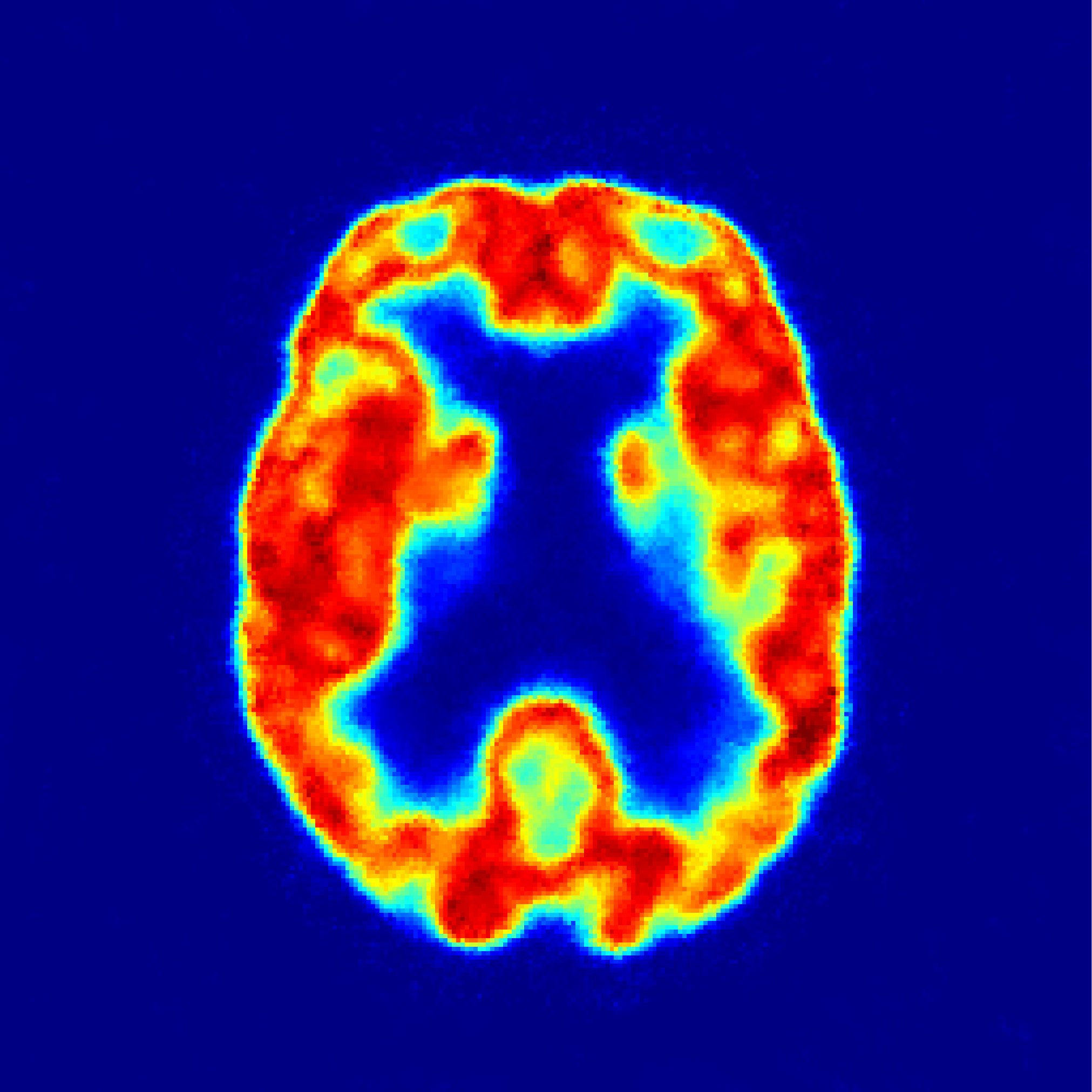}
		\end{minipage}&\hspace{-0.45cm}
		\begin{minipage}{3cm}
			\includegraphics[width=3cm]{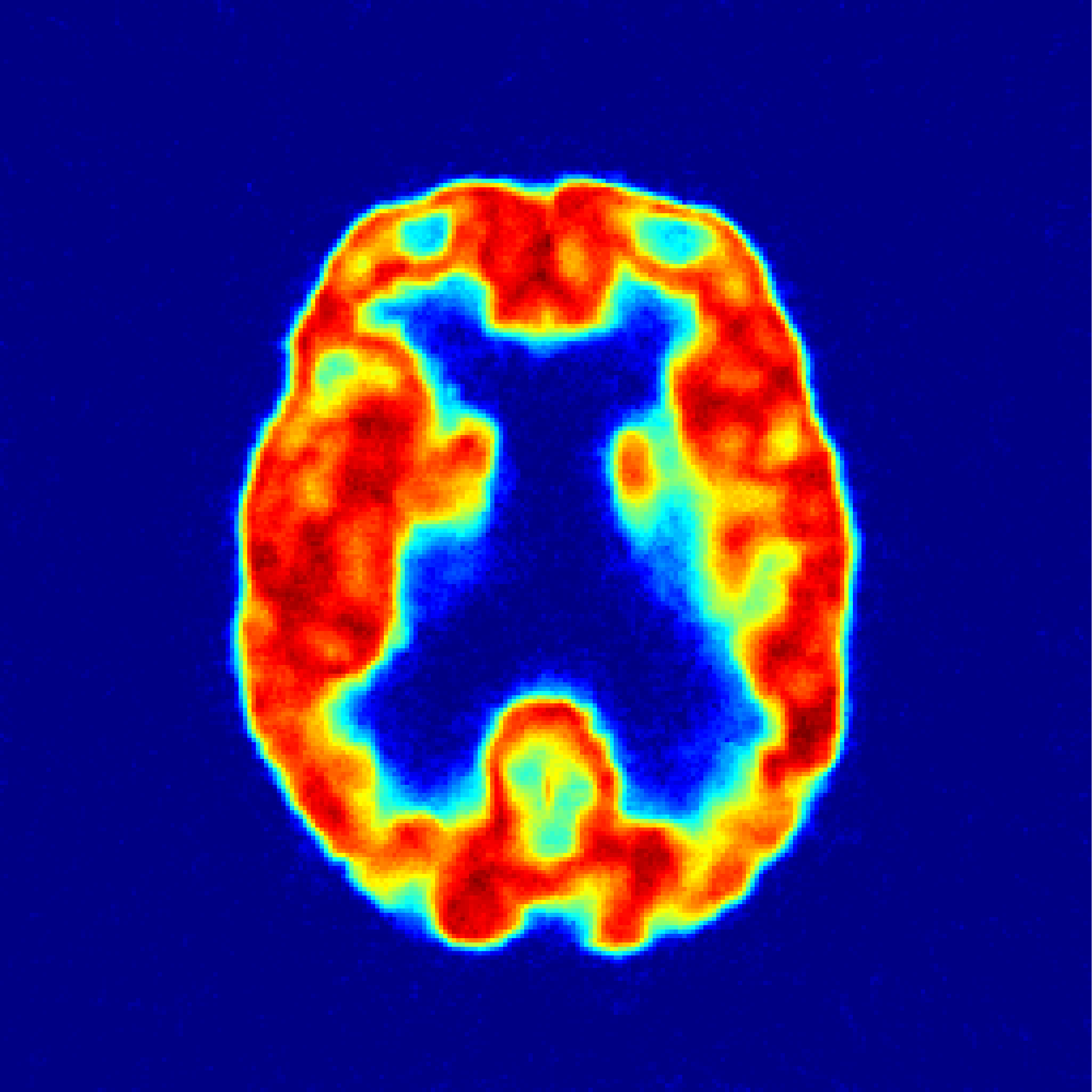}
		\end{minipage}&\hspace{-0.45cm}
		\begin{minipage}{3cm}
			\includegraphics[width=3cm]{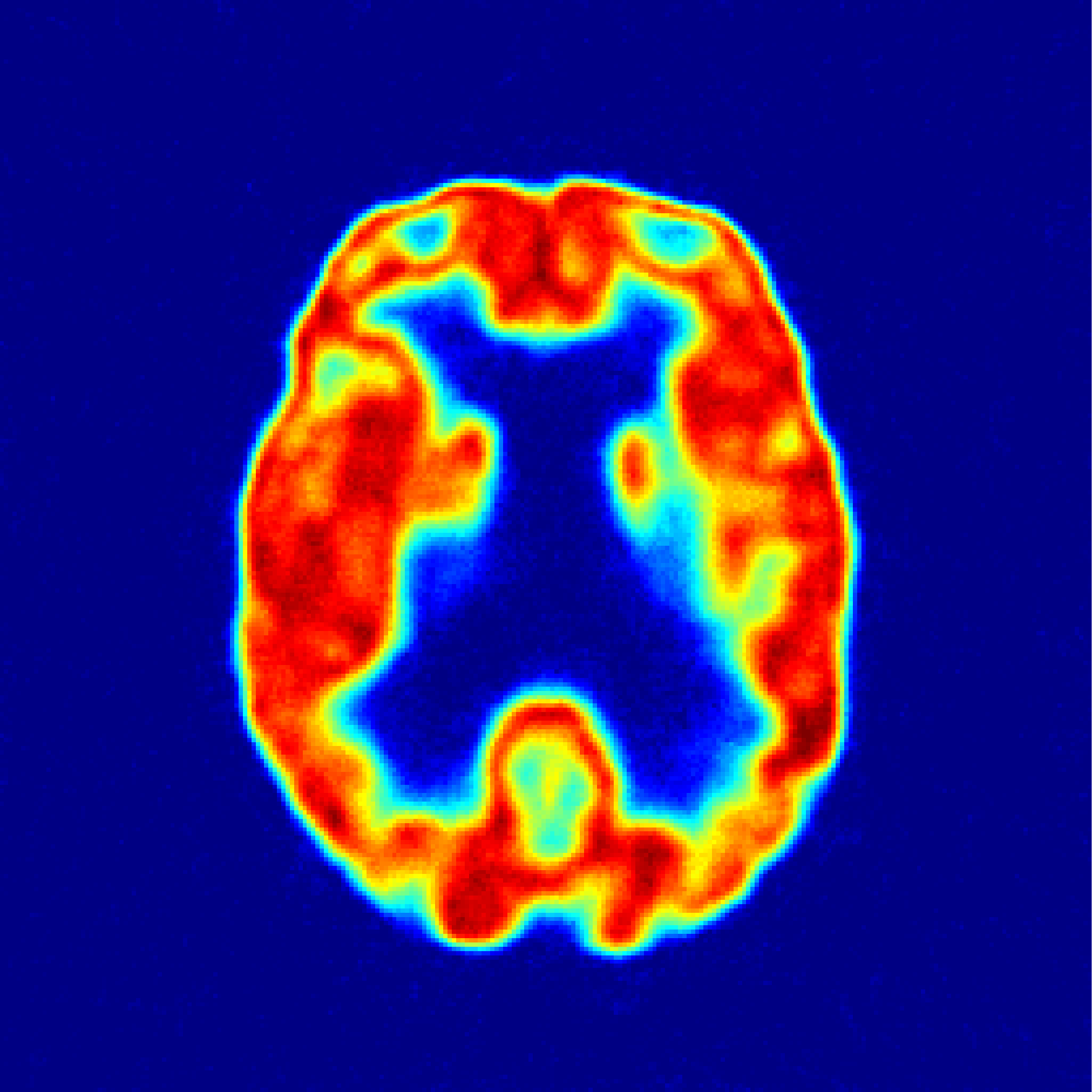}
		\end{minipage}\\
{\small{QPLS \cite{M.J.Ehrhardt2015}}}&\hspace{-0.45cm}
		{\small{JAnal \eqref{JAnalPETMRI}}}&\hspace{-0.45cm}
		{\small{JSTF \eqref{Proposed}}}&\hspace{-0.45cm}
		{\small{JSDDTF \eqref{DDTFPETMRI}}}\\
		\begin{minipage}{3cm}
			\includegraphics[width=3cm]{MRPD15Original.pdf}
		\end{minipage}&\hspace{-0.45cm}
		\begin{minipage}{3cm}
			\includegraphics[width=3cm]{MRPD15Radial30ZeroFill.pdf}
		\end{minipage}&\hspace{-0.45cm}
		\begin{minipage}{3cm}
			\includegraphics[width=3cm]{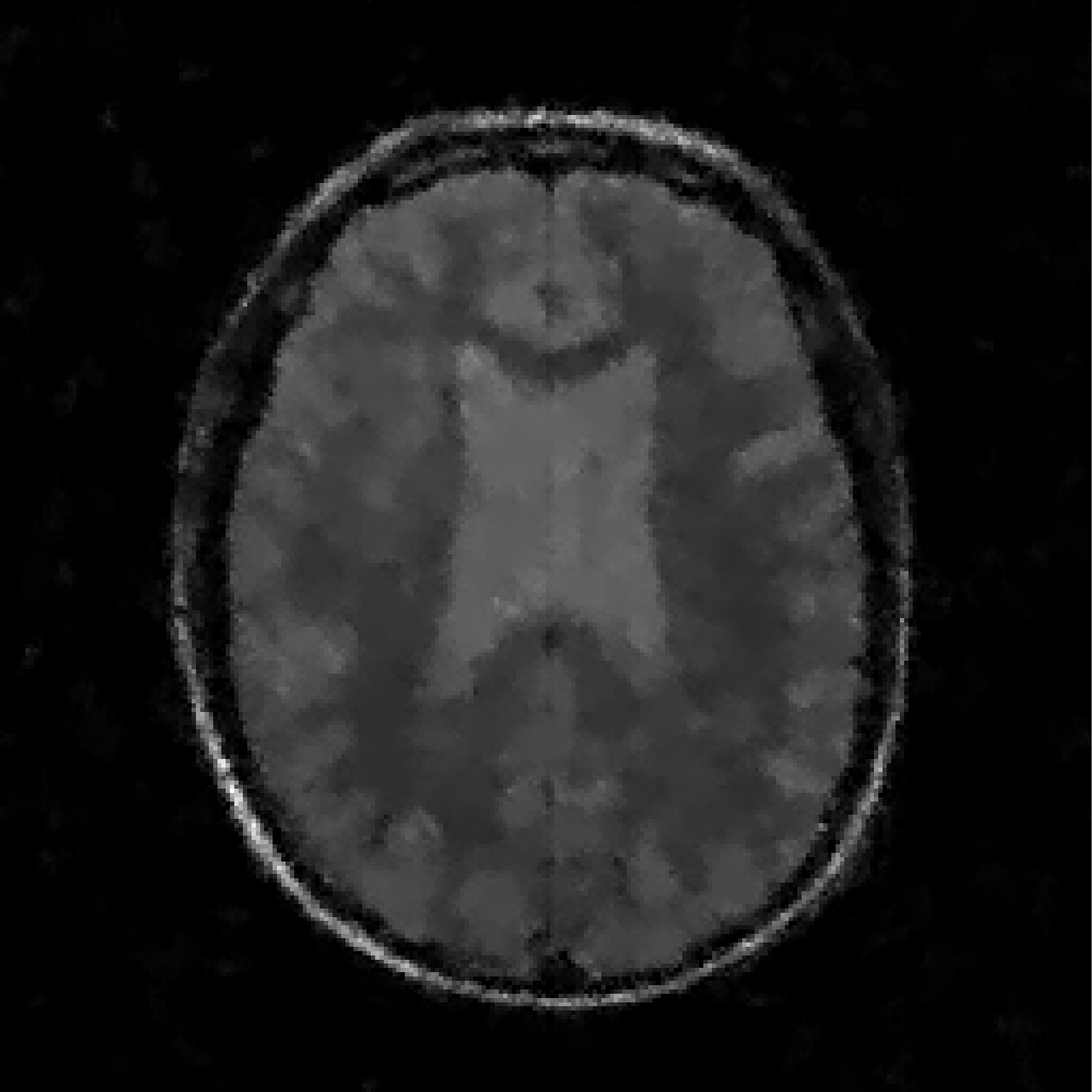}
		\end{minipage}&\hspace{-0.45cm}
\begin{minipage}{3cm}
\includegraphics[width=3cm]{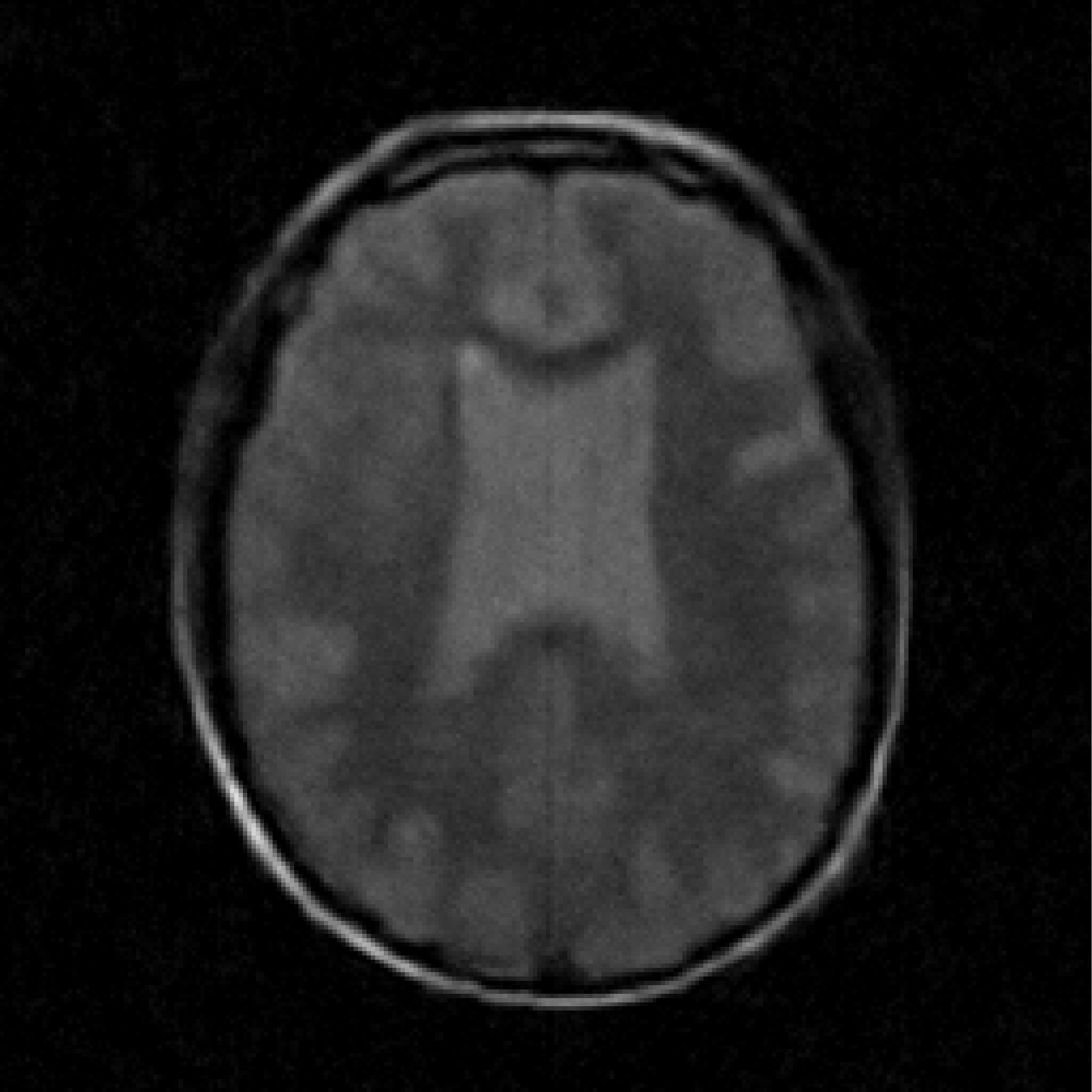}
\end{minipage}\\
		{\small{Original}}&\hspace{-0.45cm}
{\small{Initial}}&\hspace{-0.45cm}
{\small{Analysis \eqref{AnaMRI}}}&\hspace{-0.45cm}
		{\small{DDTF \eqref{DDTFMRI}}}\\
\begin{minipage}{3cm}
			\includegraphics[width=3cm]{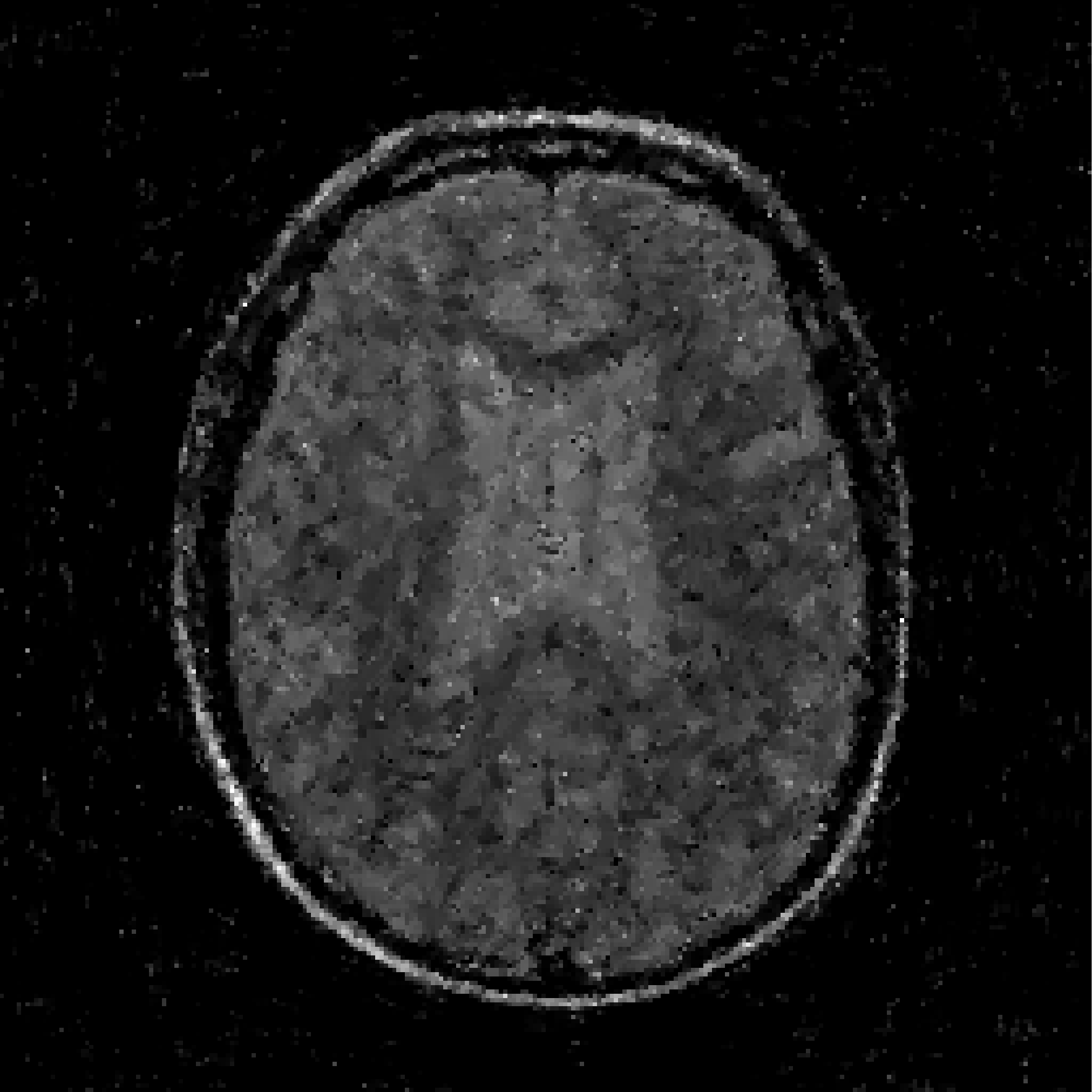}
		\end{minipage}&\hspace{-0.45cm}
		\begin{minipage}{3cm}
			\includegraphics[width=3cm]{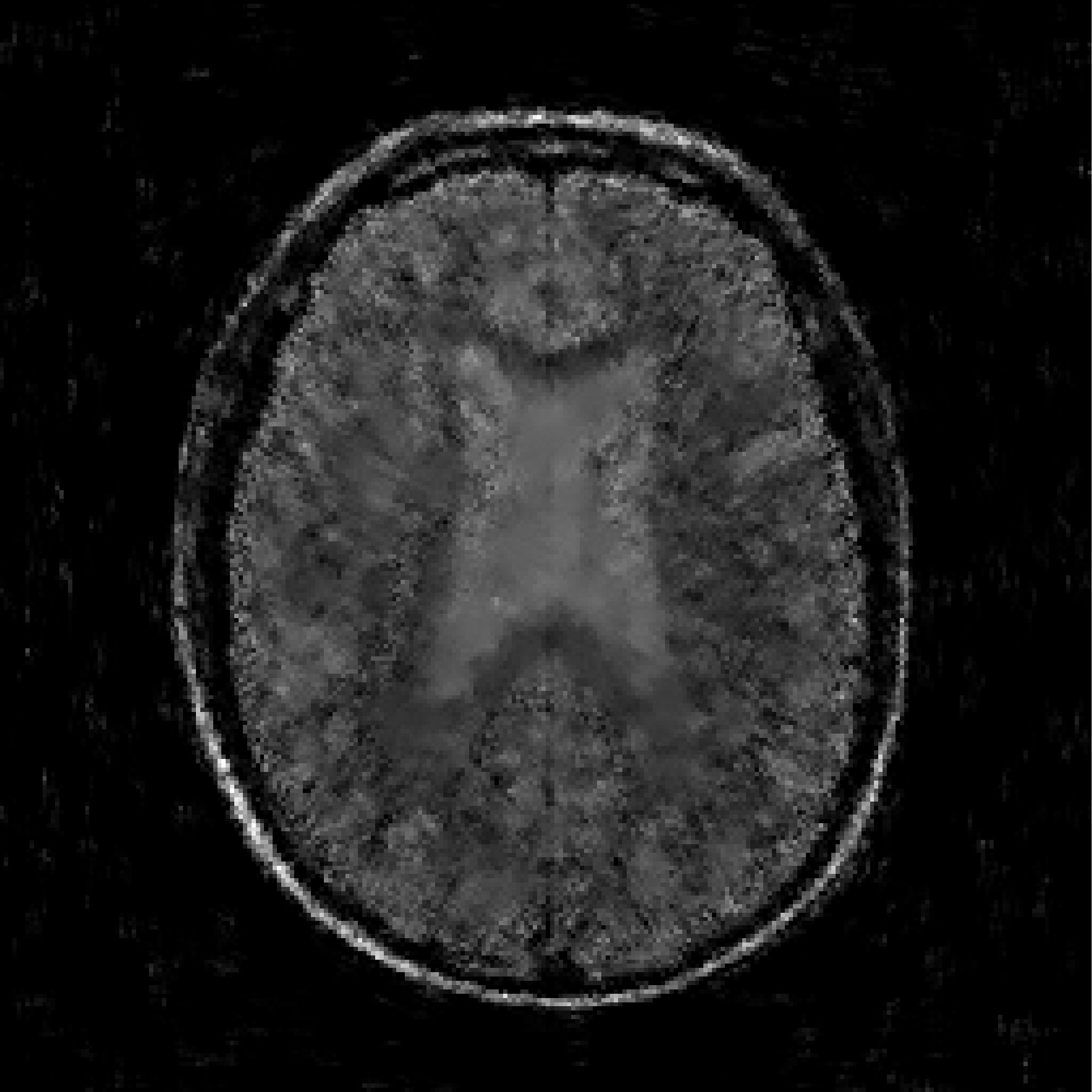}
		\end{minipage}&\hspace{-0.45cm}
		\begin{minipage}{3cm}
			\includegraphics[width=3cm]{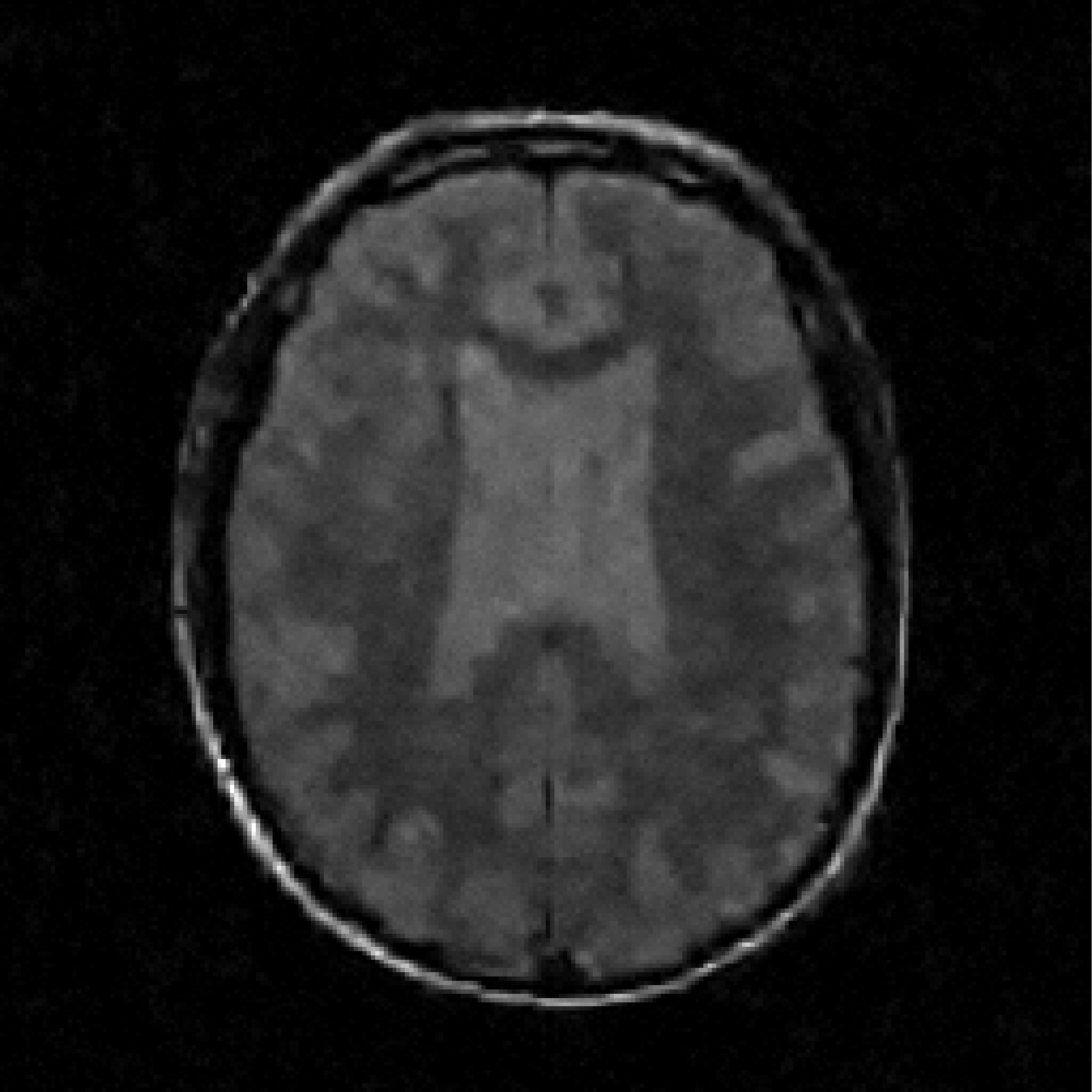}
		\end{minipage}&\hspace{-0.45cm}
		\begin{minipage}{3cm}
			\includegraphics[width=3cm]{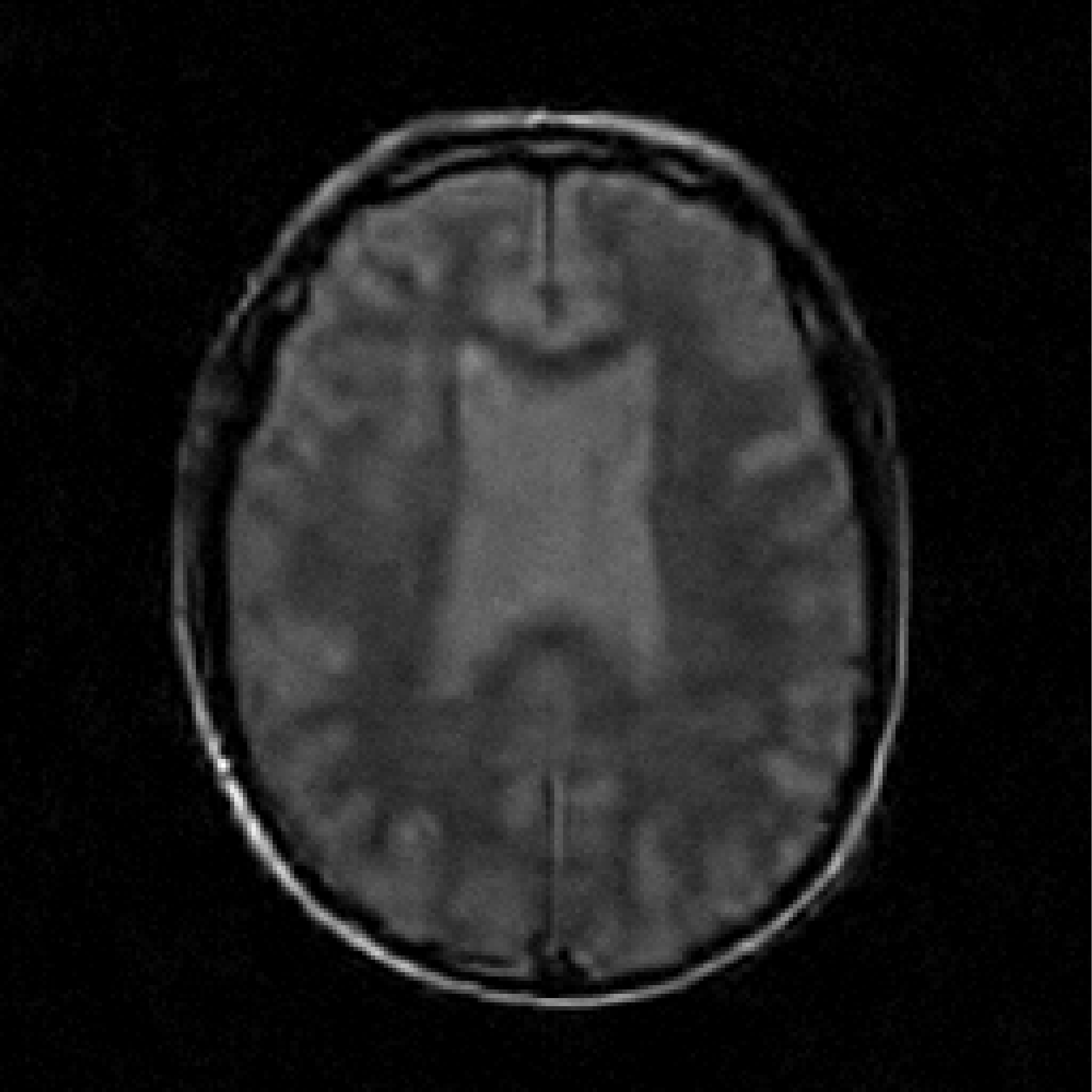}
		\end{minipage}\\
{\small{QPLS \cite{M.J.Ehrhardt2015}}}&\hspace{-0.45cm}
		{\small{JAnal \eqref{JAnalPETMRI}}}&\hspace{-0.45cm}
		{\small{JSTF \eqref{Proposed}}}&\hspace{-0.45cm}
		{\small{JSDDTF \eqref{DDTFPETMRI}}}
	\end{tabular}
	\caption{Visual comparison of PET-PD Radial joint reconstruction results. The first and second rows describe the PET images, and the third and fourth rows depict the MRI images.}\label{PETMRPD15RadialResults}
\end{figure}

\begin{figure}[htp!]
	\centering
		\begin{tabular}{cccc}
		\begin{minipage}{3cm}
			\includegraphics[width=3cm]{PET15Original.pdf}
		\end{minipage}&\hspace{-0.45cm}
\begin{minipage}{3cm}
			\includegraphics[width=3cm]{PET15EM.pdf}
		\end{minipage}&\hspace{-0.45cm}
		\begin{minipage}{3cm}
			\includegraphics[width=3cm]{PET15Anal.pdf}
		\end{minipage}&\hspace{-0.45cm}
		\begin{minipage}{3cm}
			\includegraphics[width=3cm]{PET15DDTF.pdf}
		\end{minipage}\\
		{\small{Original}}&\hspace{-0.45cm}
{\small{Initial}}&\hspace{-0.45cm}
		{\small{Analysis \eqref{AnaPET}}}&\hspace{-0.45cm}
		{\small{DDTF \eqref{DDTFPET}}}\\
		\begin{minipage}{3cm}
			\includegraphics[width=3cm]{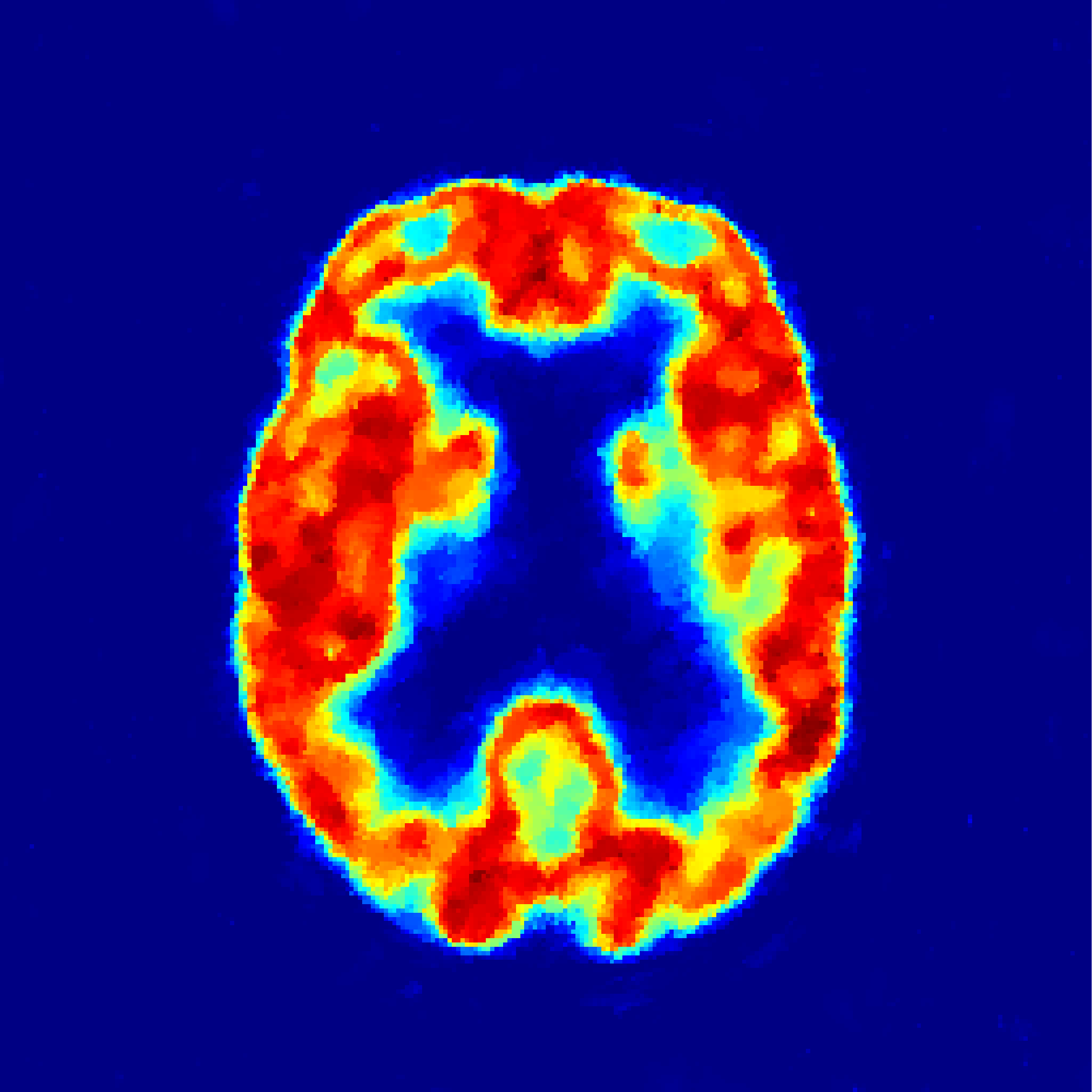}
		\end{minipage}&\hspace{-0.45cm}
		\begin{minipage}{3cm}
			\includegraphics[width=3cm]{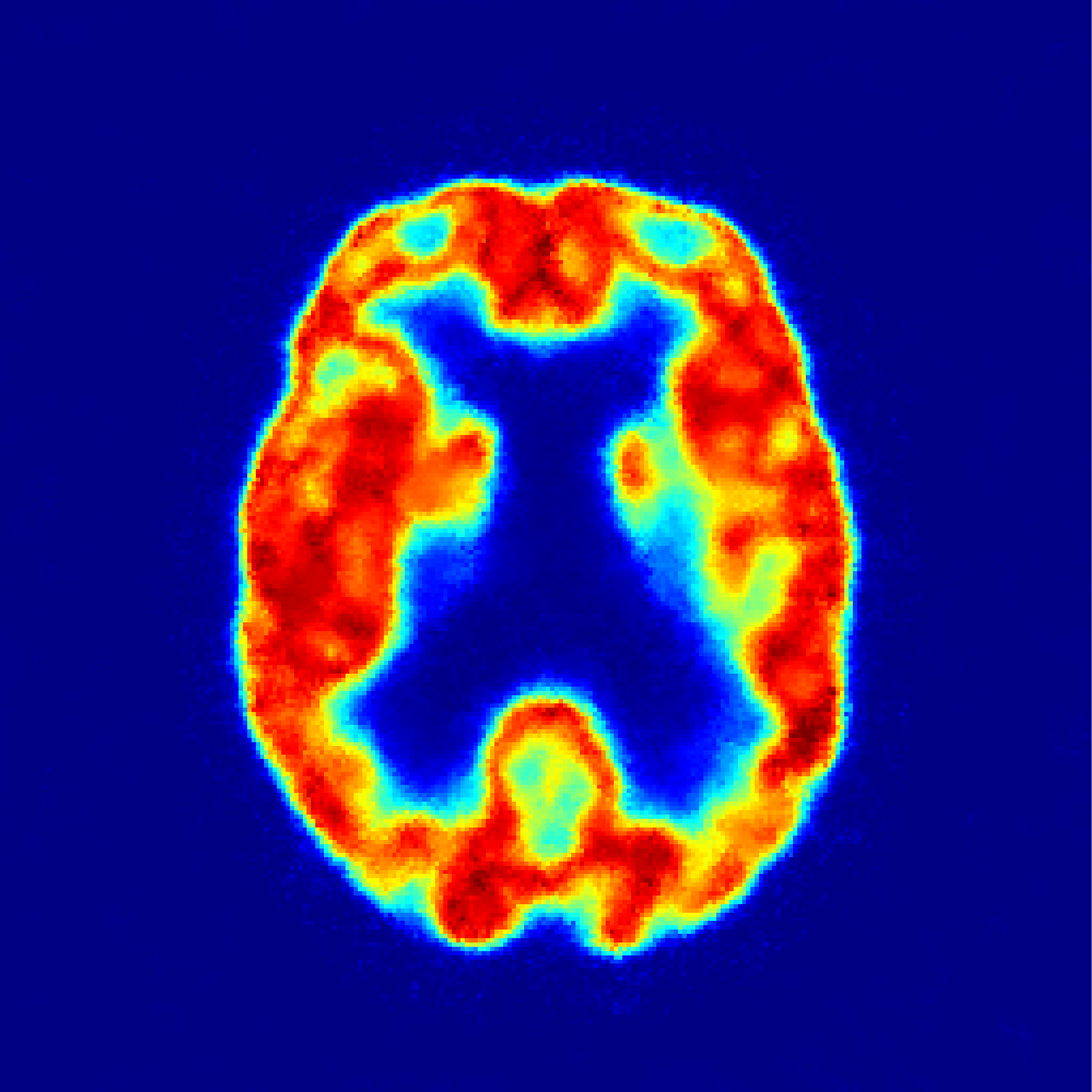}
		\end{minipage}&\hspace{-0.45cm}
		\begin{minipage}{3cm}
			\includegraphics[width=3cm]{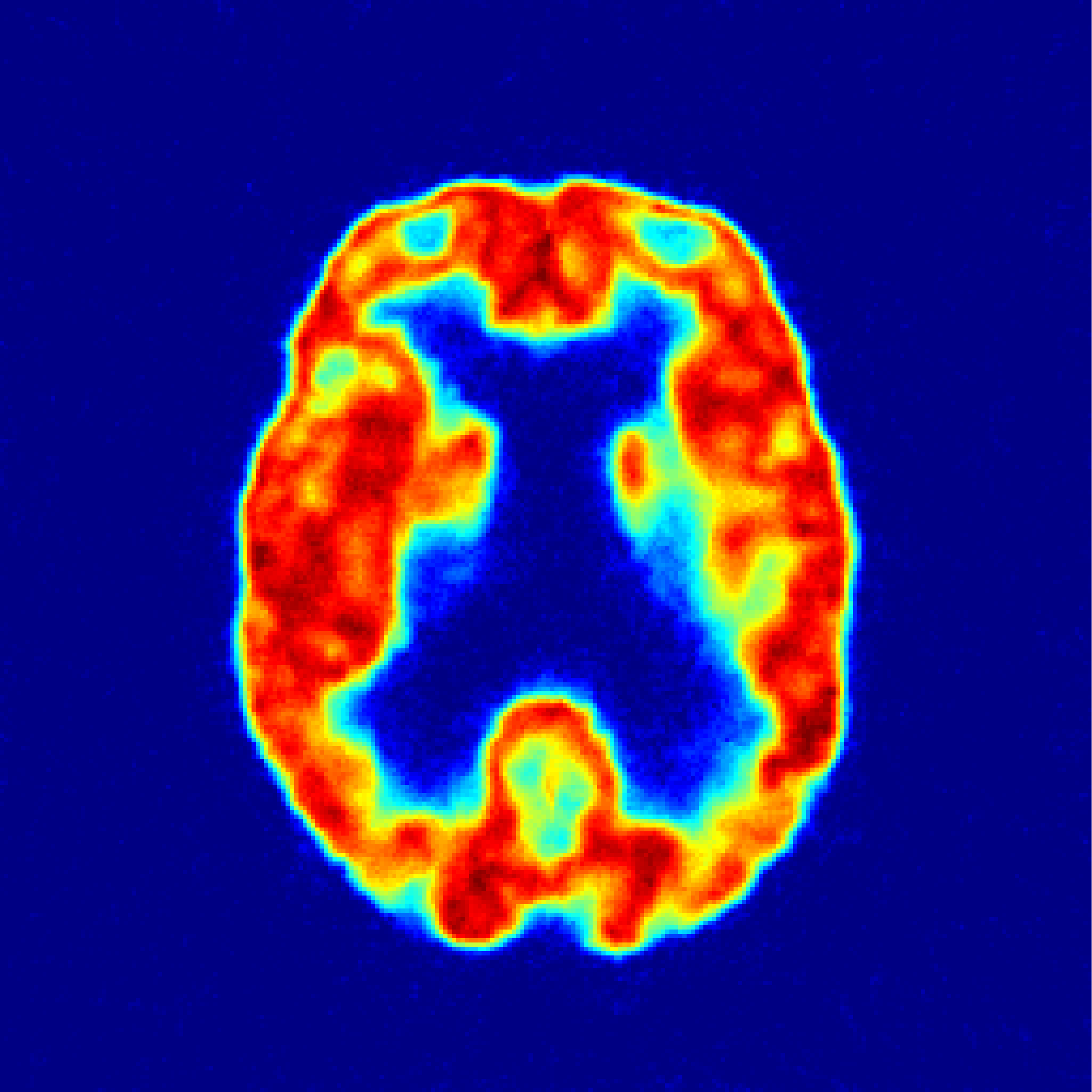}
		\end{minipage}&\hspace{-0.45cm}
		\begin{minipage}{3cm}
			\includegraphics[width=3cm]{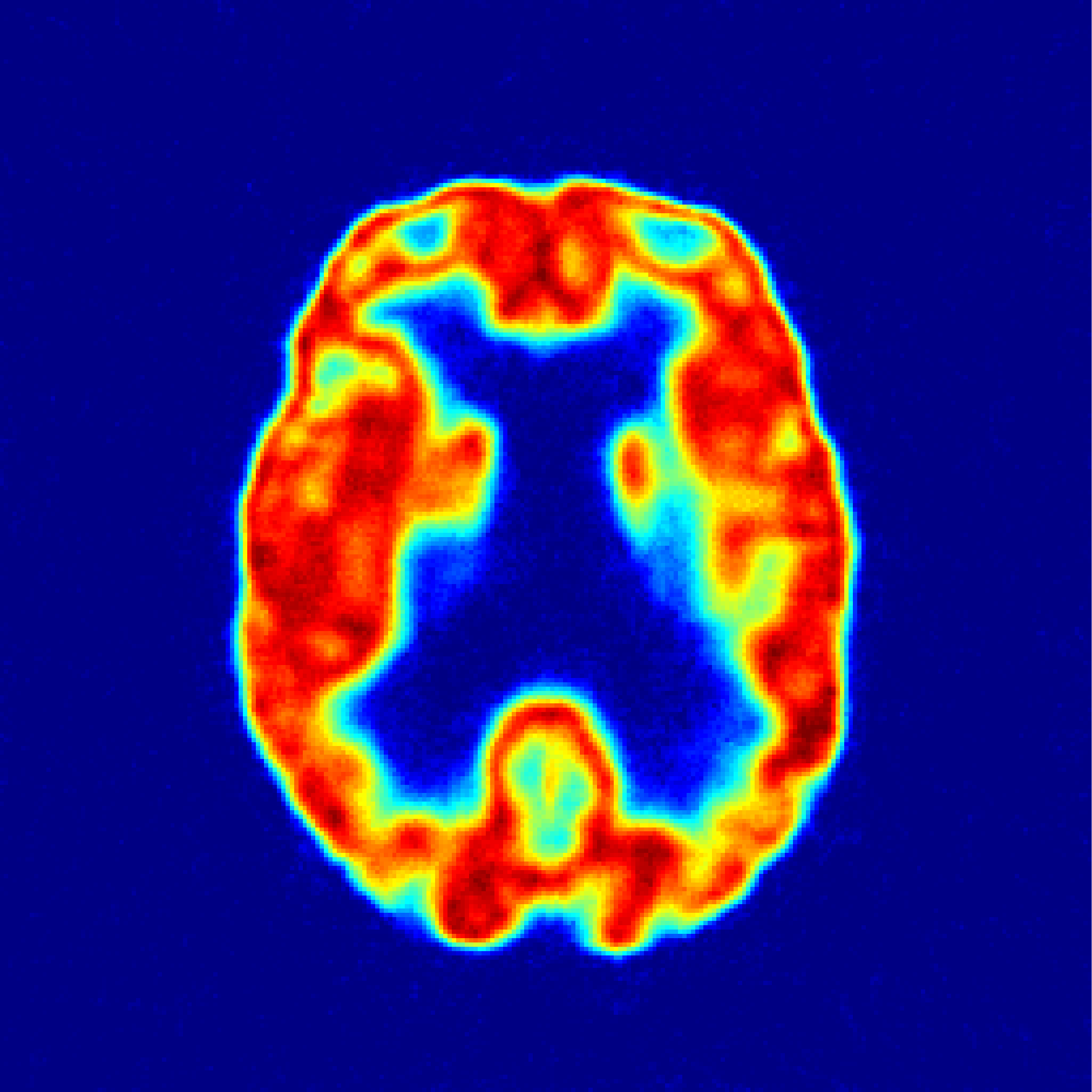}
		\end{minipage}\\
{\small{QPLS \cite{M.J.Ehrhardt2015}}}&\hspace{-0.45cm}
		{\small{JAnal \eqref{JAnalPETMRI}}}&\hspace{-0.45cm}
		{\small{JSTF \eqref{Proposed}}}&\hspace{-0.45cm}
		{\small{JSDDTF \eqref{DDTFPETMRI}}}\\
		\begin{minipage}{3cm}
			\includegraphics[width=3cm]{MRT115Original.pdf}
		\end{minipage}&\hspace{-0.45cm}
		\begin{minipage}{3cm}
			\includegraphics[width=3cm]{MRT115Radial30ZeroFill.pdf}
		\end{minipage}&\hspace{-0.45cm}
		\begin{minipage}{3cm}
			\includegraphics[width=3cm]{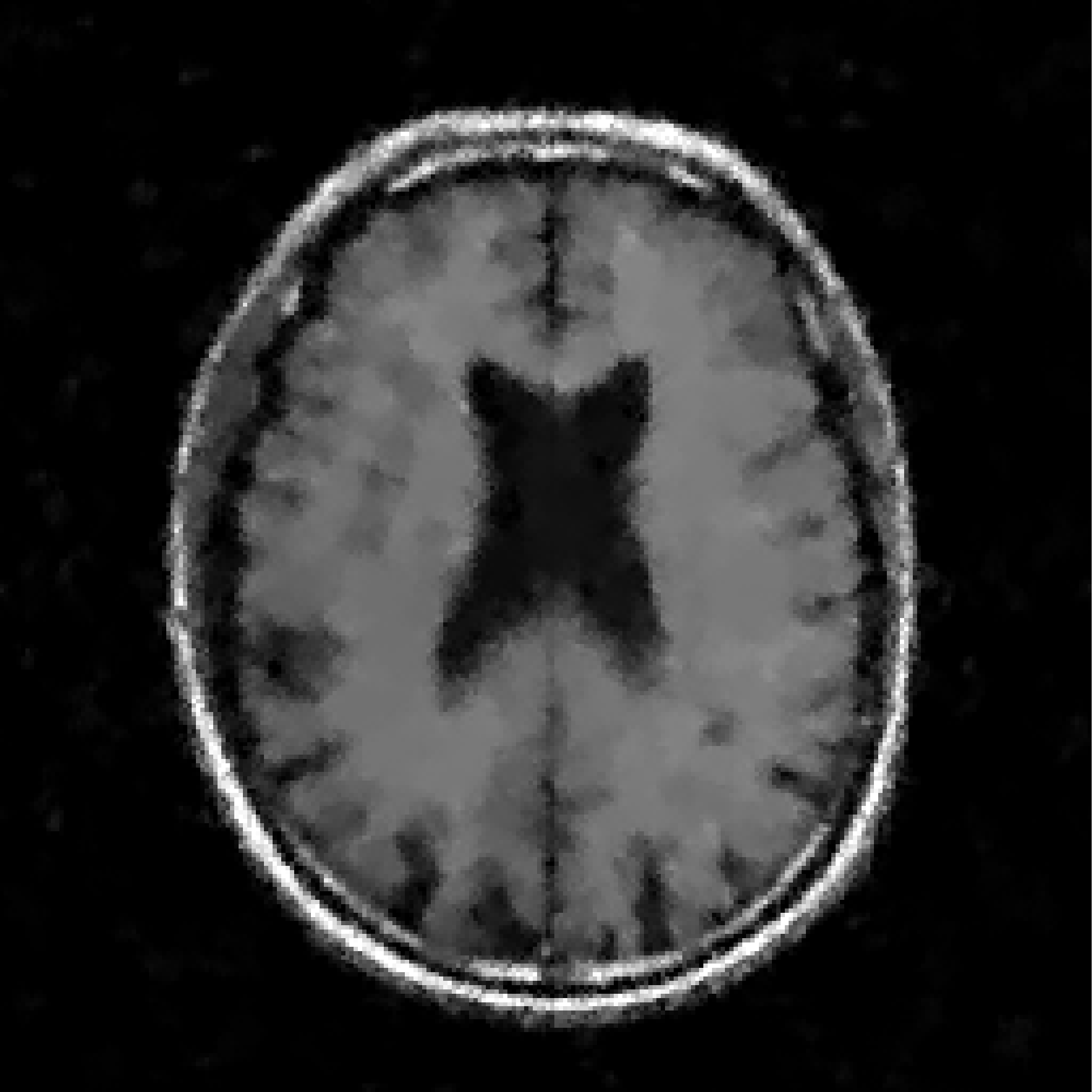}
		\end{minipage}&\hspace{-0.45cm}
\begin{minipage}{3cm}
\includegraphics[width=3cm]{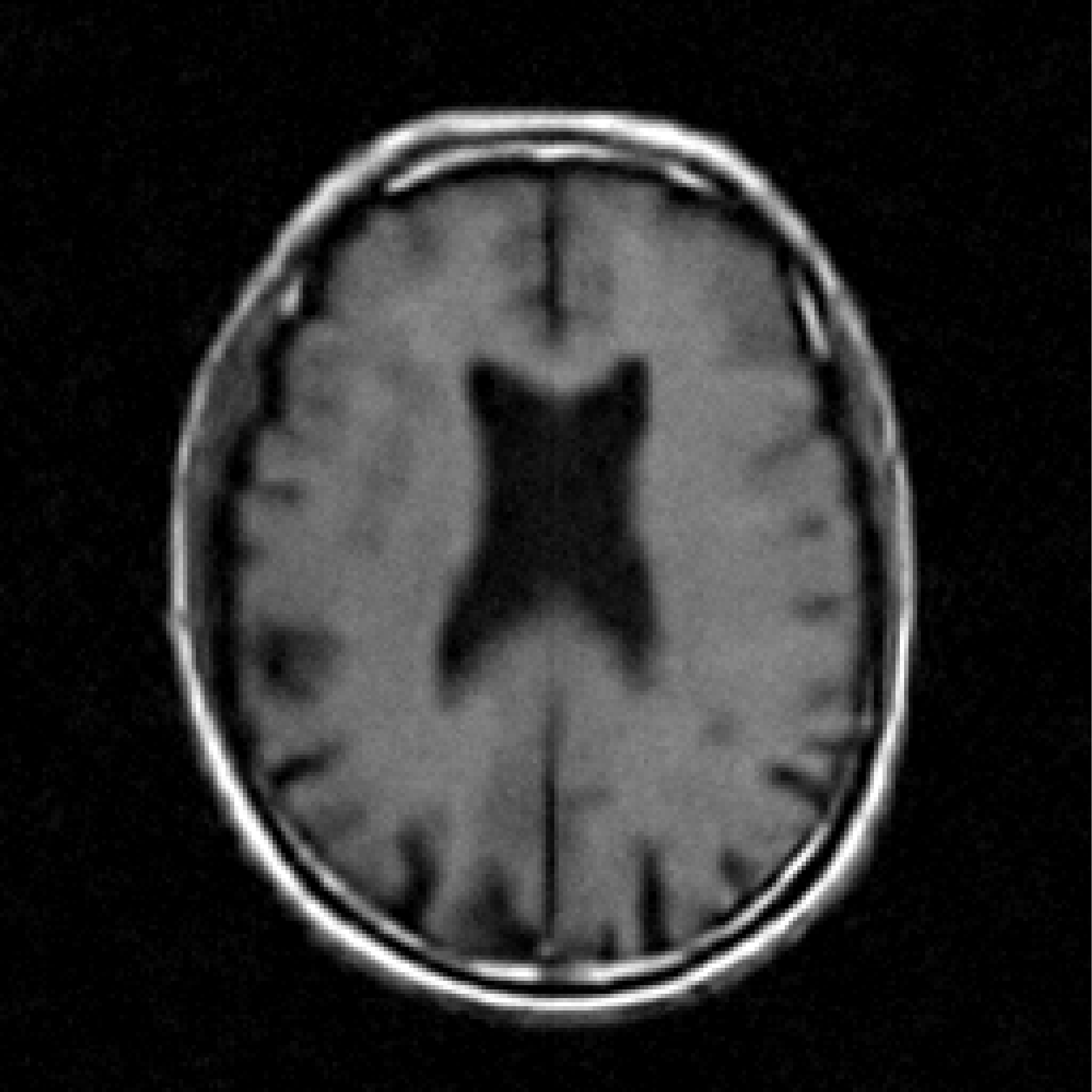}
\end{minipage}\\
		{\small{Original}}&\hspace{-0.45cm}
{\small{Initial}}&\hspace{-0.45cm}
{\small{Analysis \eqref{AnaMRI}}}&\hspace{-0.45cm}
		{\small{DDTF \eqref{DDTFMRI}}}\\
\begin{minipage}{3cm}
			\includegraphics[width=3cm]{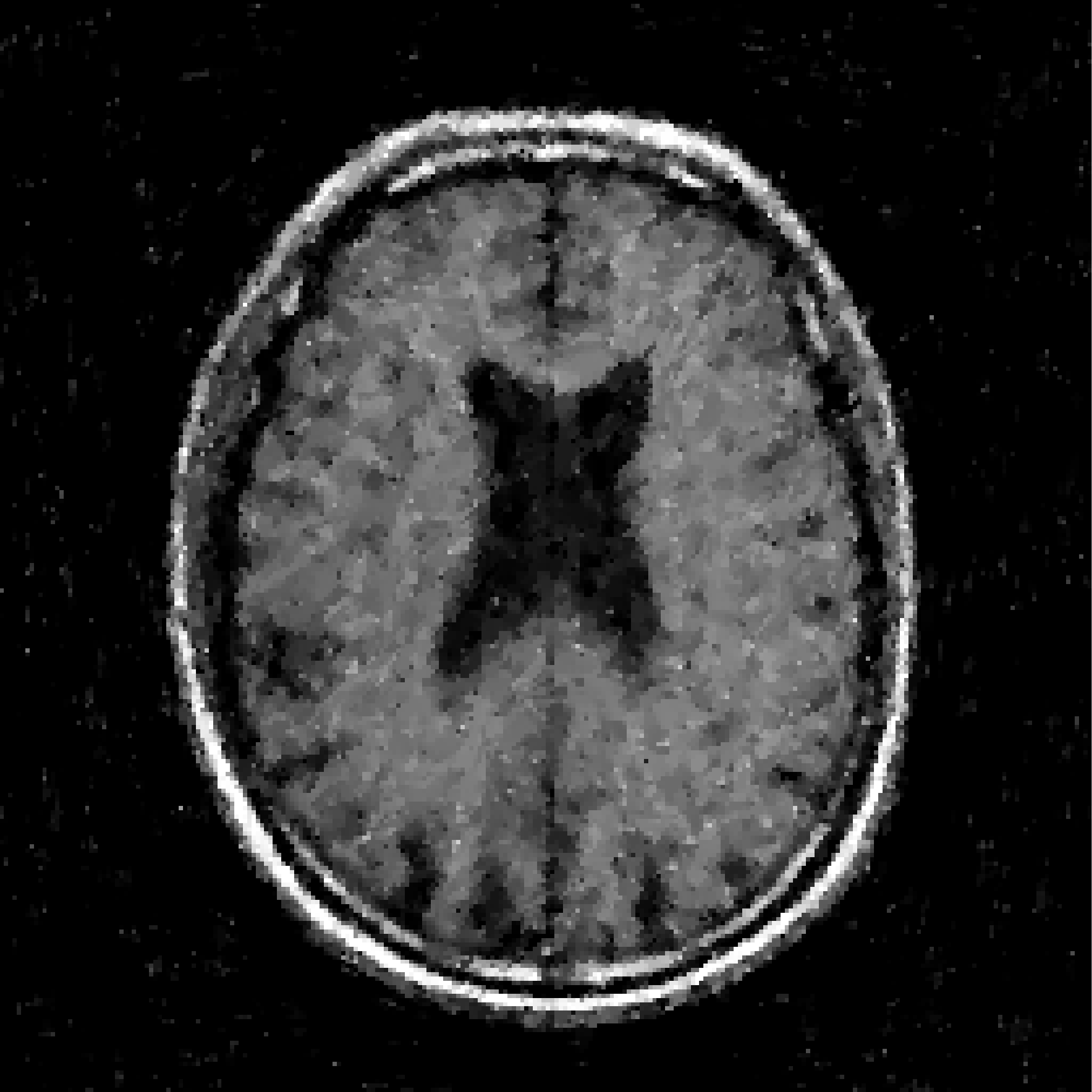}
		\end{minipage}&\hspace{-0.45cm}
		\begin{minipage}{3cm}
			\includegraphics[width=3cm]{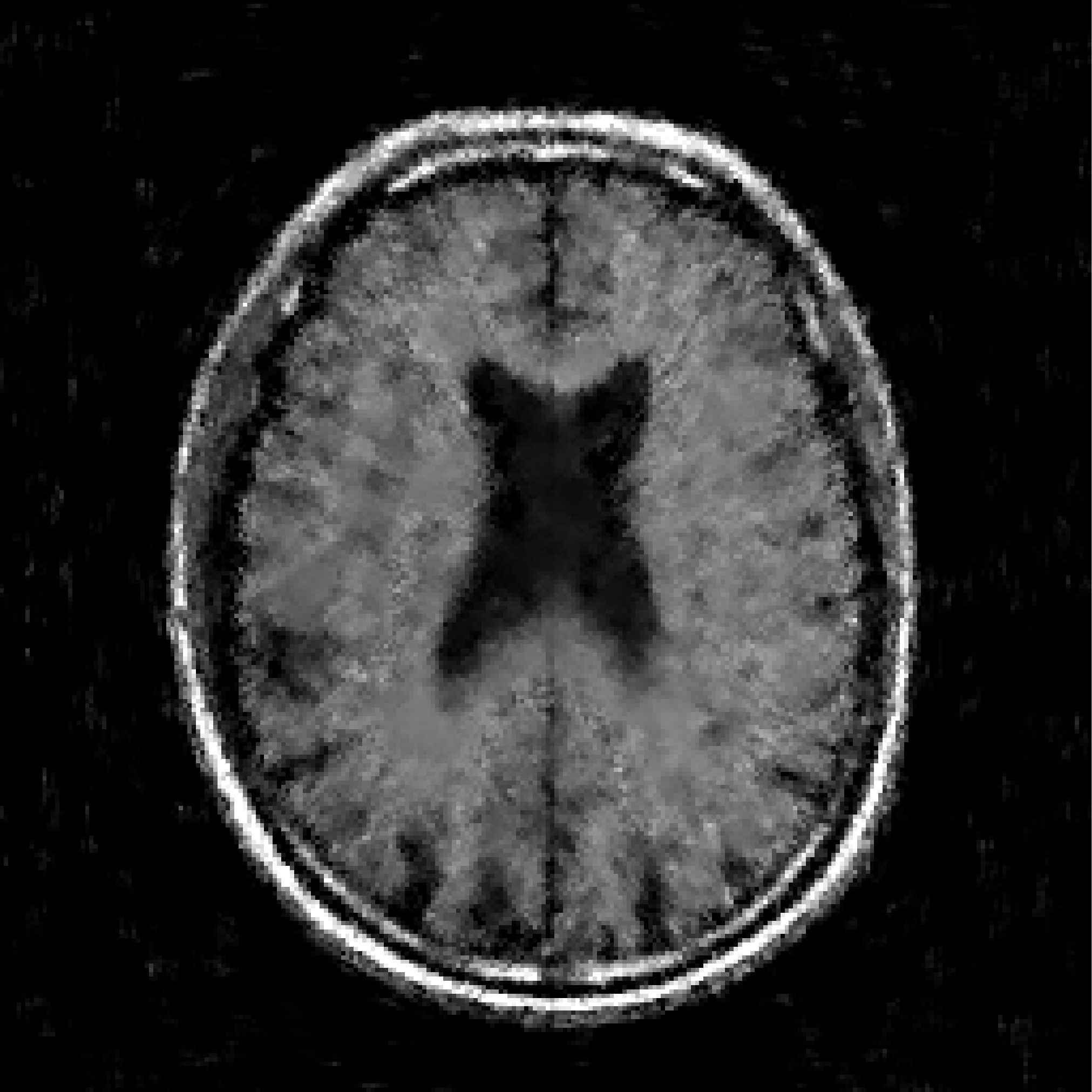}
		\end{minipage}&\hspace{-0.45cm}
		\begin{minipage}{3cm}
			\includegraphics[width=3cm]{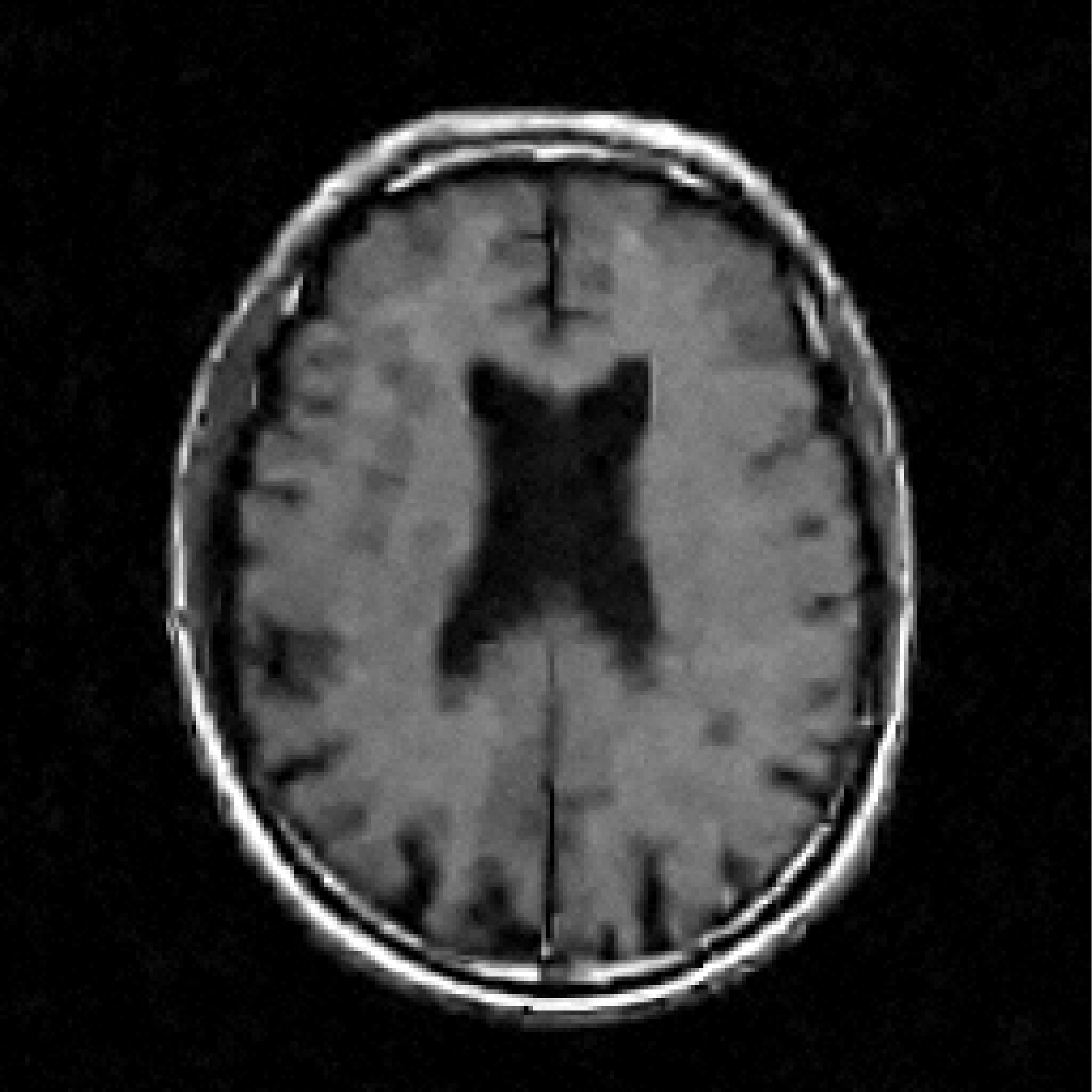}
		\end{minipage}&\hspace{-0.45cm}
		\begin{minipage}{3cm}
			\includegraphics[width=3cm]{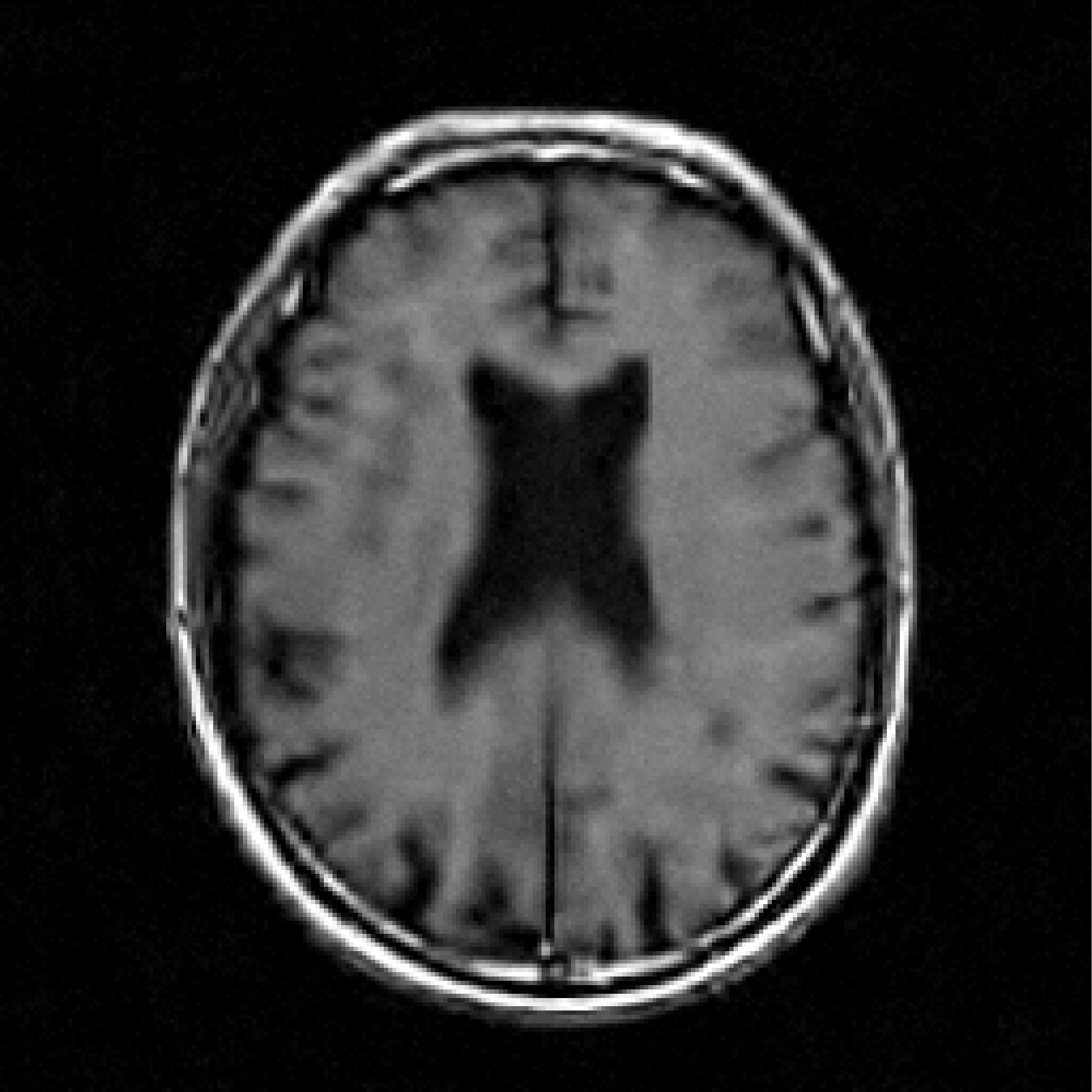}
		\end{minipage}\\
{\small{QPLS \cite{M.J.Ehrhardt2015}}}&\hspace{-0.45cm}
		{\small{JAnal \eqref{JAnalPETMRI}}}&\hspace{-0.45cm}
		{\small{JSTF \eqref{Proposed}}}&\hspace{-0.45cm}
		{\small{JSDDTF \eqref{DDTFPETMRI}}}
	\end{tabular}
	\caption{Visual comparison of PET-T1 Radial joint reconstruction results. The first and second rows describe the PET images, and the third and fourth rows depicts the MRI images.}\label{PETMRT115RadialResults}
\end{figure}

\begin{figure}[htp!]
	\centering
			\begin{tabular}{cccc}
		\begin{minipage}{3cm}
			\includegraphics[width=3cm]{PET15Original.pdf}
		\end{minipage}&\hspace{-0.45cm}
\begin{minipage}{3cm}
			\includegraphics[width=3cm]{PET15EM.pdf}
		\end{minipage}&\hspace{-0.45cm}
		\begin{minipage}{3cm}
			\includegraphics[width=3cm]{PET15Anal.pdf}
		\end{minipage}&\hspace{-0.45cm}
		\begin{minipage}{3cm}
			\includegraphics[width=3cm]{PET15DDTF.pdf}
		\end{minipage}\\
		{\small{Original}}&\hspace{-0.45cm}
{\small{Initial}}&\hspace{-0.45cm}
		{\small{Analysis \eqref{AnaPET}}}&\hspace{-0.45cm}
		{\small{DDTF \eqref{DDTFPET}}}\\
		\begin{minipage}{3cm}
			\includegraphics[width=3cm]{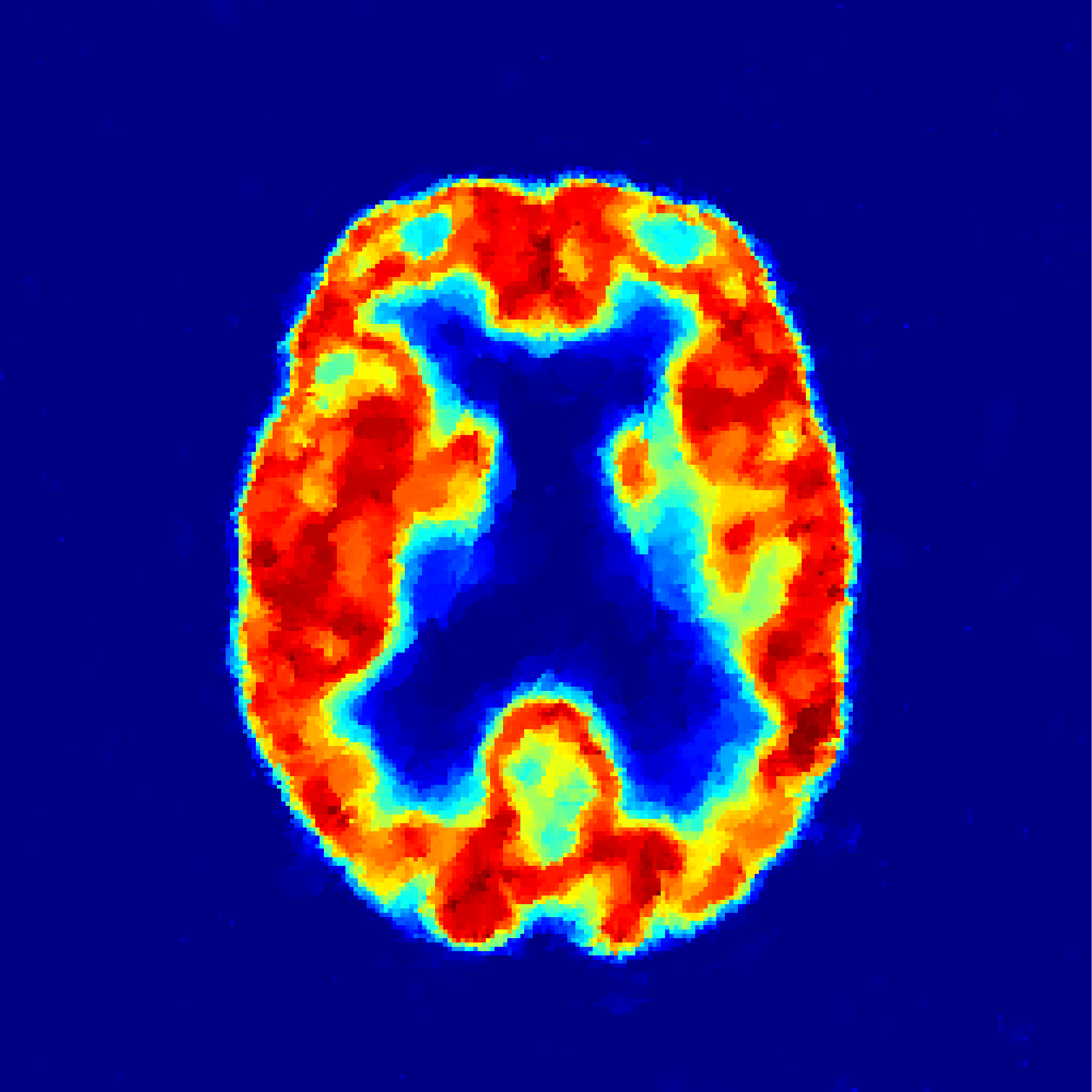}
		\end{minipage}&\hspace{-0.45cm}
		\begin{minipage}{3cm}
			\includegraphics[width=3cm]{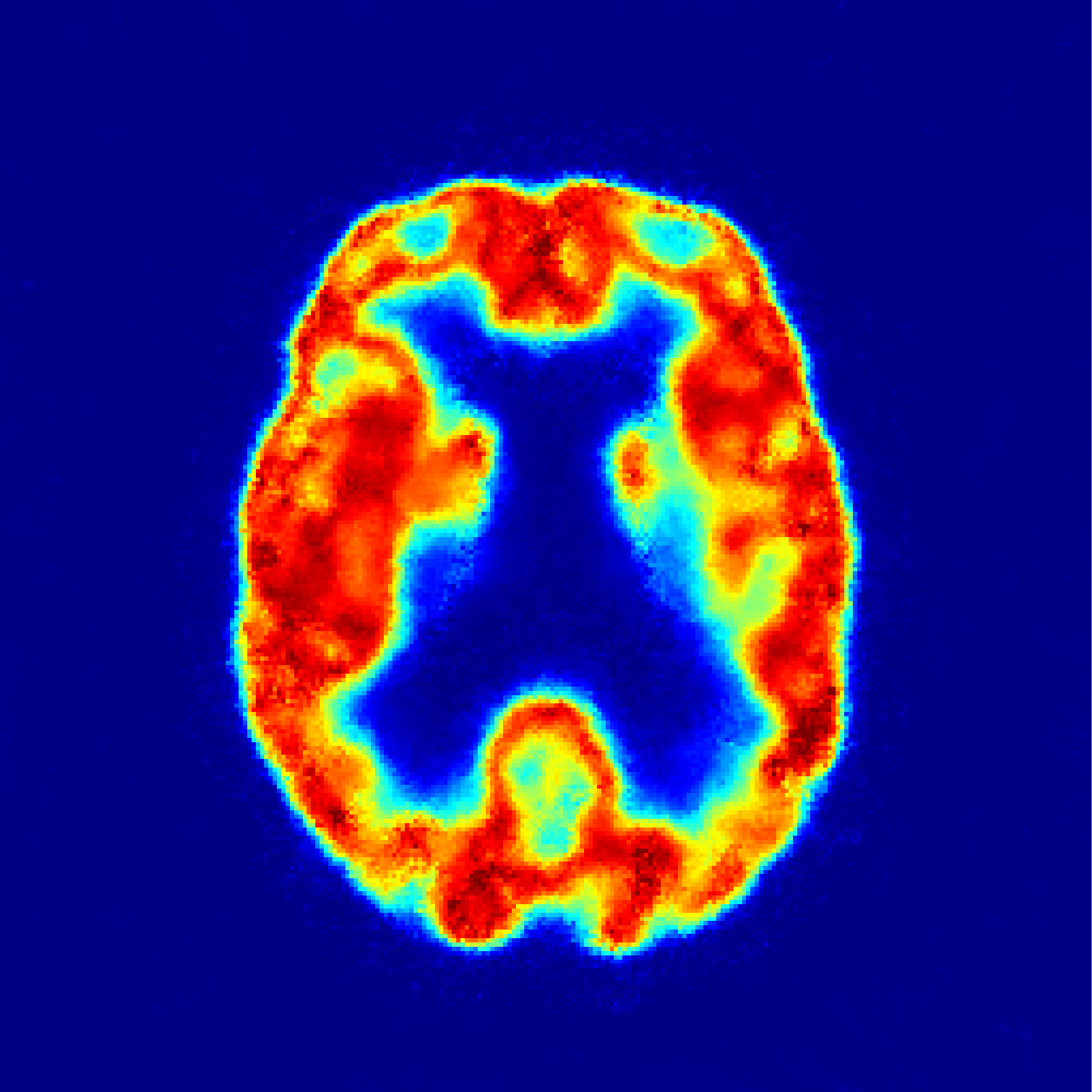}
		\end{minipage}&\hspace{-0.45cm}
		\begin{minipage}{3cm}
			\includegraphics[width=3cm]{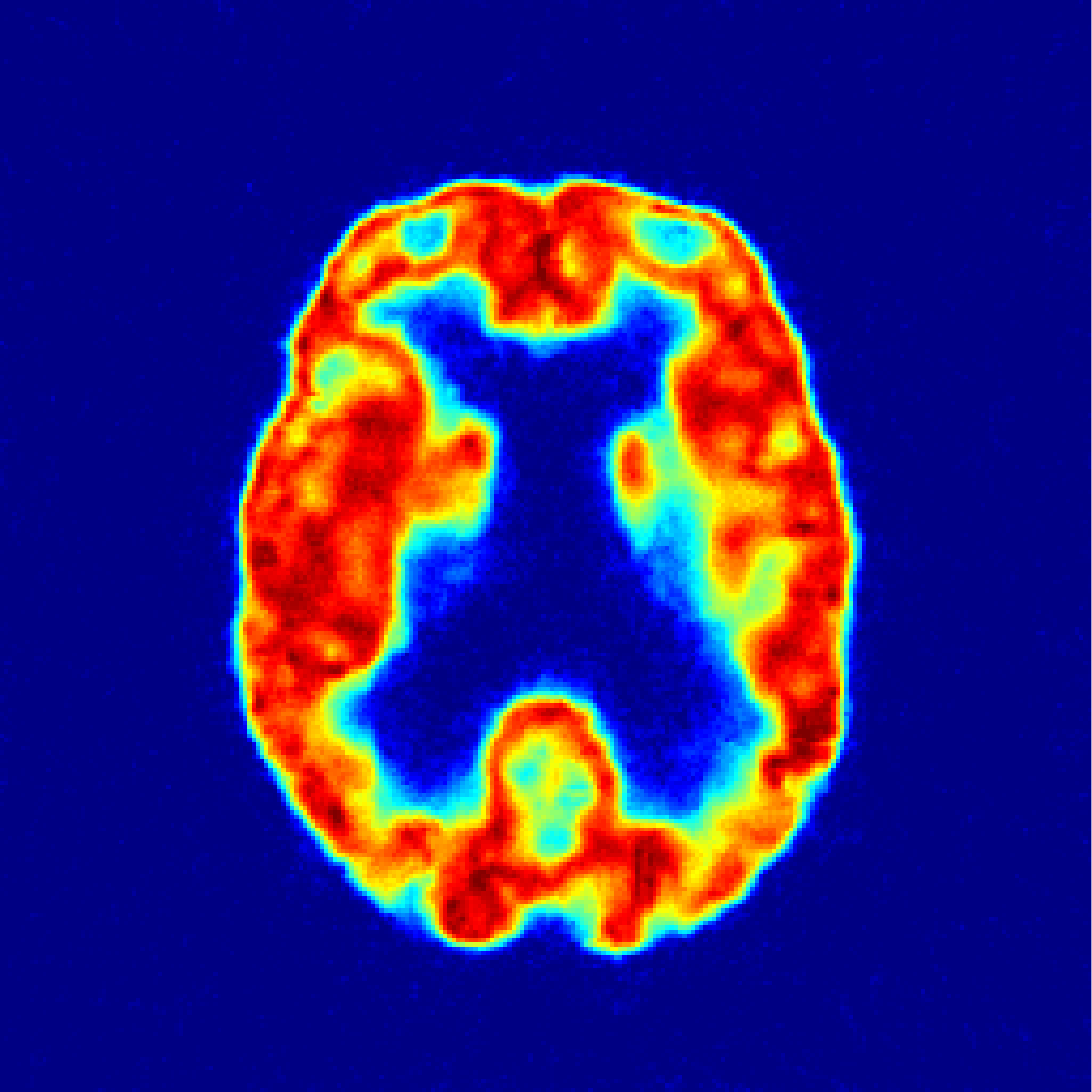}
		\end{minipage}&\hspace{-0.45cm}
		\begin{minipage}{3cm}
			\includegraphics[width=3cm]{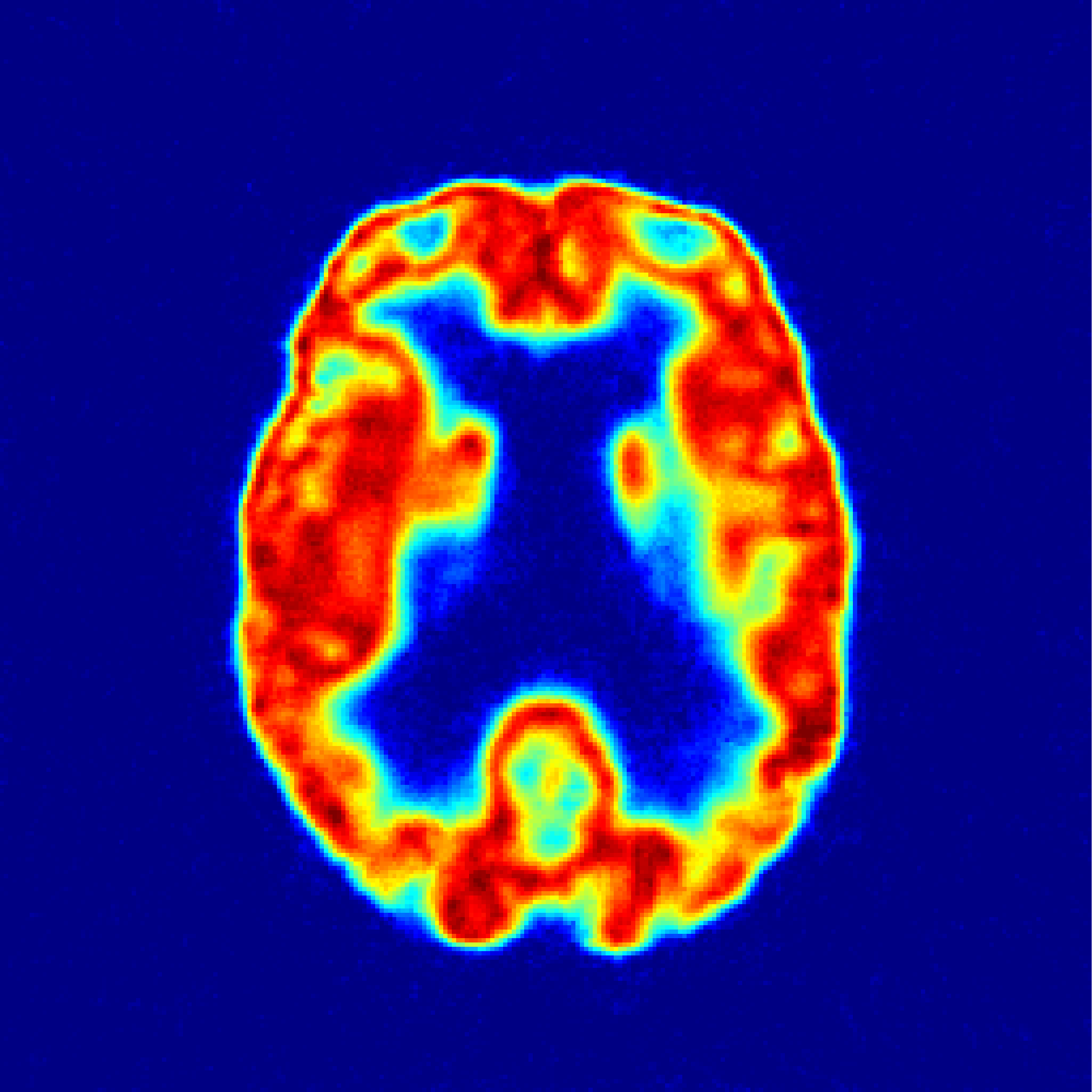}
		\end{minipage}\\
{\small{QPLS \cite{M.J.Ehrhardt2015}}}&\hspace{-0.45cm}
		{\small{JAnal \eqref{JAnalPETMRI}}}&\hspace{-0.45cm}
		{\small{JSTF \eqref{Proposed}}}&\hspace{-0.45cm}
		{\small{JSDDTF \eqref{DDTFPETMRI}}}\\
		\begin{minipage}{3cm}
			\includegraphics[width=3cm]{MRT215Original.pdf}
		\end{minipage}&\hspace{-0.45cm}
		\begin{minipage}{3cm}
			\includegraphics[width=3cm]{MRT215Radial30ZeroFill.pdf}
		\end{minipage}&\hspace{-0.45cm}
		\begin{minipage}{3cm}
			\includegraphics[width=3cm]{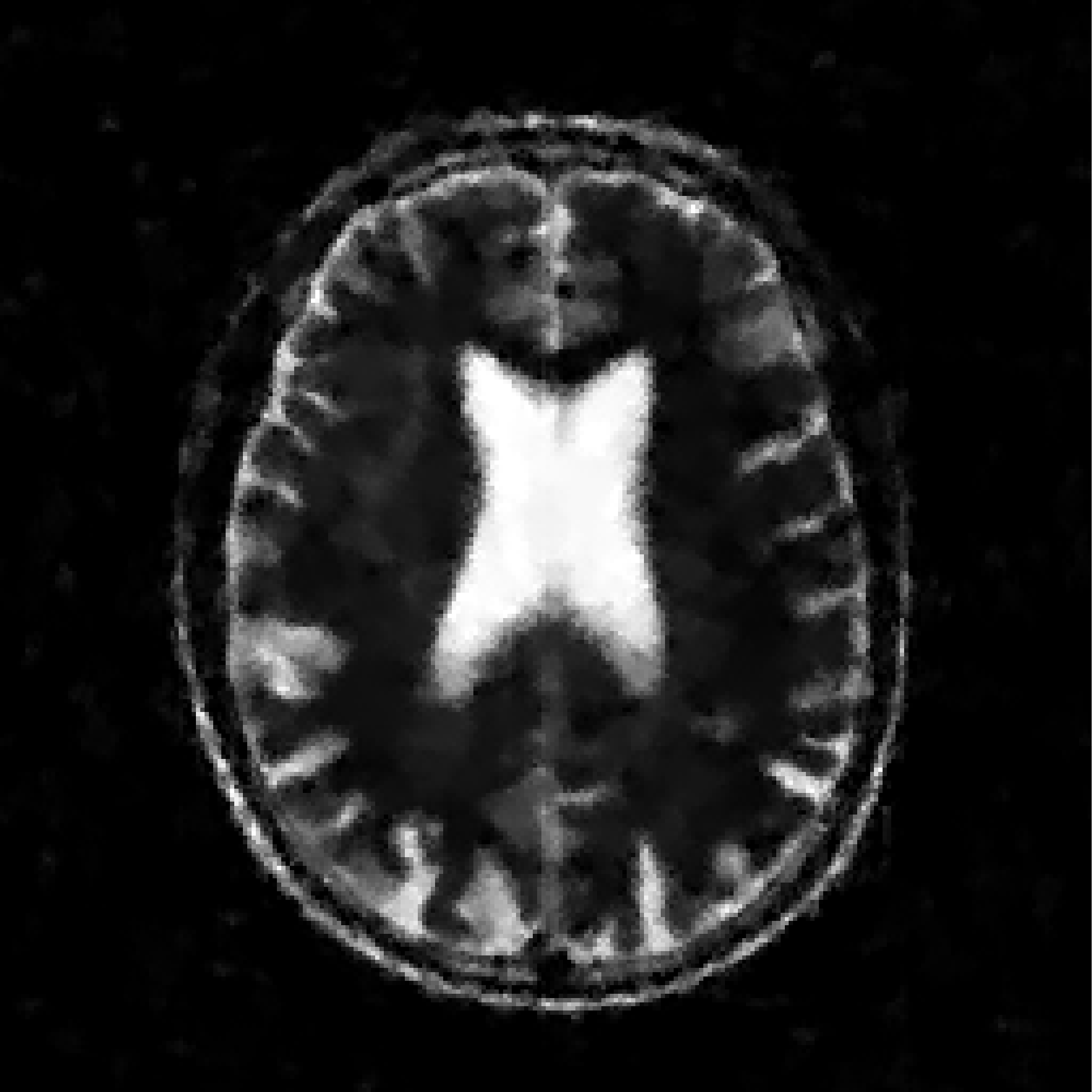}
		\end{minipage}&\hspace{-0.45cm}
\begin{minipage}{3cm}
\includegraphics[width=3cm]{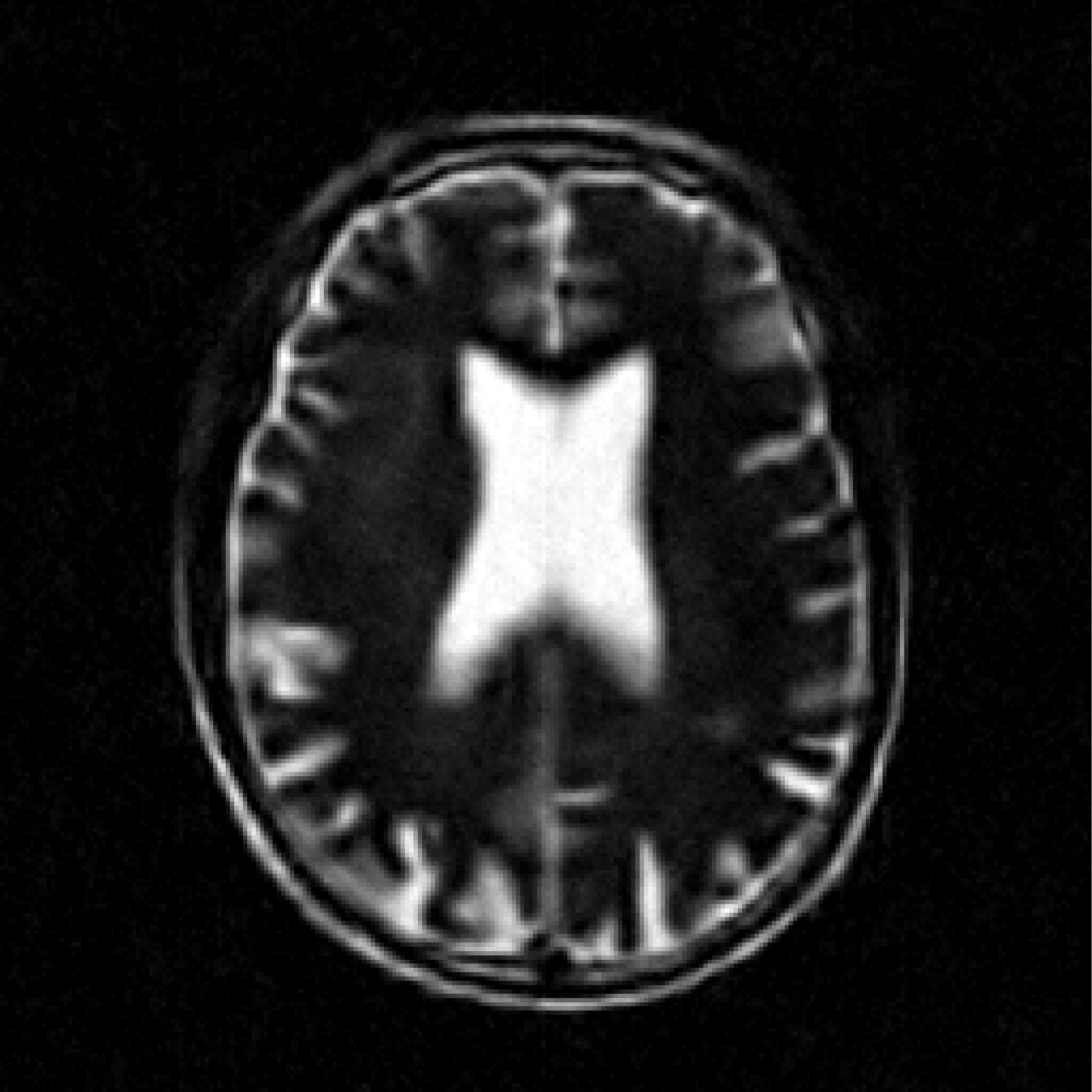}
\end{minipage}\\
		{\small{Original}}&\hspace{-0.45cm}
{\small{Initial}}&\hspace{-0.45cm}
{\small{Analysis \eqref{AnaMRI}}}&\hspace{-0.45cm}
		{\small{DDTF \eqref{DDTFMRI}}}\\
\begin{minipage}{3cm}
			\includegraphics[width=3cm]{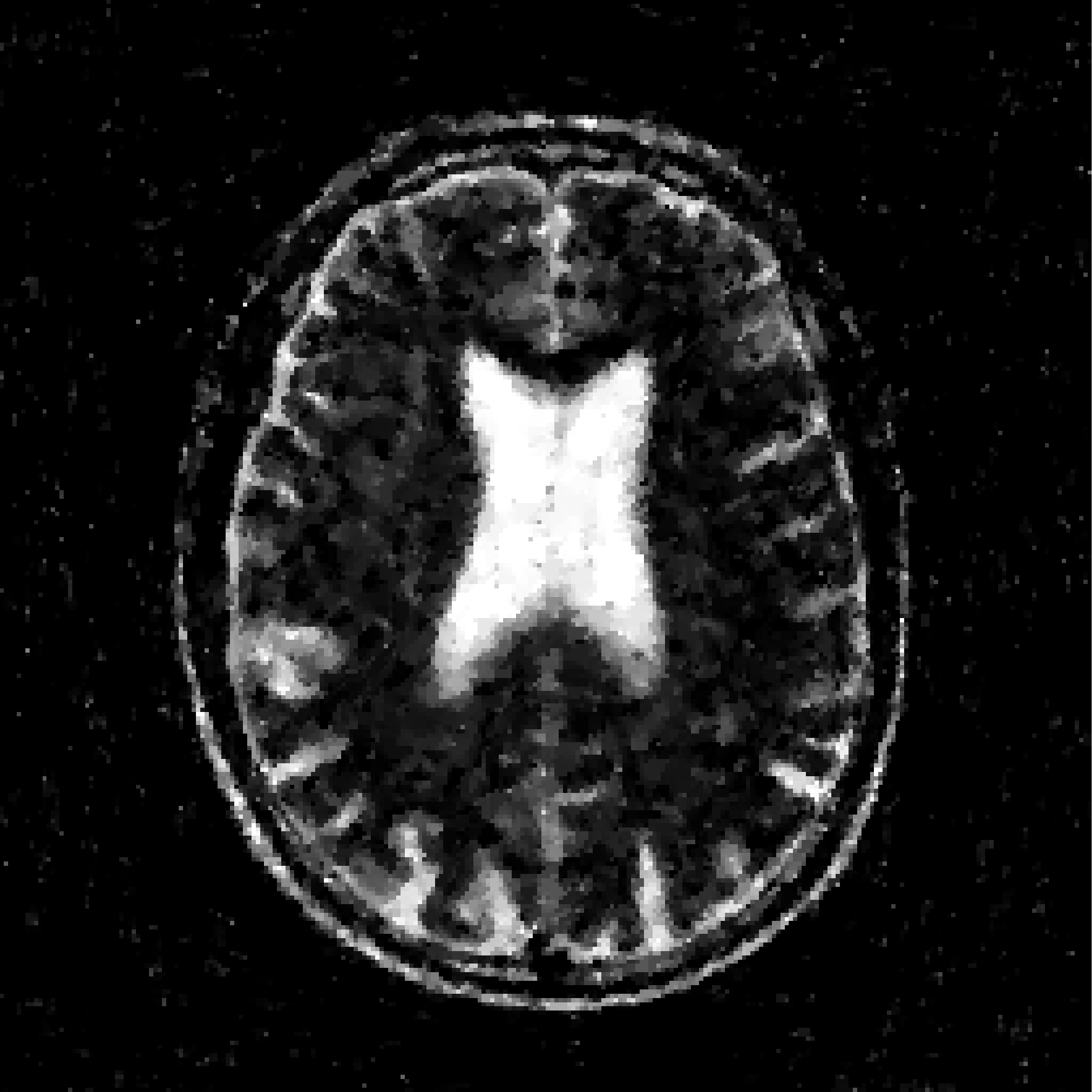}
		\end{minipage}&\hspace{-0.45cm}
		\begin{minipage}{3cm}
			\includegraphics[width=3cm]{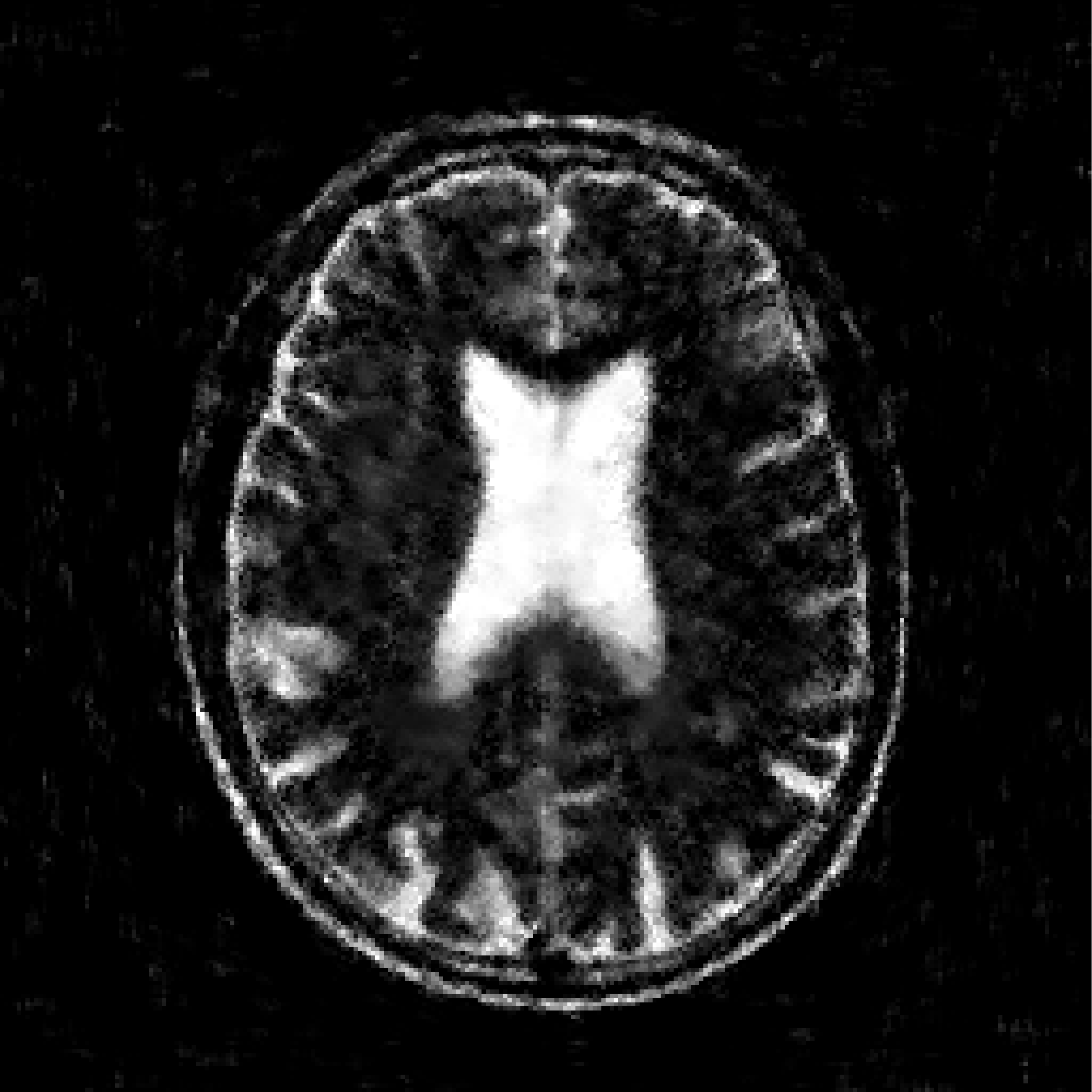}
		\end{minipage}&\hspace{-0.45cm}
		\begin{minipage}{3cm}
			\includegraphics[width=3cm]{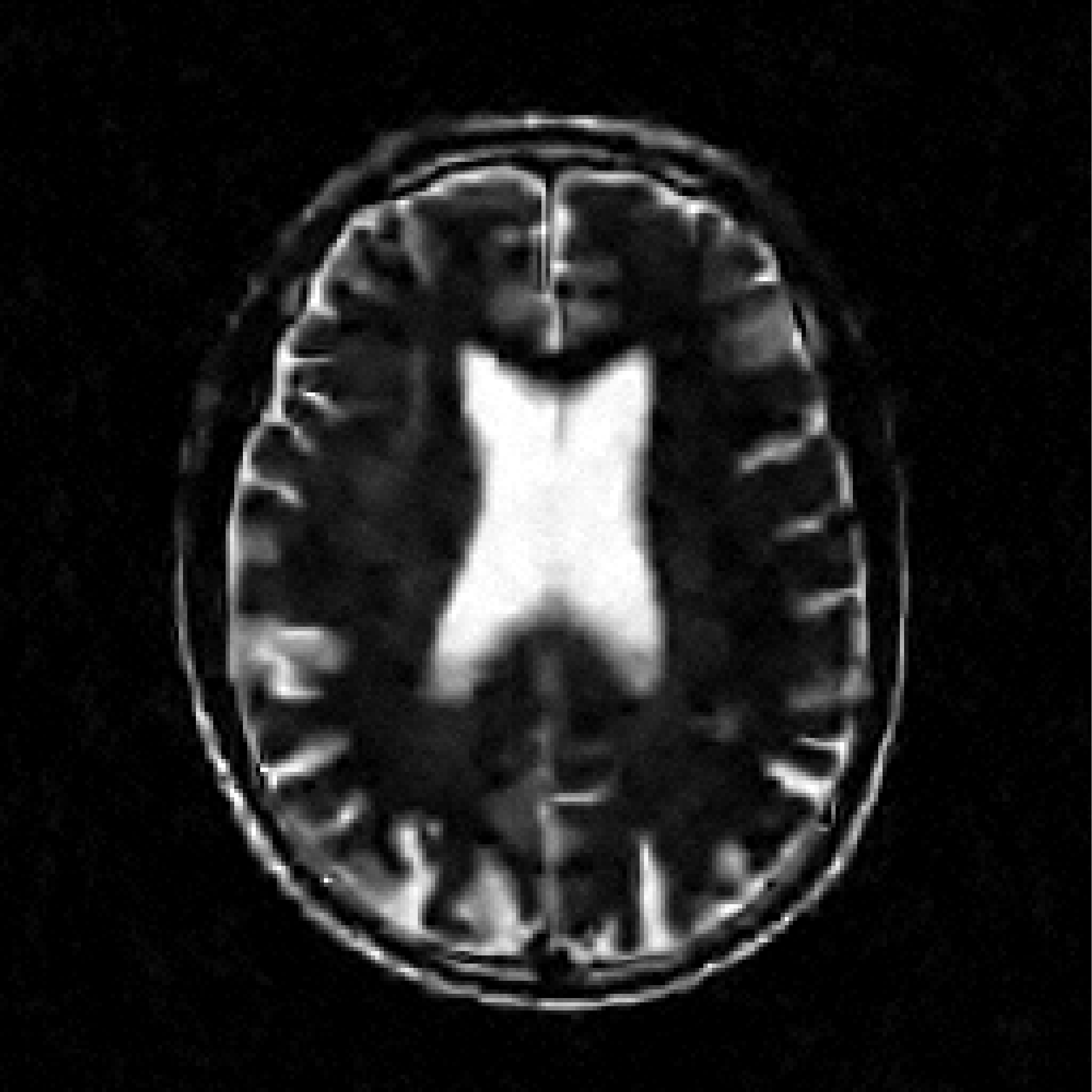}
		\end{minipage}&\hspace{-0.45cm}
		\begin{minipage}{3cm}
			\includegraphics[width=3cm]{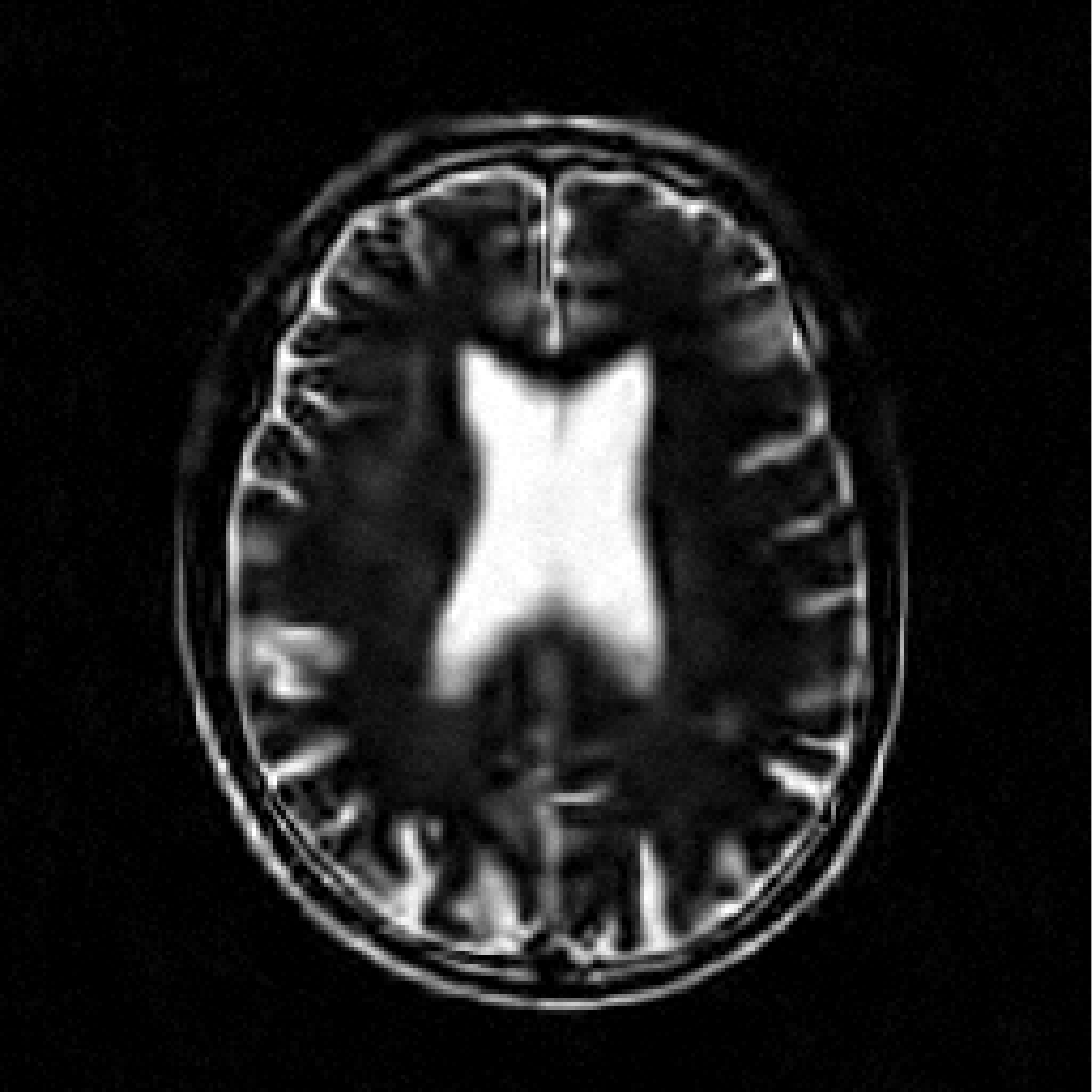}
		\end{minipage}\\
{\small{QPLS \cite{M.J.Ehrhardt2015}}}&\hspace{-0.45cm}
		{\small{JAnal \eqref{JAnalPETMRI}}}&\hspace{-0.45cm}
		{\small{JSTF \eqref{Proposed}}}&\hspace{-0.45cm}
		{\small{JSDDTF \eqref{DDTFPETMRI}}}
	\end{tabular}
	\caption{Visual comparison of PET-T2 Radial joint reconstruction results. The first and second rows describe the PET images, and the third and fourth rows depict the MRI images.}\label{PETMRT215RadialResults}
\end{figure}

\begin{figure}[htp!]
	\centering
	\begin{tabular}{cccc}
		\begin{minipage}{3cm}
			\includegraphics[width=3cm]{PET15Original.pdf}
		\end{minipage}&\hspace{-0.45cm}
\begin{minipage}{3cm}
			\includegraphics[width=3cm]{PET15EM.pdf}
		\end{minipage}&\hspace{-0.45cm}
		\begin{minipage}{3cm}
			\includegraphics[width=3cm]{PET15Anal.pdf}
		\end{minipage}&\hspace{-0.45cm}
		\begin{minipage}{3cm}
			\includegraphics[width=3cm]{PET15DDTF.pdf}
		\end{minipage}\\
		{\small{Original}}&\hspace{-0.45cm}
{\small{Initial}}&\hspace{-0.45cm}
		{\small{Analysis \eqref{AnaPET}}}&\hspace{-0.45cm}
		{\small{DDTF \eqref{DDTFPET}}}\\
		\begin{minipage}{3cm}
			\includegraphics[width=3cm]{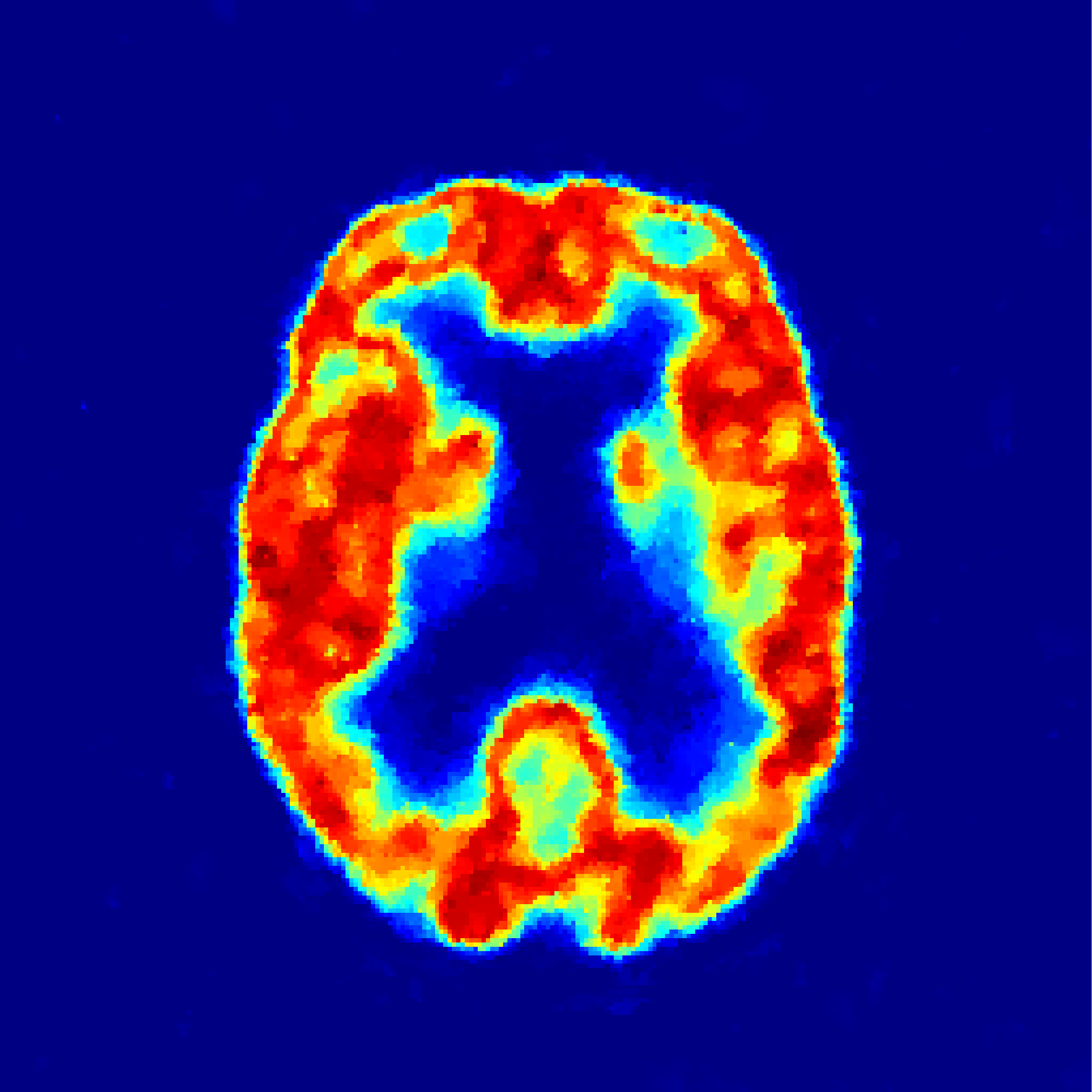}
		\end{minipage}&\hspace{-0.45cm}
		\begin{minipage}{3cm}
			\includegraphics[width=3cm]{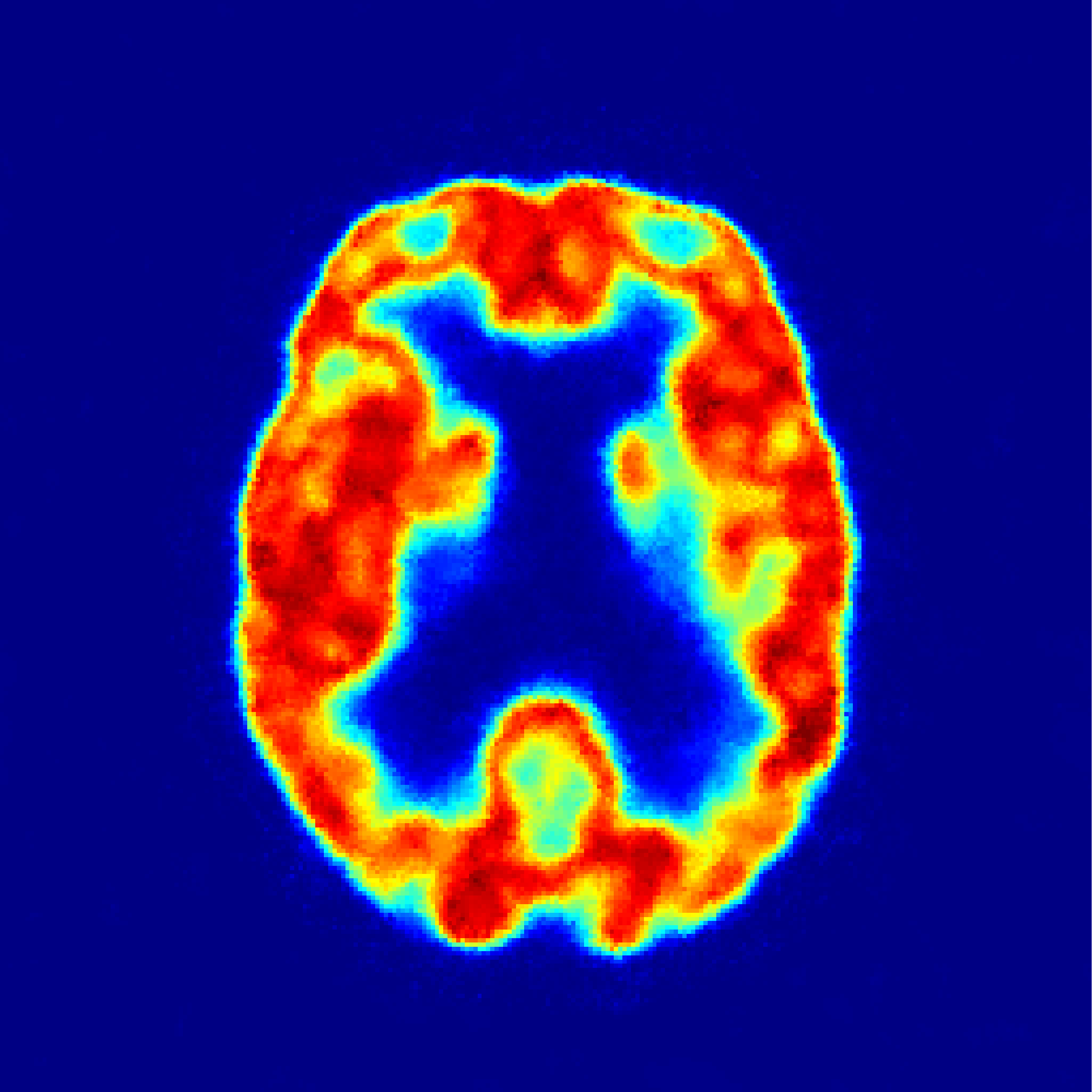}
		\end{minipage}&\hspace{-0.45cm}
		\begin{minipage}{3cm}
			\includegraphics[width=3cm]{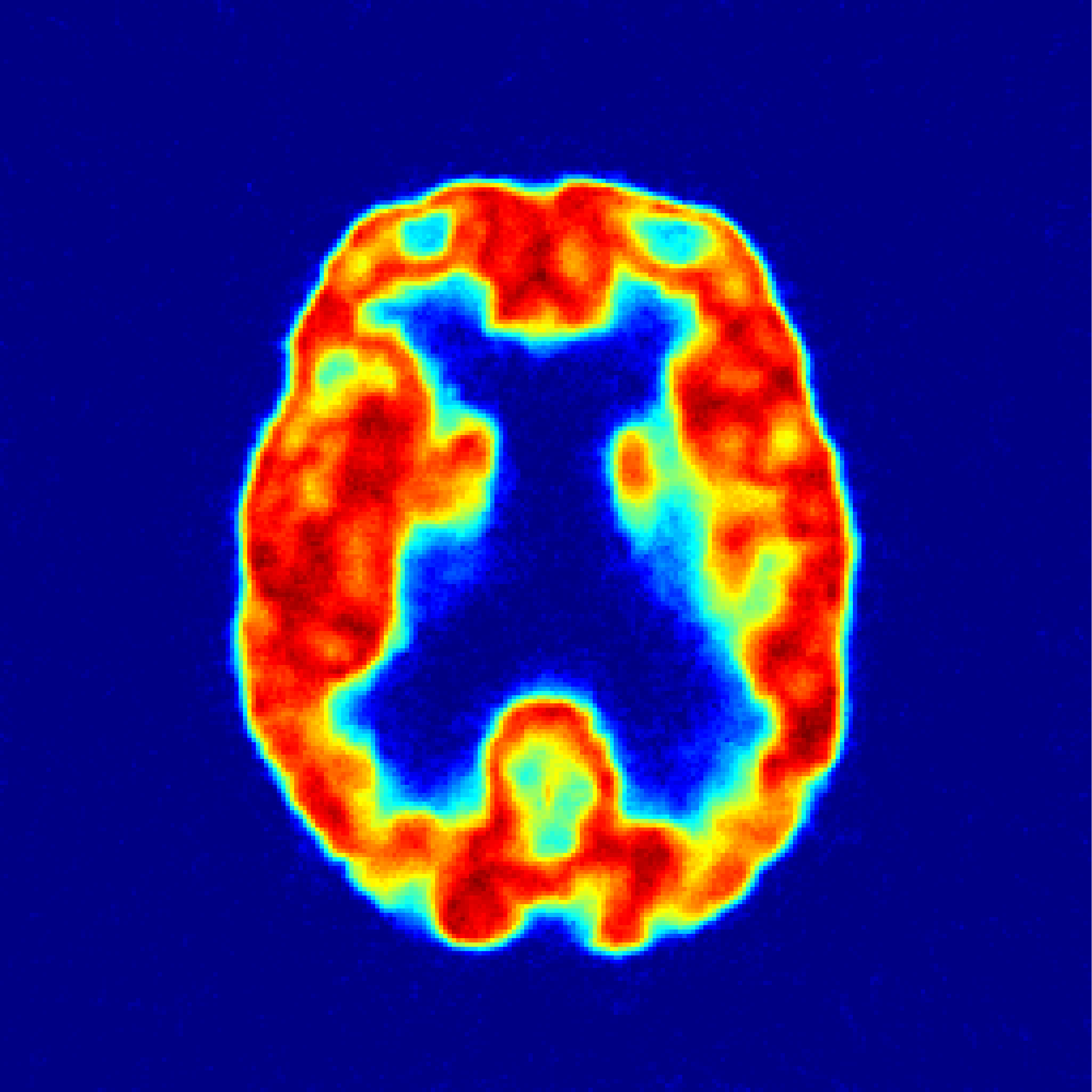}
		\end{minipage}&\hspace{-0.45cm}
		\begin{minipage}{3cm}
			\includegraphics[width=3cm]{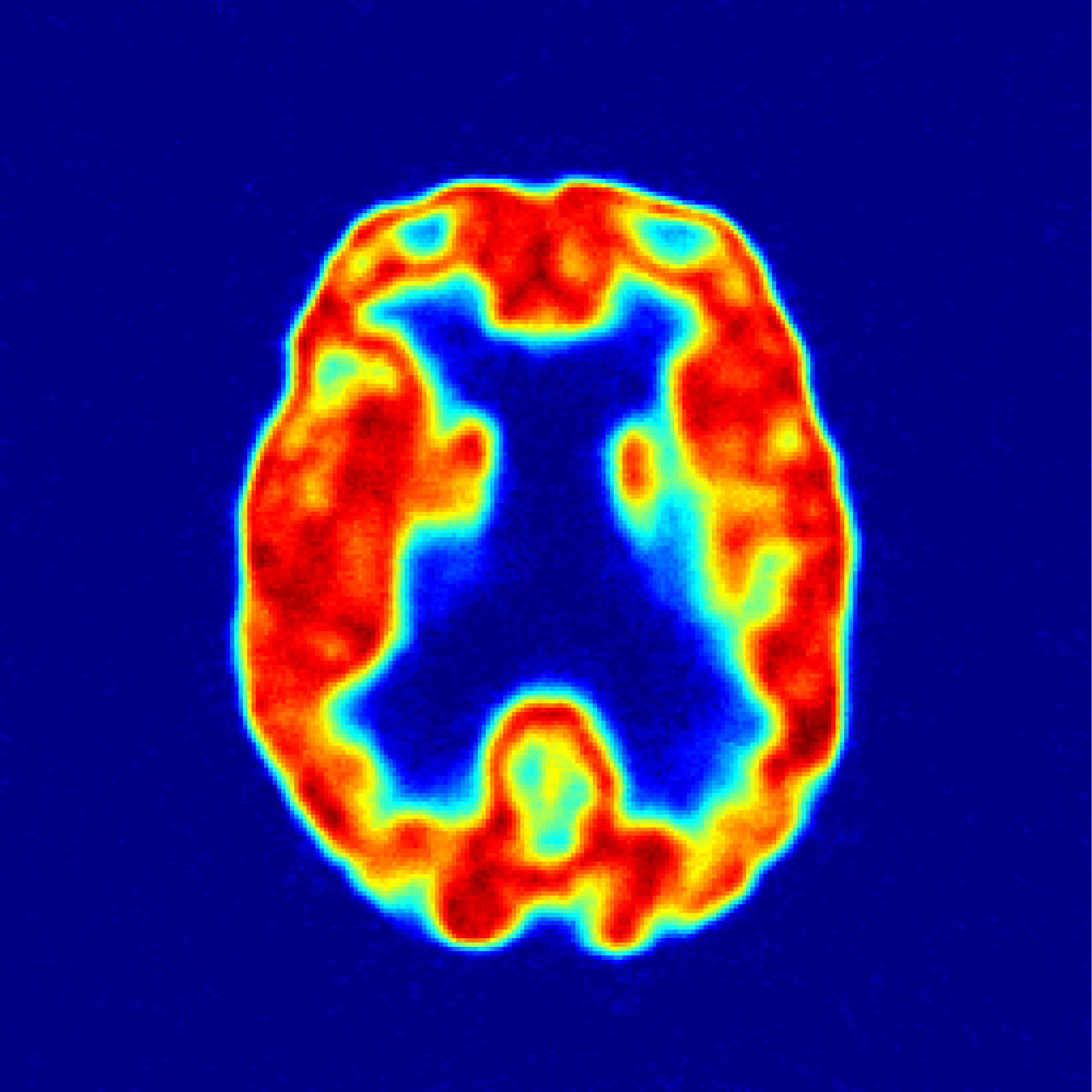}
		\end{minipage}\\
{\small{QPLS \cite{M.J.Ehrhardt2015}}}&\hspace{-0.45cm}
		{\small{JAnal \eqref{JAnalPETMRI}}}&\hspace{-0.45cm}
		{\small{JSTF \eqref{Proposed}}}&\hspace{-0.45cm}
		{\small{JSDDTF \eqref{DDTFPETMRI}}}\\
		\begin{minipage}{3cm}
			\includegraphics[width=3cm]{MRPD15Original.pdf}
		\end{minipage}&\hspace{-0.45cm}
		\begin{minipage}{3cm}
			\includegraphics[width=3cm]{MRPD15RandomZeroFill.pdf}
		\end{minipage}&\hspace{-0.45cm}
		\begin{minipage}{3cm}
			\includegraphics[width=3cm]{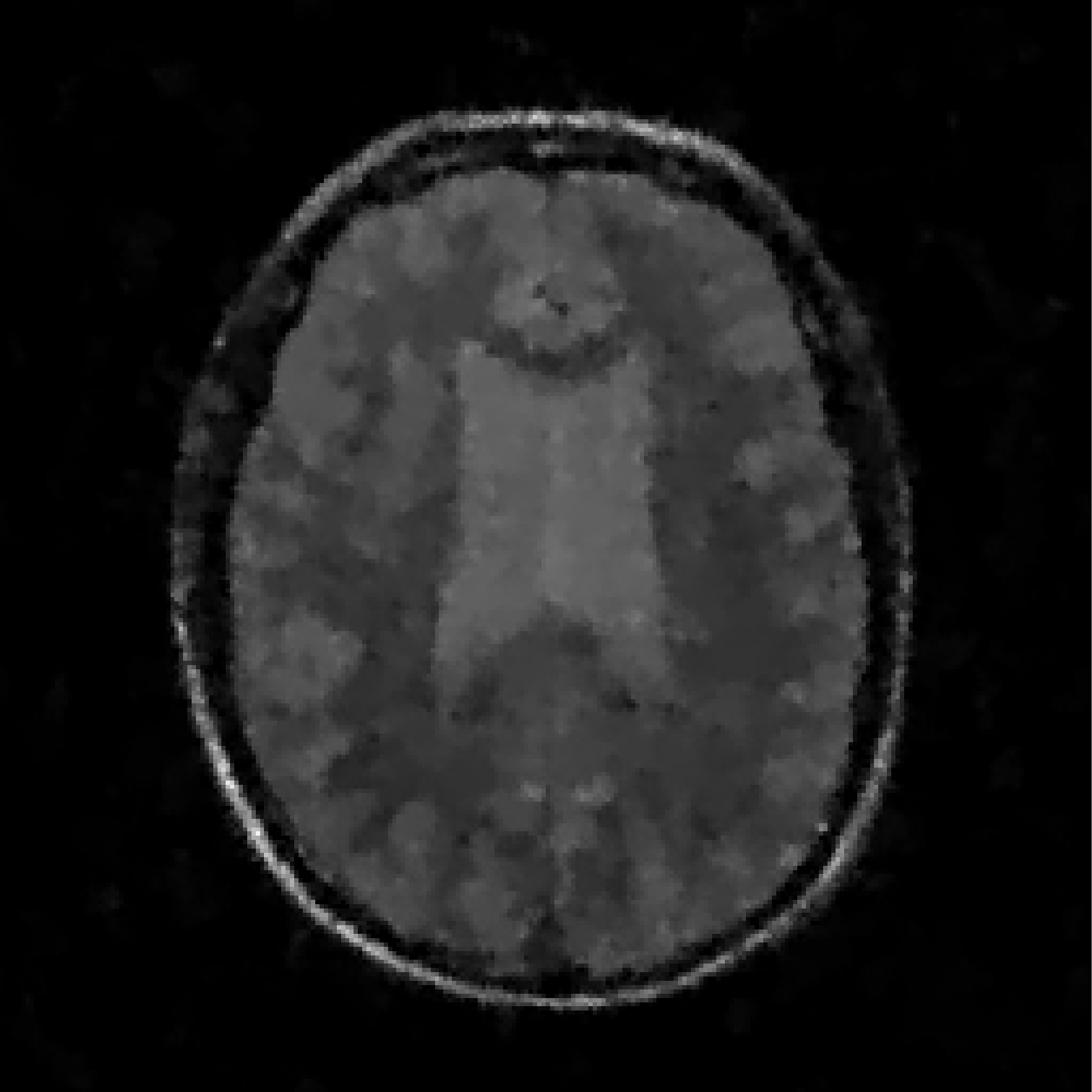}
		\end{minipage}&\hspace{-0.45cm}
\begin{minipage}{3cm}
\includegraphics[width=3cm]{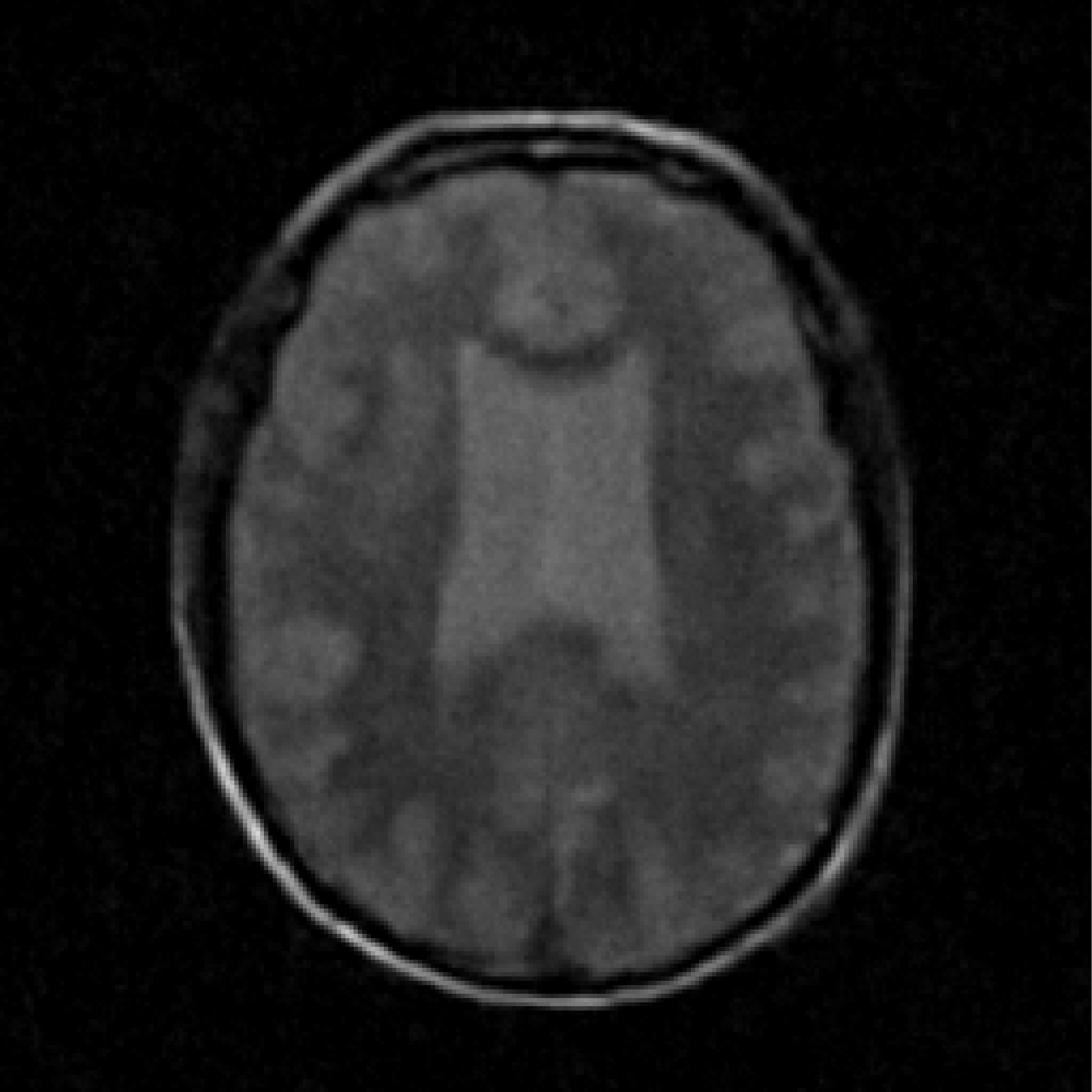}
\end{minipage}\\
		{\small{Original}}&\hspace{-0.45cm}
{\small{Initial}}&\hspace{-0.45cm}
{\small{Analysis \eqref{AnaMRI}}}&\hspace{-0.45cm}
		{\small{DDTF \eqref{DDTFMRI}}}\\
\begin{minipage}{3cm}
			\includegraphics[width=3cm]{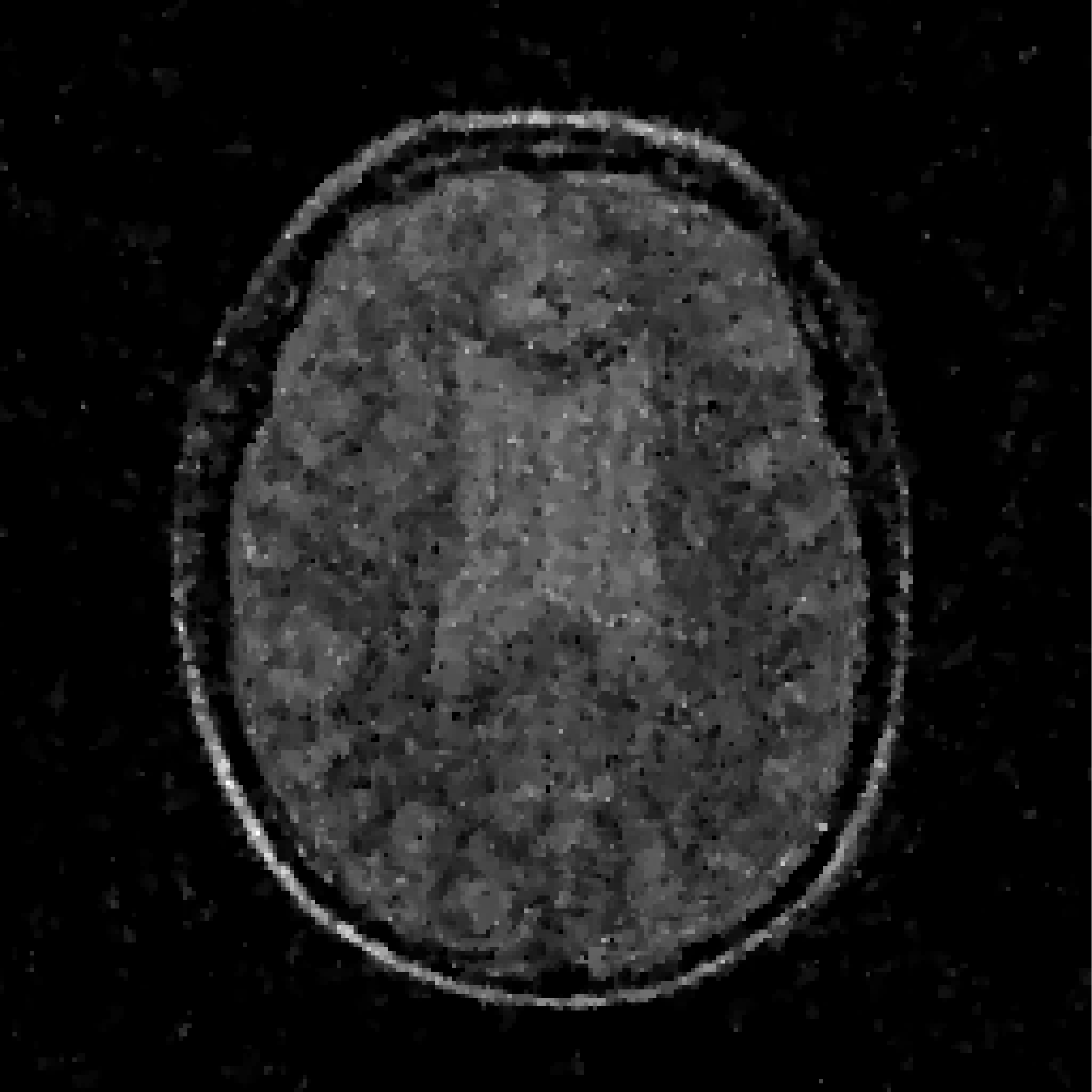}
		\end{minipage}&\hspace{-0.45cm}
		\begin{minipage}{3cm}
			\includegraphics[width=3cm]{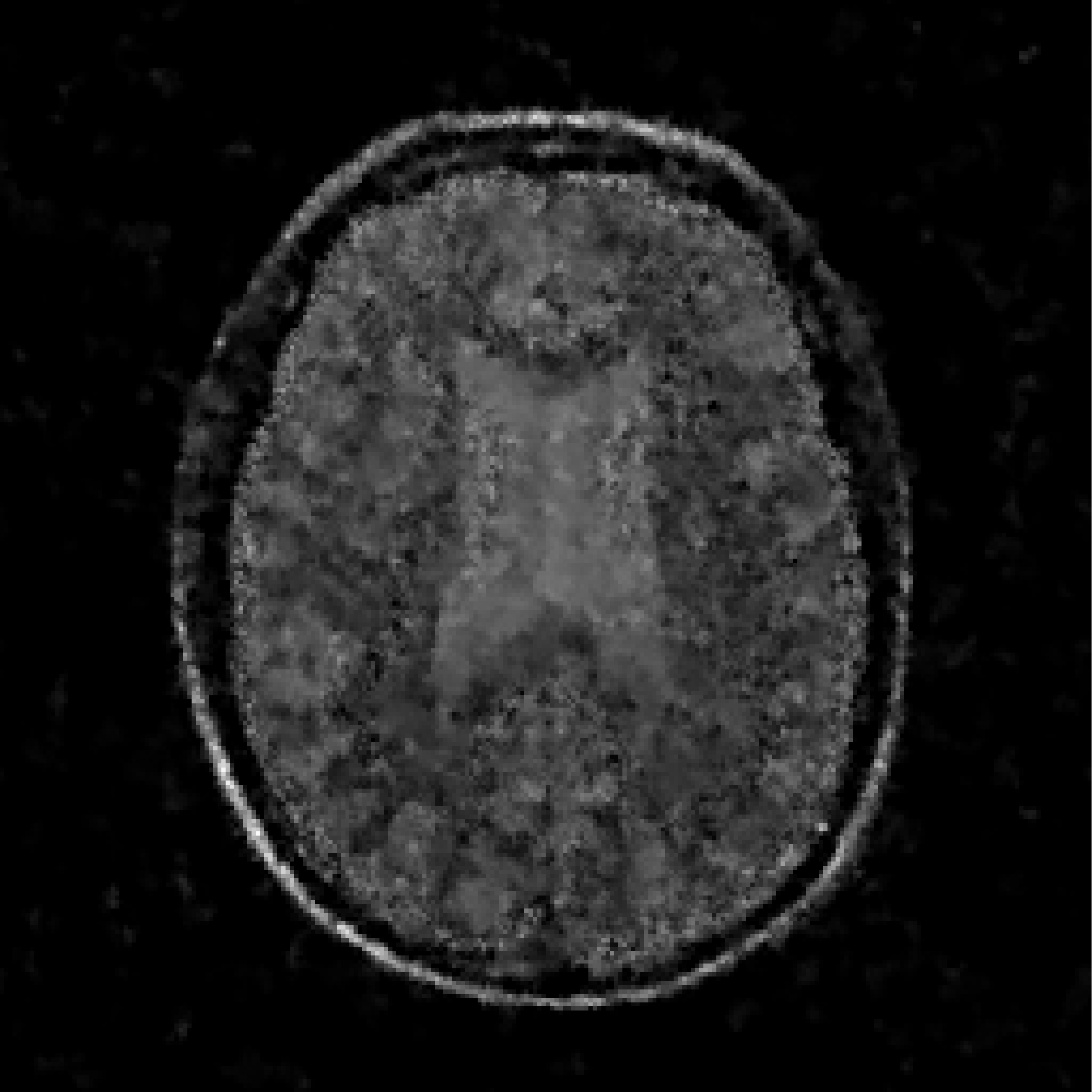}
		\end{minipage}&\hspace{-0.45cm}
		\begin{minipage}{3cm}
			\includegraphics[width=3cm]{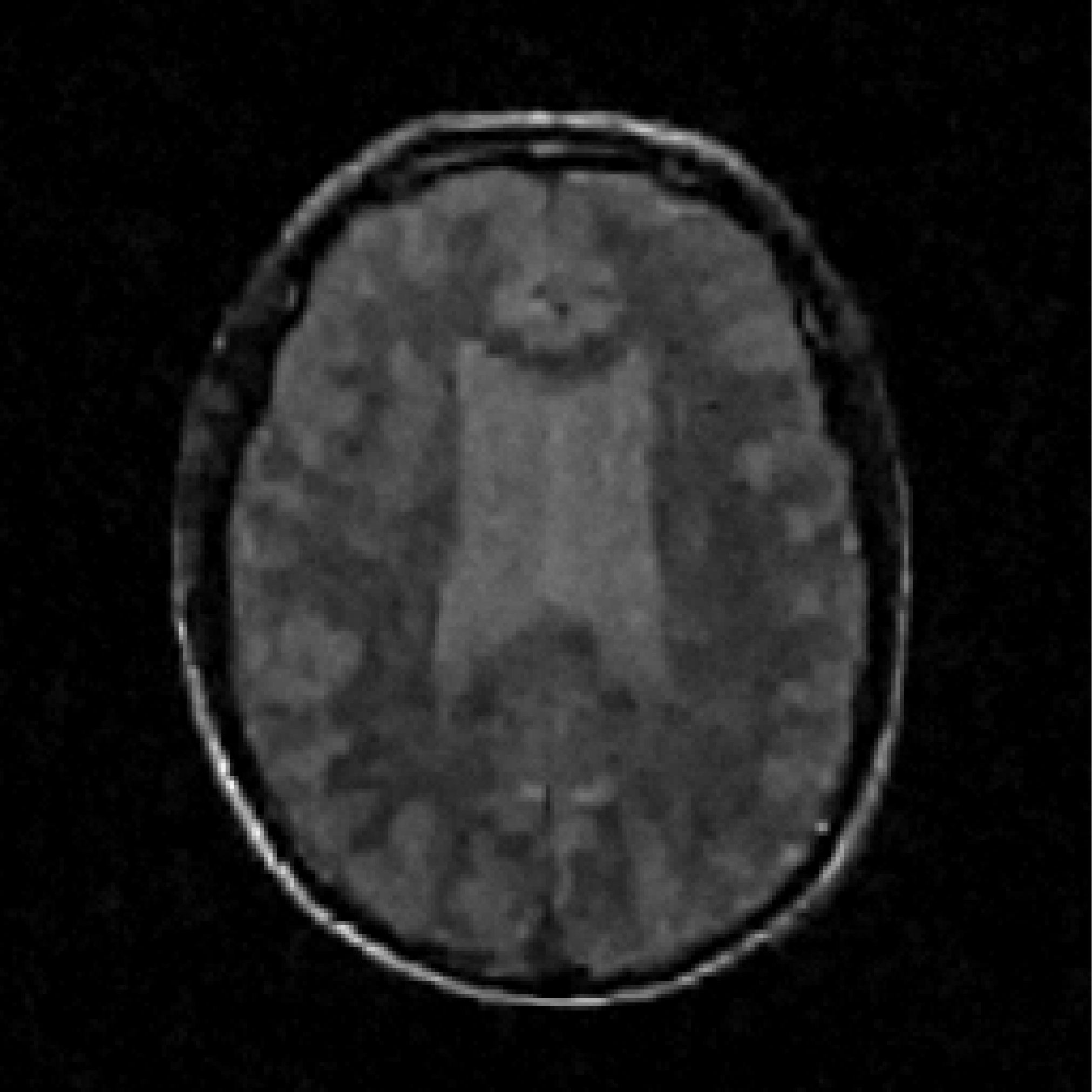}
		\end{minipage}&\hspace{-0.45cm}
		\begin{minipage}{3cm}
			\includegraphics[width=3cm]{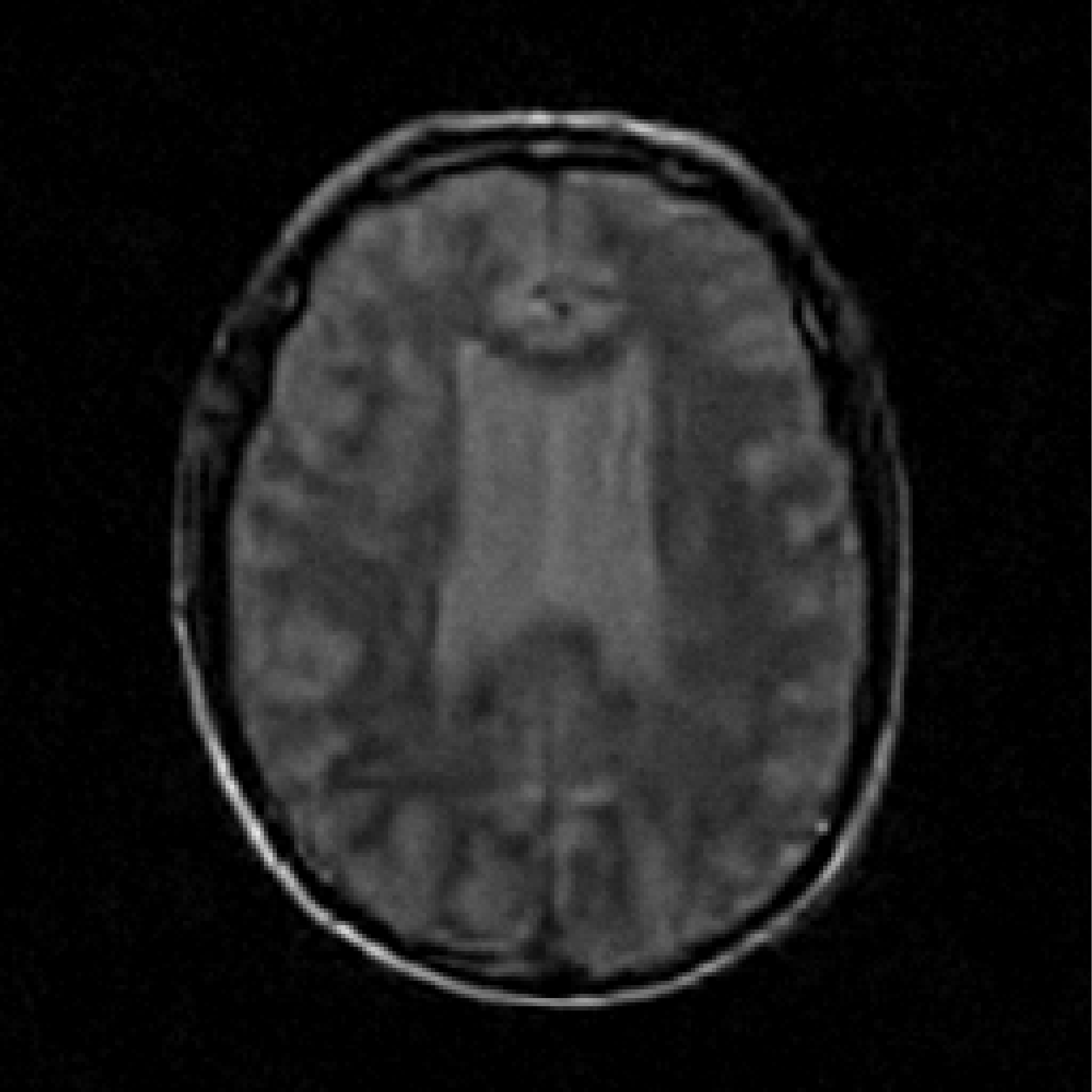}
		\end{minipage}\\
{\small{QPLS \cite{M.J.Ehrhardt2015}}}&\hspace{-0.45cm}
		{\small{JAnal \eqref{JAnalPETMRI}}}&\hspace{-0.45cm}
		{\small{JSTF \eqref{Proposed}}}&\hspace{-0.45cm}
		{\small{JSDDTF \eqref{DDTFPETMRI}}}
	\end{tabular}
	\caption{Visual comparison of PET-PD Random joint reconstruction results. The first and second rows describe the PET images, and the third and fourth rows depict the MRI images.}\label{PETMRPD15RandomResults}
\end{figure}

\begin{figure}[htp!]
	\centering
	\begin{tabular}{cccc}
		\begin{minipage}{3cm}
			\includegraphics[width=3cm]{PET15Original.pdf}
		\end{minipage}&\hspace{-0.45cm}
\begin{minipage}{3cm}
			\includegraphics[width=3cm]{PET15EM.pdf}
		\end{minipage}&\hspace{-0.45cm}
		\begin{minipage}{3cm}
			\includegraphics[width=3cm]{PET15Anal.pdf}
		\end{minipage}&\hspace{-0.45cm}
		\begin{minipage}{3cm}
			\includegraphics[width=3cm]{PET15DDTF.pdf}
		\end{minipage}\\
		{\small{Original}}&\hspace{-0.45cm}
{\small{Initial}}&\hspace{-0.45cm}
		{\small{Analysis \eqref{AnaPET}}}&\hspace{-0.45cm}
		{\small{DDTF \eqref{DDTFPET}}}\\
		\begin{minipage}{3cm}
			\includegraphics[width=3cm]{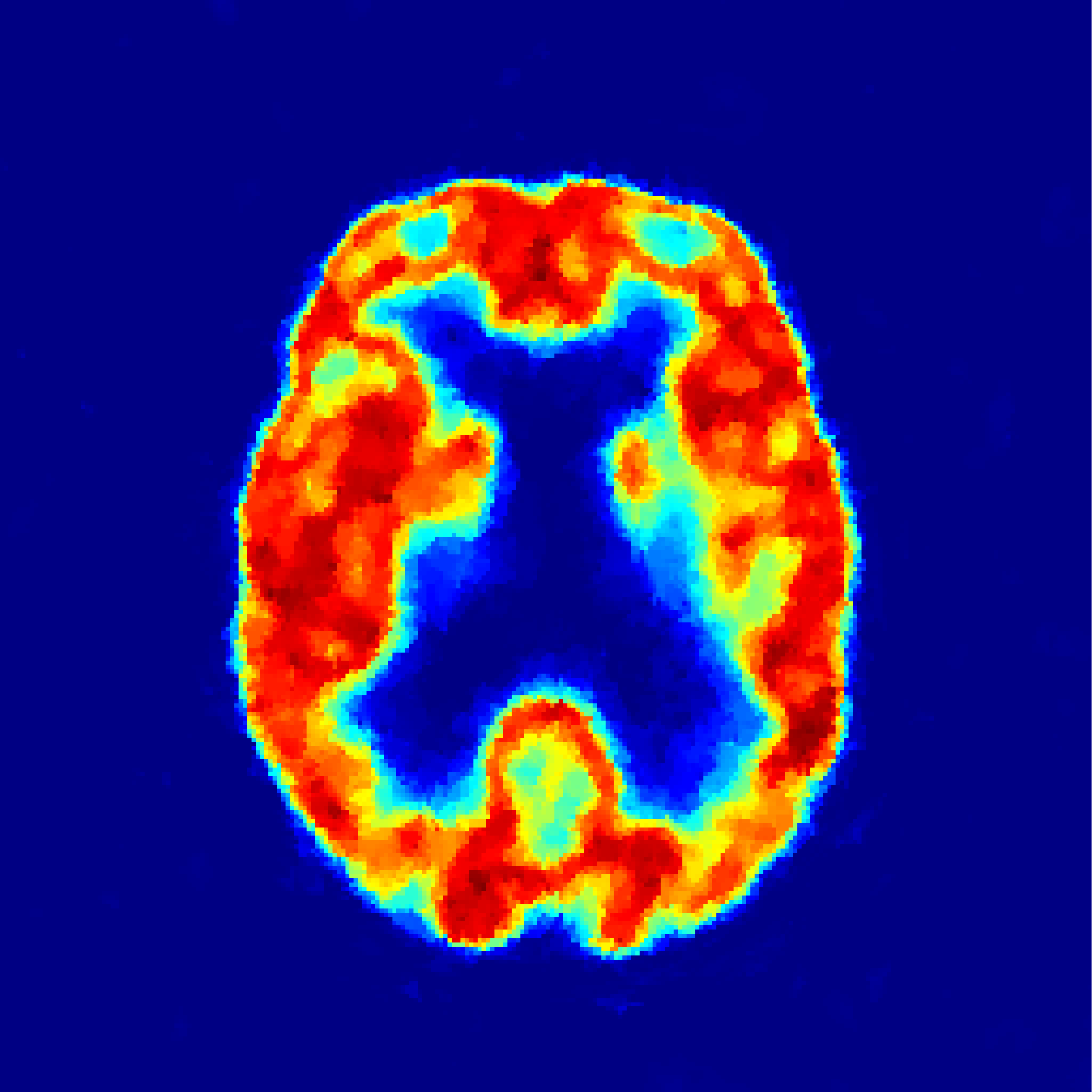}
		\end{minipage}&\hspace{-0.45cm}
		\begin{minipage}{3cm}
			\includegraphics[width=3cm]{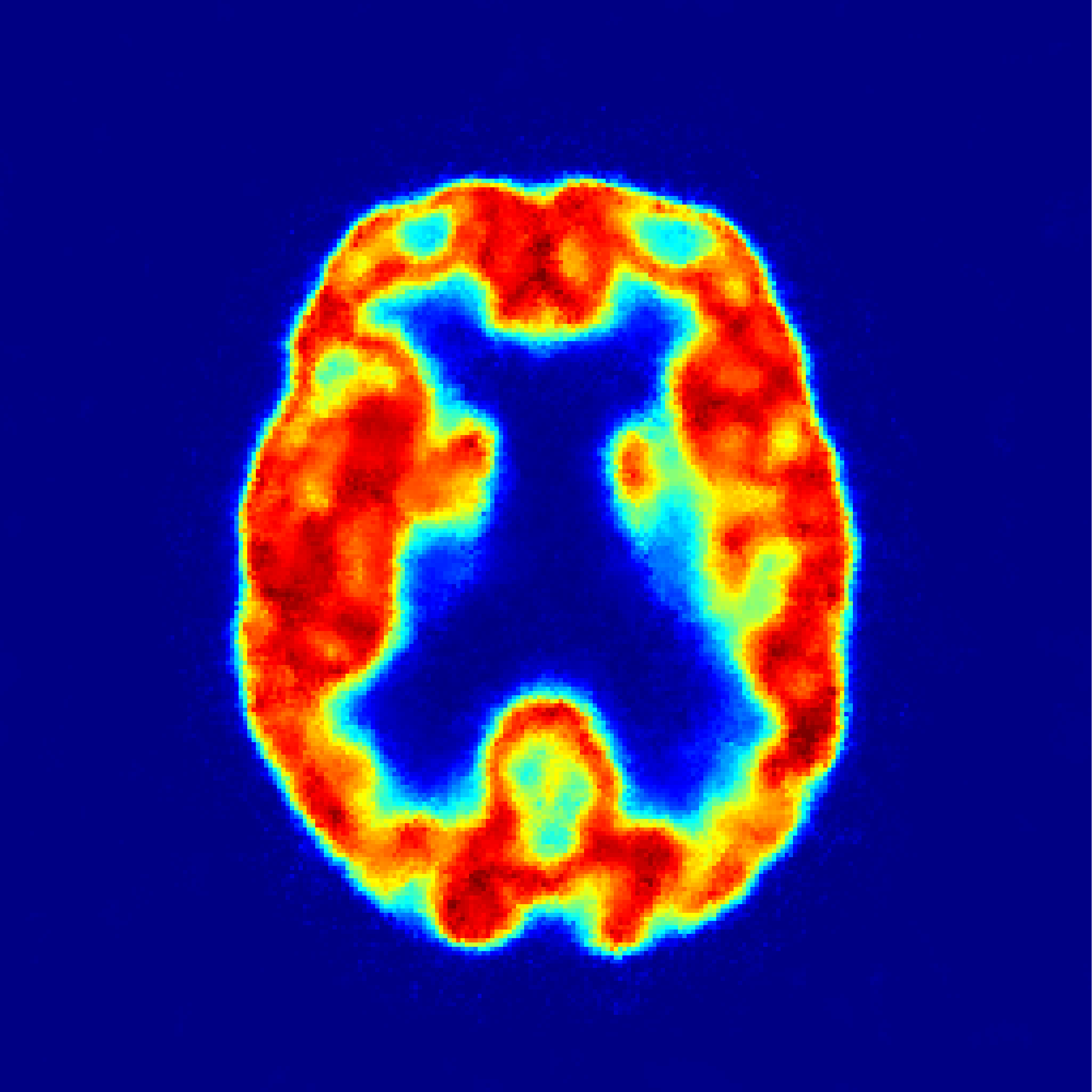}
		\end{minipage}&\hspace{-0.45cm}
		\begin{minipage}{3cm}
			\includegraphics[width=3cm]{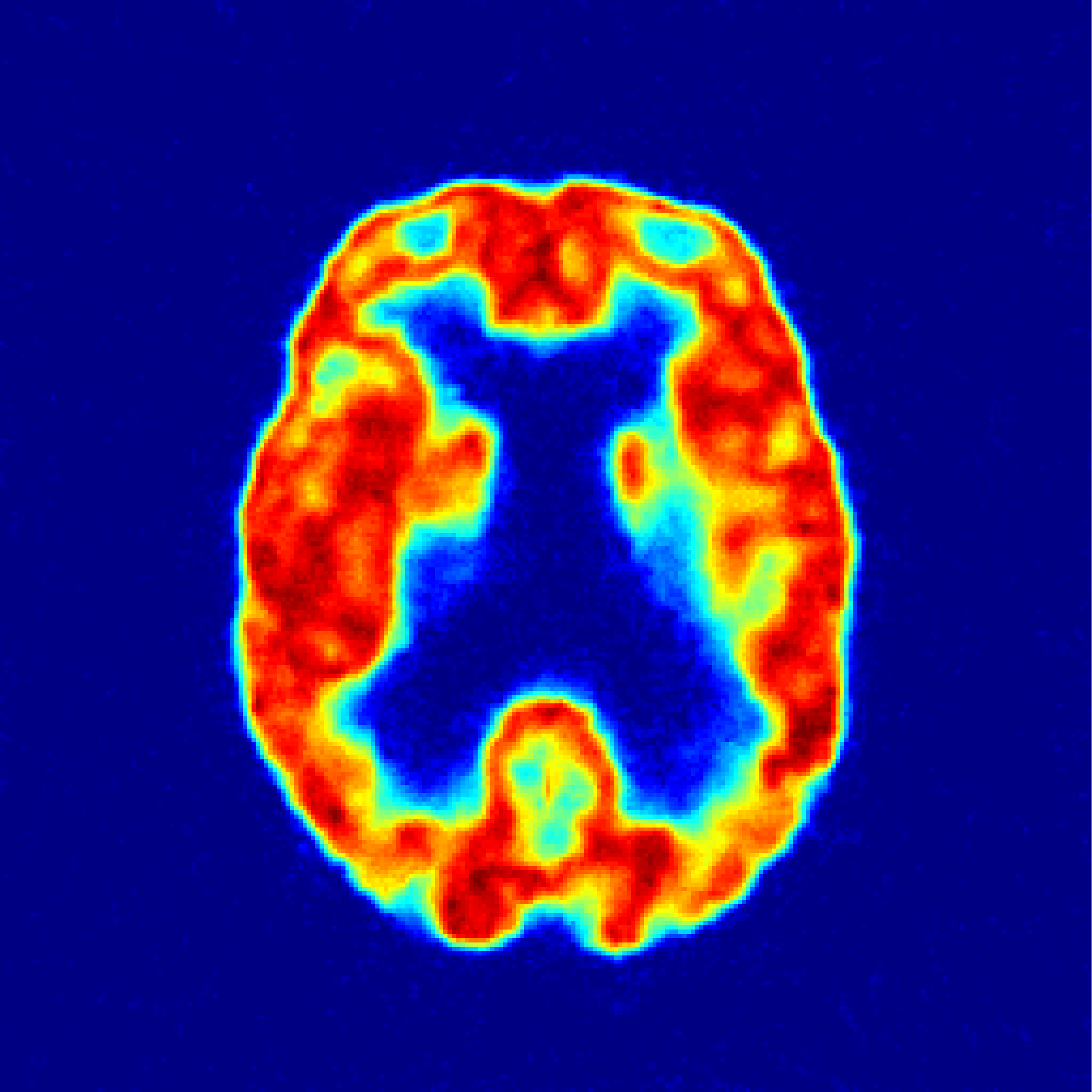}
		\end{minipage}&\hspace{-0.45cm}
		\begin{minipage}{3cm}
			\includegraphics[width=3cm]{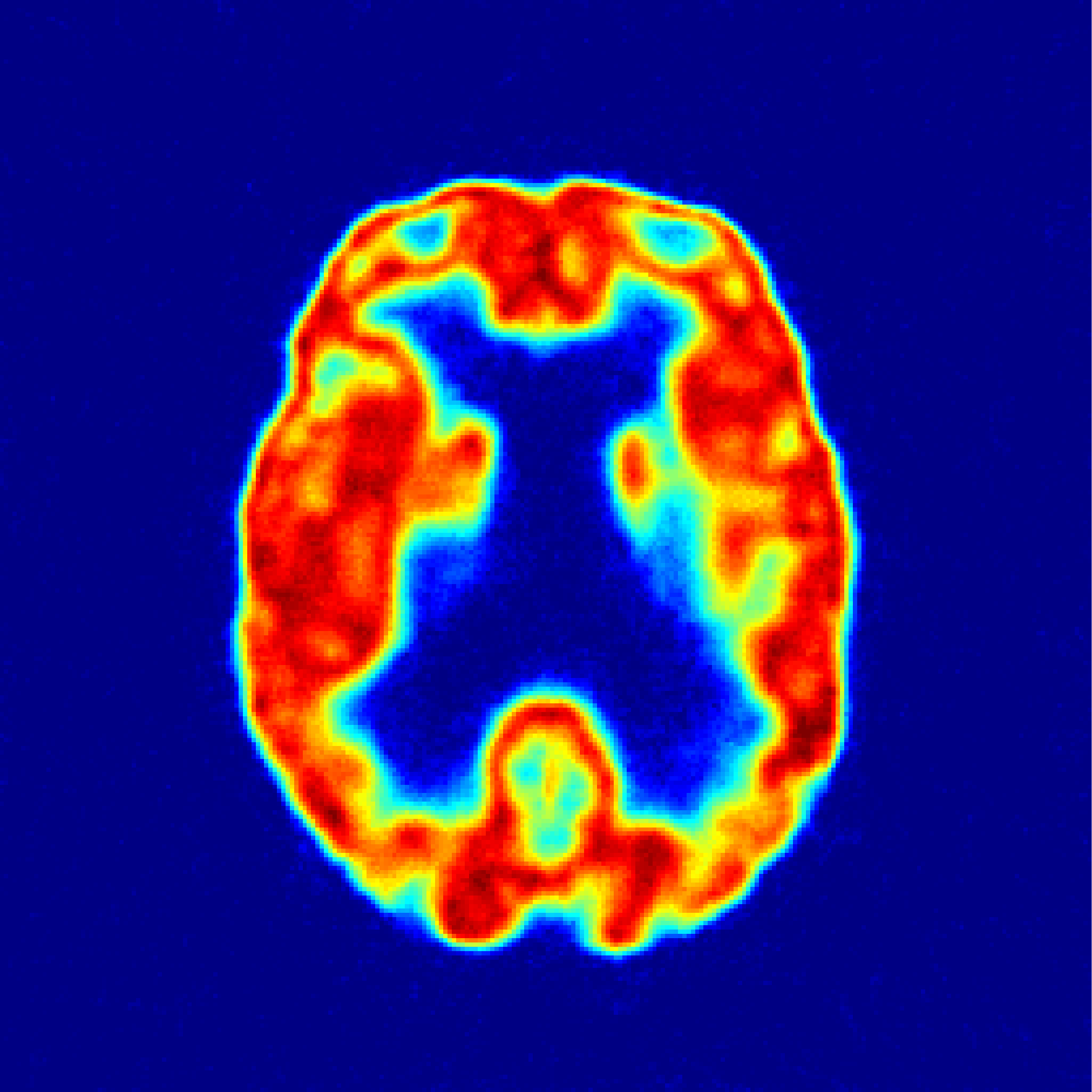}
		\end{minipage}\\
{\small{QPLS \cite{M.J.Ehrhardt2015}}}&\hspace{-0.45cm}
		{\small{JAnal \eqref{JAnalPETMRI}}}&\hspace{-0.45cm}
		{\small{JSTF \eqref{Proposed}}}&\hspace{-0.45cm}
		{\small{JSDDTF \eqref{DDTFPETMRI}}}\\
		\begin{minipage}{3cm}
			\includegraphics[width=3cm]{MRT115Original.pdf}
		\end{minipage}&\hspace{-0.45cm}
		\begin{minipage}{3cm}
			\includegraphics[width=3cm]{MRT115RandomZeroFill.pdf}
		\end{minipage}&\hspace{-0.45cm}
		\begin{minipage}{3cm}
			\includegraphics[width=3cm]{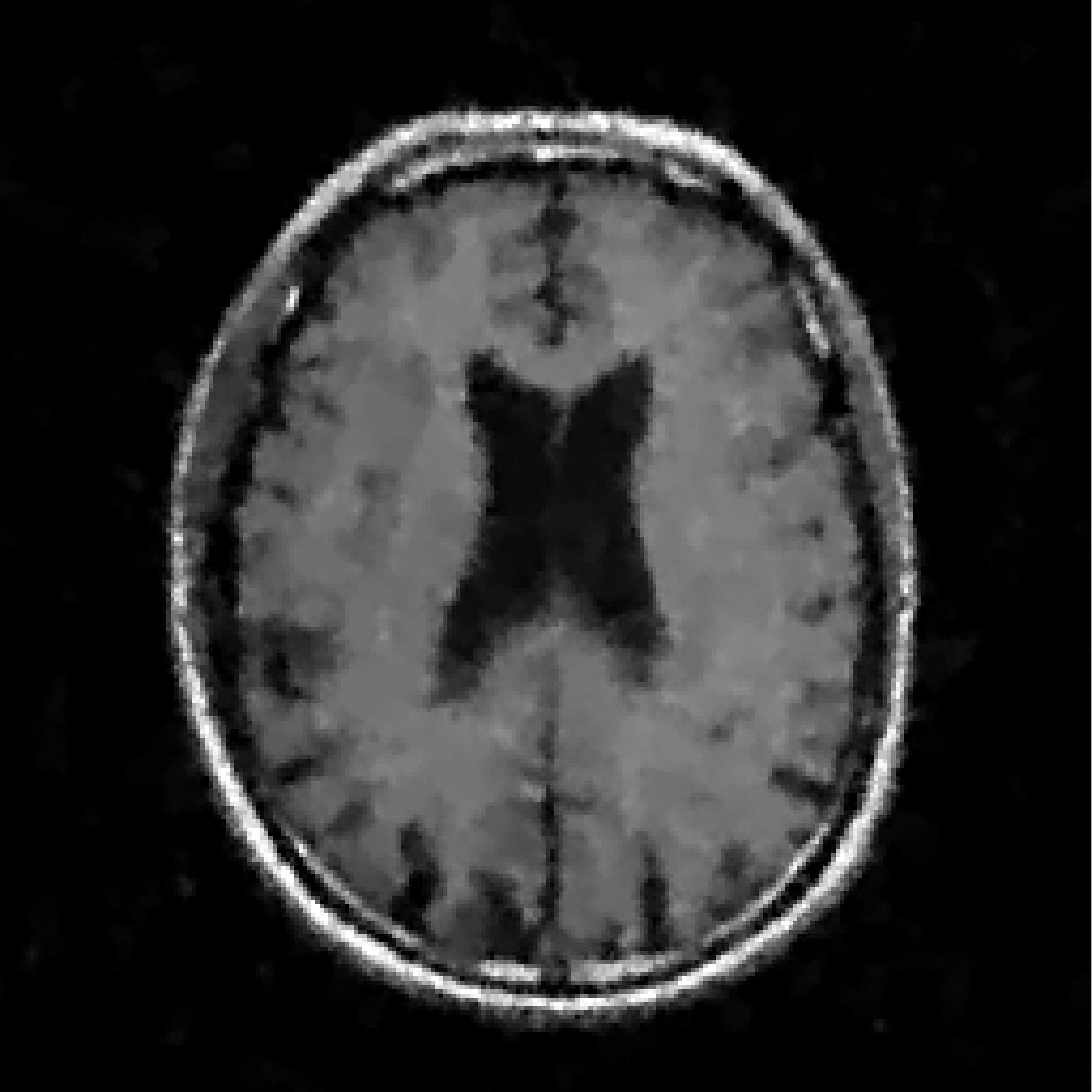}
		\end{minipage}&\hspace{-0.45cm}
\begin{minipage}{3cm}
\includegraphics[width=3cm]{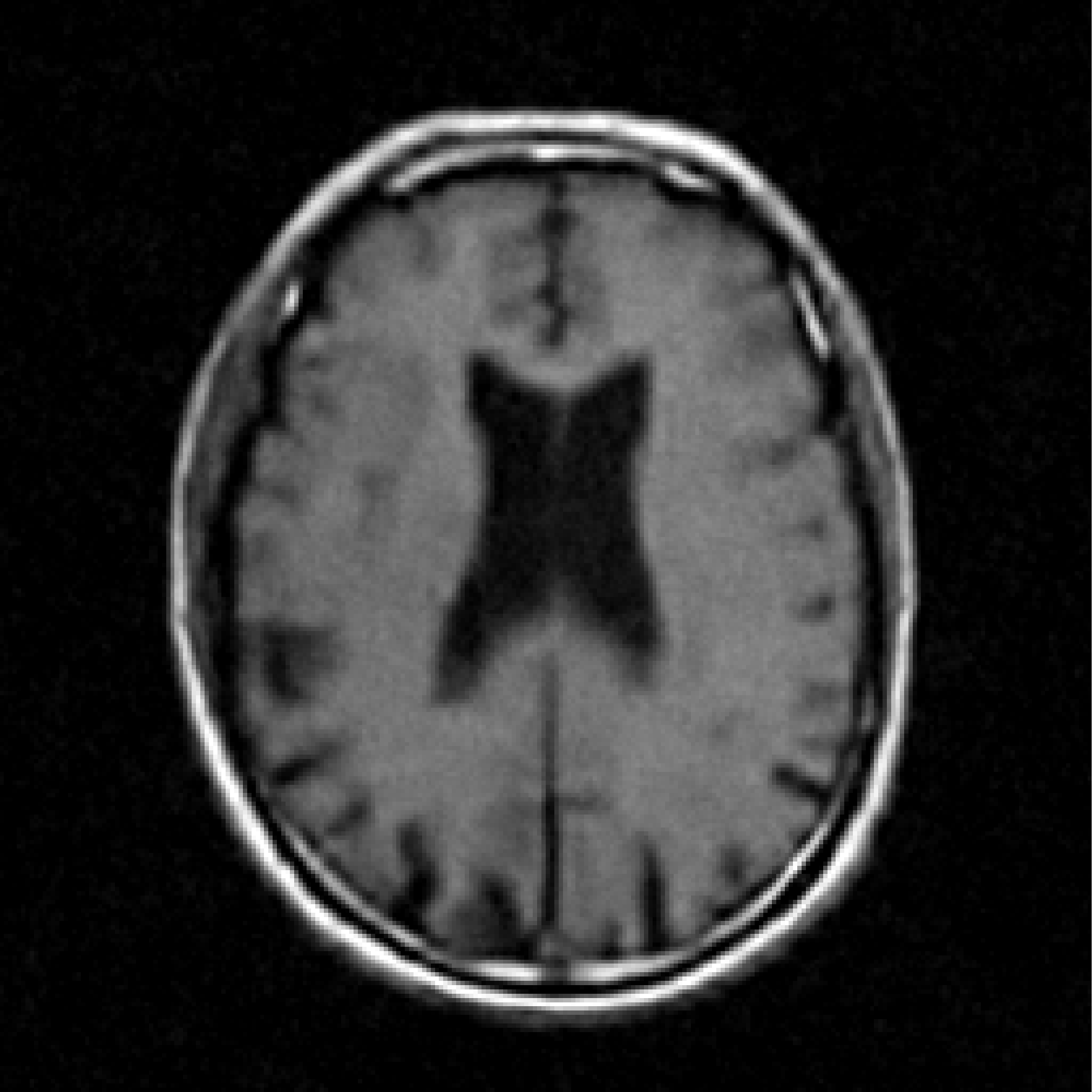}
\end{minipage}\\
		{\small{Original}}&\hspace{-0.45cm}
{\small{Initial}}&\hspace{-0.45cm}
{\small{Analysis \eqref{AnaMRI}}}&\hspace{-0.45cm}
		{\small{DDTF \eqref{DDTFMRI}}}\\
\begin{minipage}{3cm}
			\includegraphics[width=3cm]{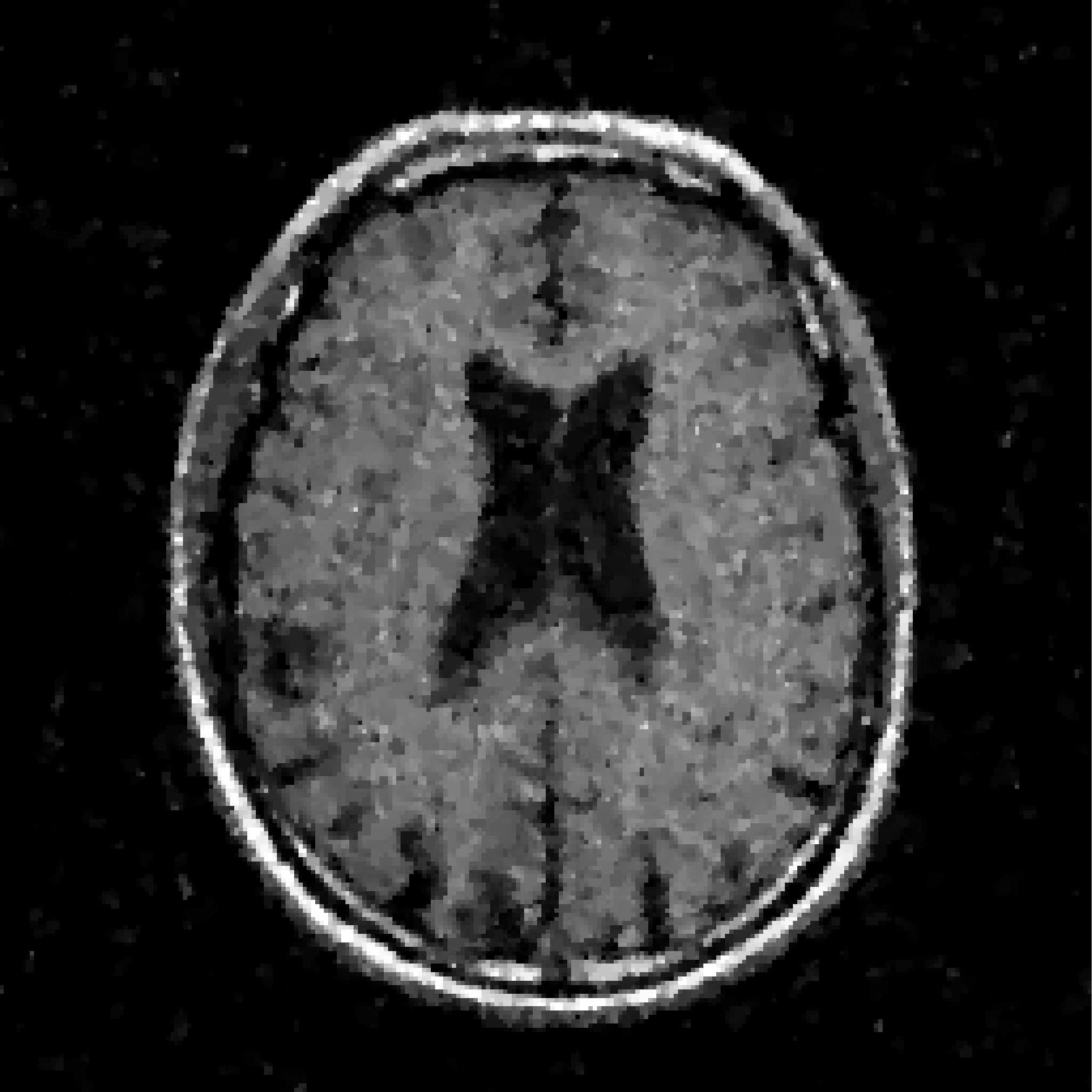}
		\end{minipage}&\hspace{-0.45cm}
		\begin{minipage}{3cm}
			\includegraphics[width=3cm]{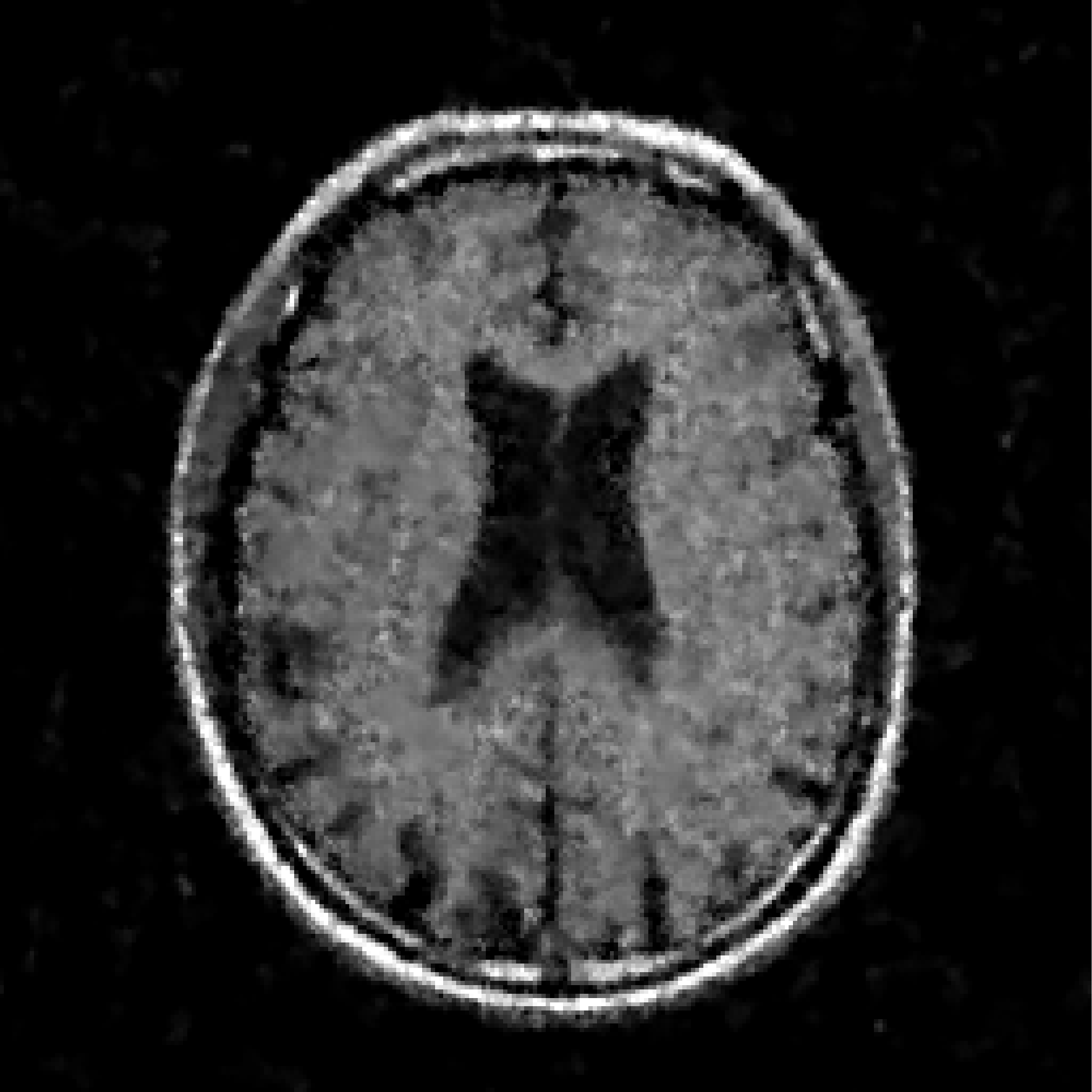}
		\end{minipage}&\hspace{-0.45cm}
		\begin{minipage}{3cm}
			\includegraphics[width=3cm]{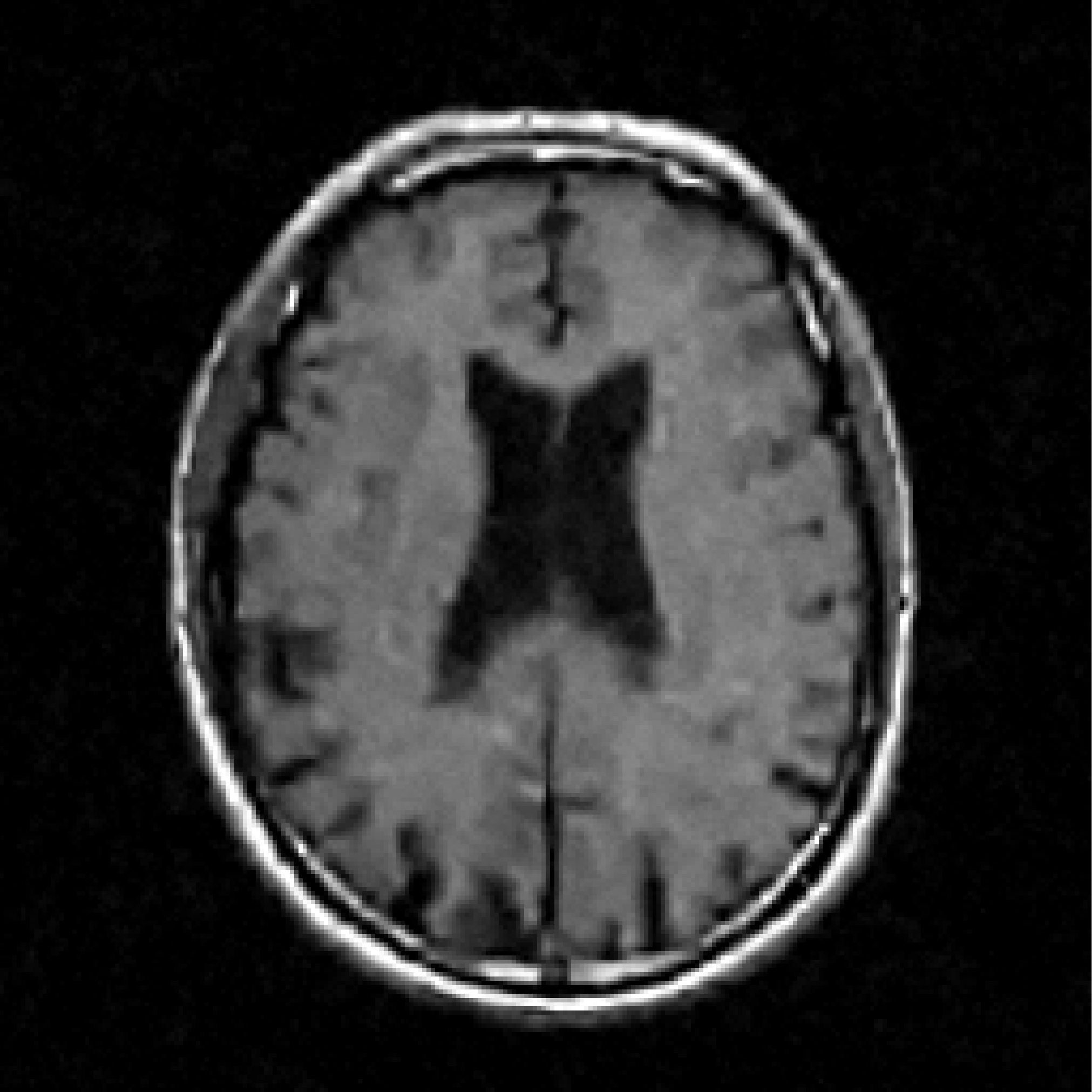}
		\end{minipage}&\hspace{-0.45cm}
		\begin{minipage}{3cm}
			\includegraphics[width=3cm]{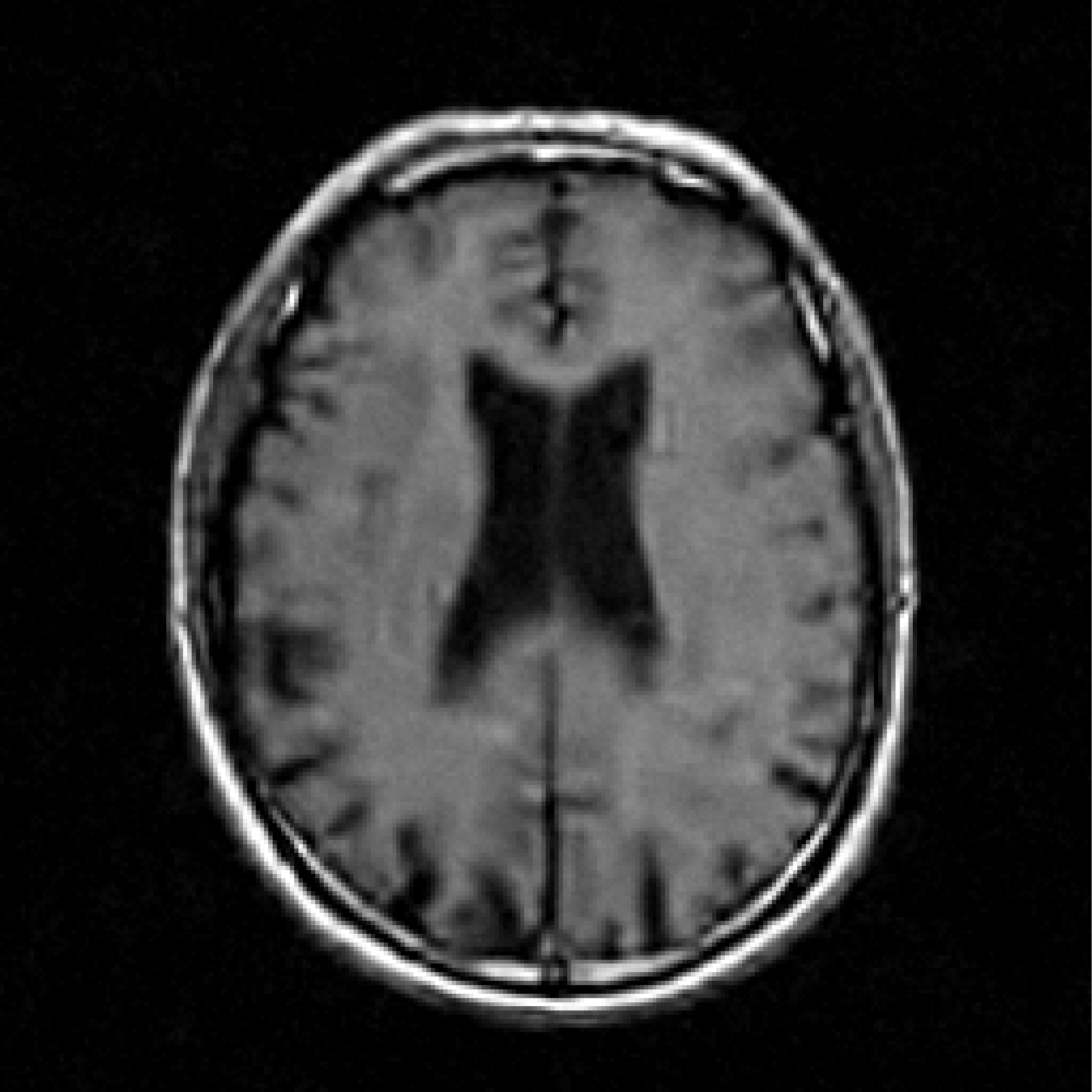}
		\end{minipage}\\
{\small{QPLS \cite{M.J.Ehrhardt2015}}}&\hspace{-0.45cm}
		{\small{JAnal \eqref{JAnalPETMRI}}}&\hspace{-0.45cm}
		{\small{JSTF \eqref{Proposed}}}&\hspace{-0.45cm}
		{\small{JSDDTF \eqref{DDTFPETMRI}}}
	\end{tabular}
	\caption{Visual comparison of PET-T1 Random joint reconstruction results. The first and second rows describe the PET images, and the third and fourth rows depict the MRI images.}\label{PETMRT115RandomResults}
\end{figure}

\begin{figure}[htp!]
	\centering
	\begin{tabular}{cccc}
		\begin{minipage}{3cm}
			\includegraphics[width=3cm]{PET15Original.pdf}
		\end{minipage}&\hspace{-0.45cm}
\begin{minipage}{3cm}
			\includegraphics[width=3cm]{PET15EM.pdf}
		\end{minipage}&\hspace{-0.45cm}
		\begin{minipage}{3cm}
			\includegraphics[width=3cm]{PET15Anal.pdf}
		\end{minipage}&\hspace{-0.45cm}
		\begin{minipage}{3cm}
			\includegraphics[width=3cm]{PET15DDTF.pdf}
		\end{minipage}\\
		{\small{Original}}&\hspace{-0.45cm}
{\small{Initial}}&\hspace{-0.45cm}
		{\small{Analysis \eqref{AnaPET}}}&\hspace{-0.45cm}
		{\small{DDTF \eqref{DDTFPET}}}\\
		\begin{minipage}{3cm}
			\includegraphics[width=3cm]{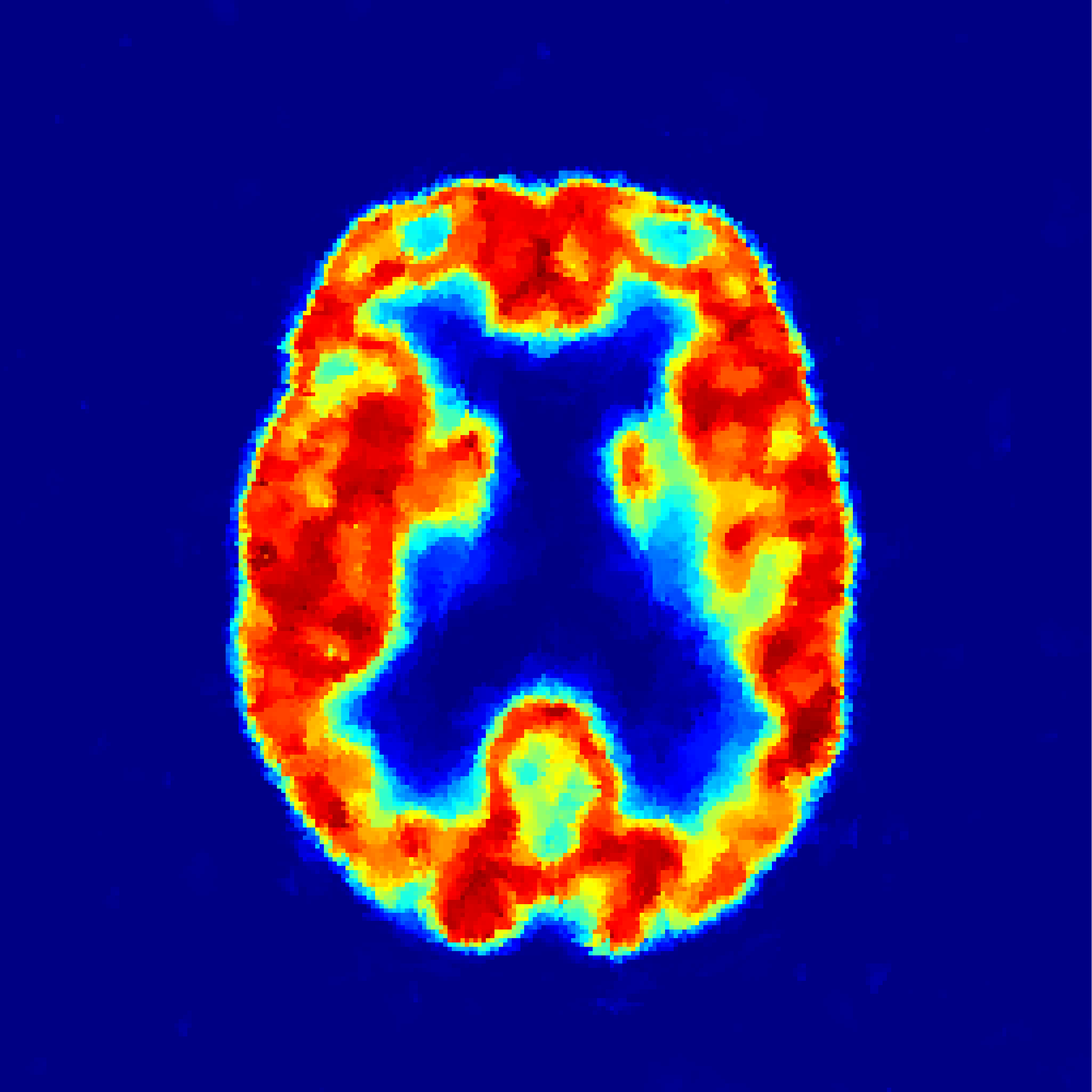}
		\end{minipage}&\hspace{-0.45cm}
		\begin{minipage}{3cm}
			\includegraphics[width=3cm]{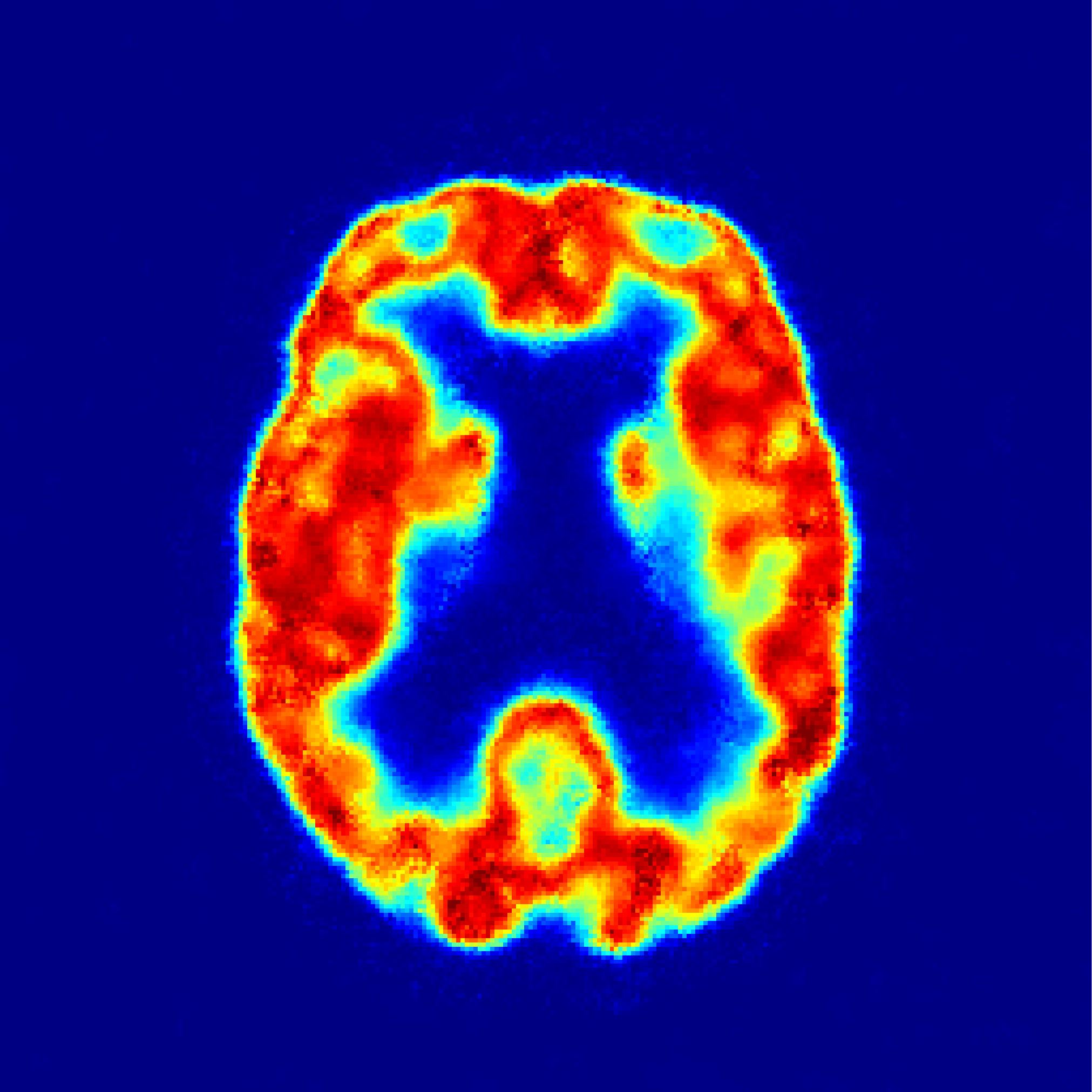}
		\end{minipage}&\hspace{-0.45cm}
		\begin{minipage}{3cm}
			\includegraphics[width=3cm]{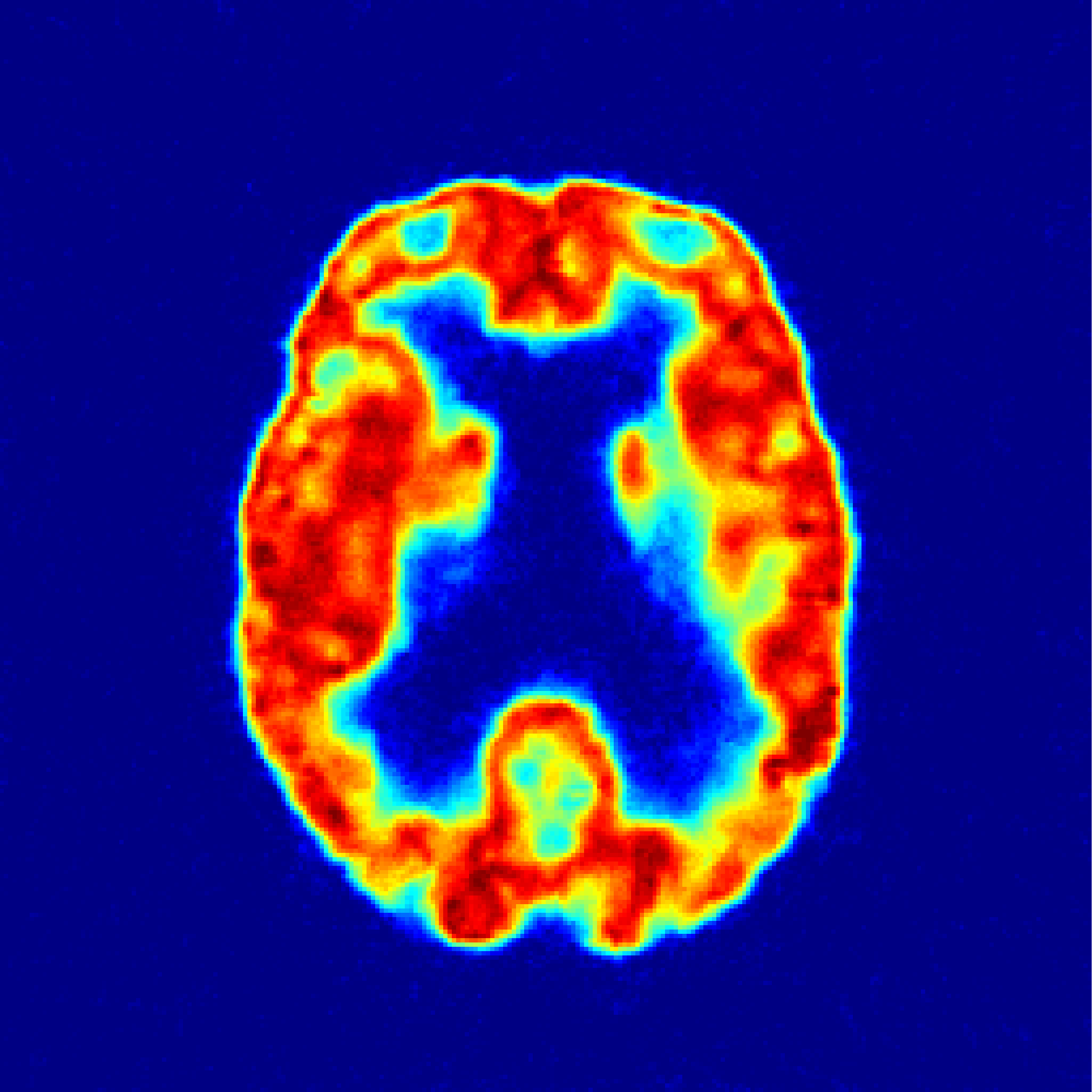}
		\end{minipage}&\hspace{-0.45cm}
		\begin{minipage}{3cm}
			\includegraphics[width=3cm]{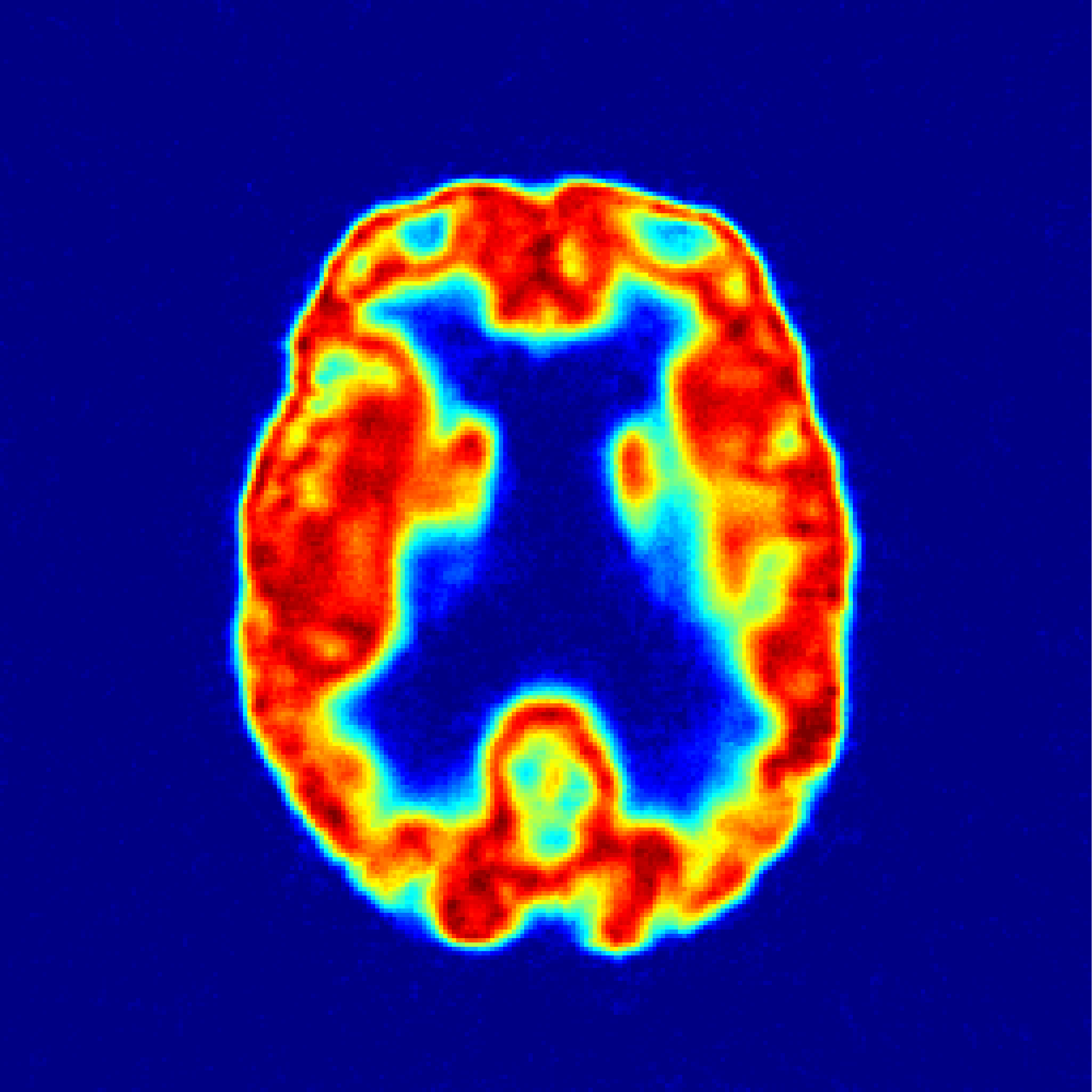}
		\end{minipage}\\
{\small{QPLS \cite{M.J.Ehrhardt2015}}}&\hspace{-0.45cm}
		{\small{JAnal \eqref{JAnalPETMRI}}}&\hspace{-0.45cm}
		{\small{JSTF \eqref{Proposed}}}&\hspace{-0.45cm}
		{\small{JSDDTF \eqref{DDTFPETMRI}}}\\
		\begin{minipage}{3cm}
			\includegraphics[width=3cm]{MRT215Original.pdf}
		\end{minipage}&\hspace{-0.45cm}
		\begin{minipage}{3cm}
			\includegraphics[width=3cm]{MRT215RandomZeroFill.pdf}
		\end{minipage}&\hspace{-0.45cm}
		\begin{minipage}{3cm}
			\includegraphics[width=3cm]{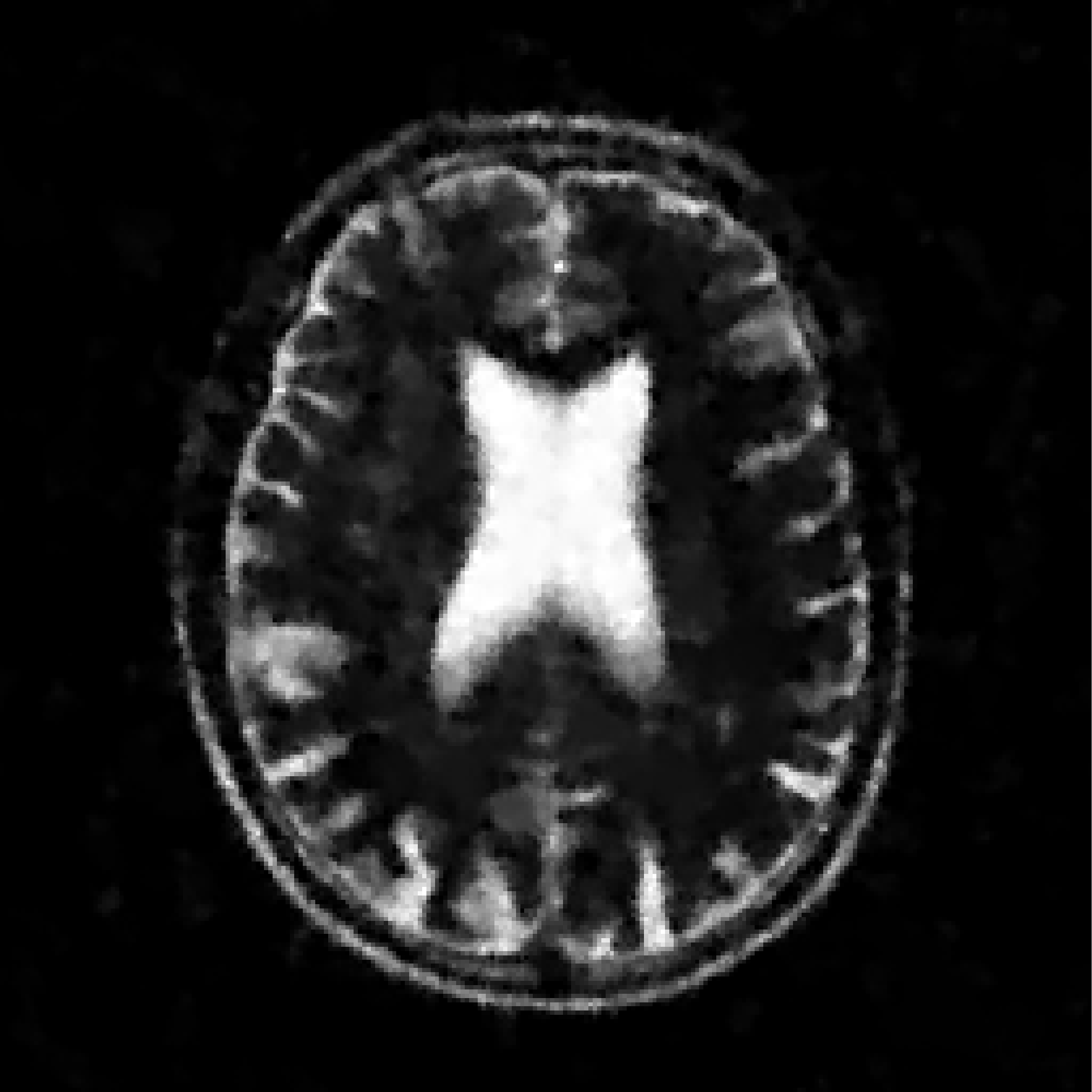}
		\end{minipage}&\hspace{-0.45cm}
\begin{minipage}{3cm}
\includegraphics[width=3cm]{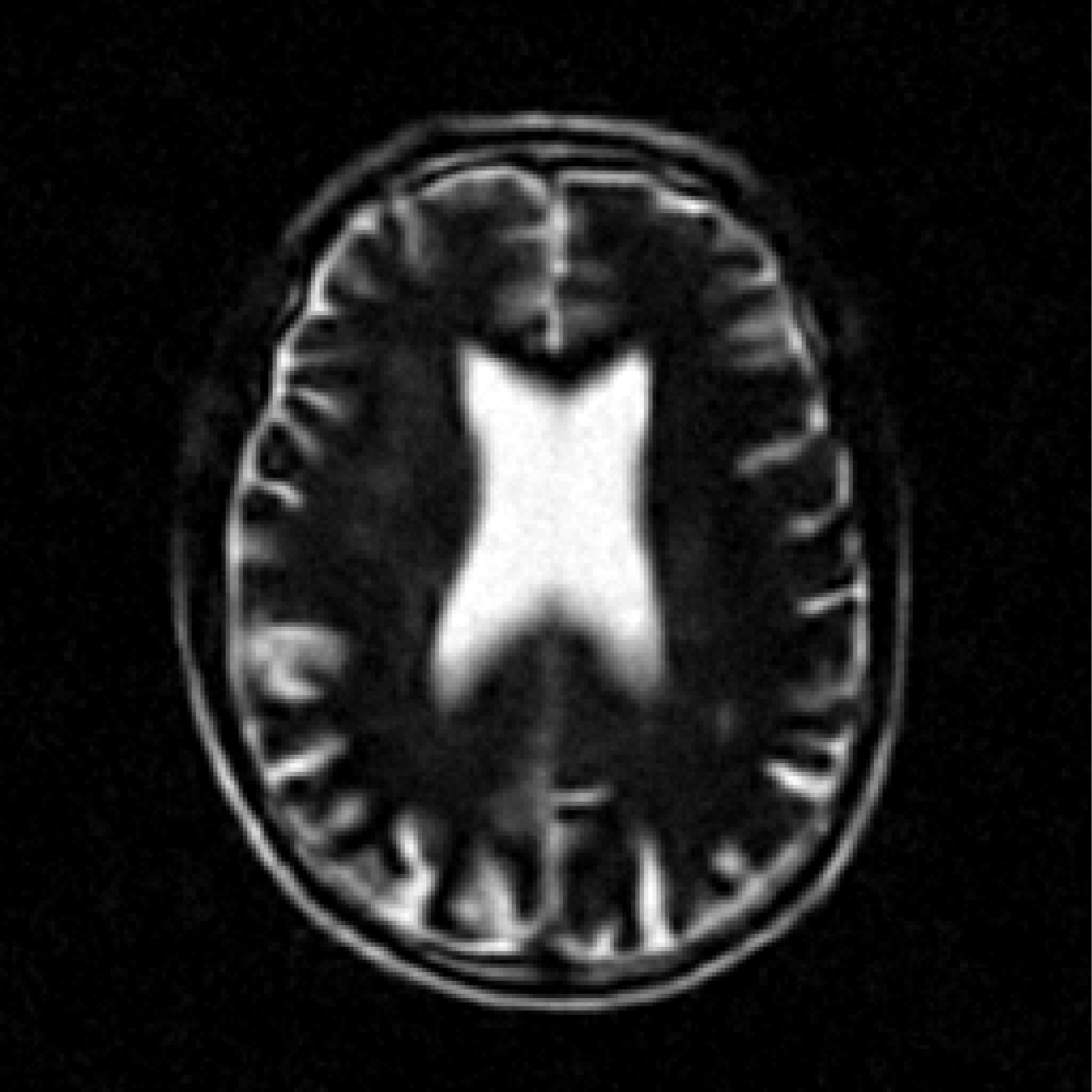}
\end{minipage}\\
		{\small{Original}}&\hspace{-0.45cm}
{\small{Initial}}&\hspace{-0.45cm}
{\small{Analysis \eqref{AnaMRI}}}&\hspace{-0.45cm}
		{\small{DDTF \eqref{DDTFMRI}}}\\
\begin{minipage}{3cm}
			\includegraphics[width=3cm]{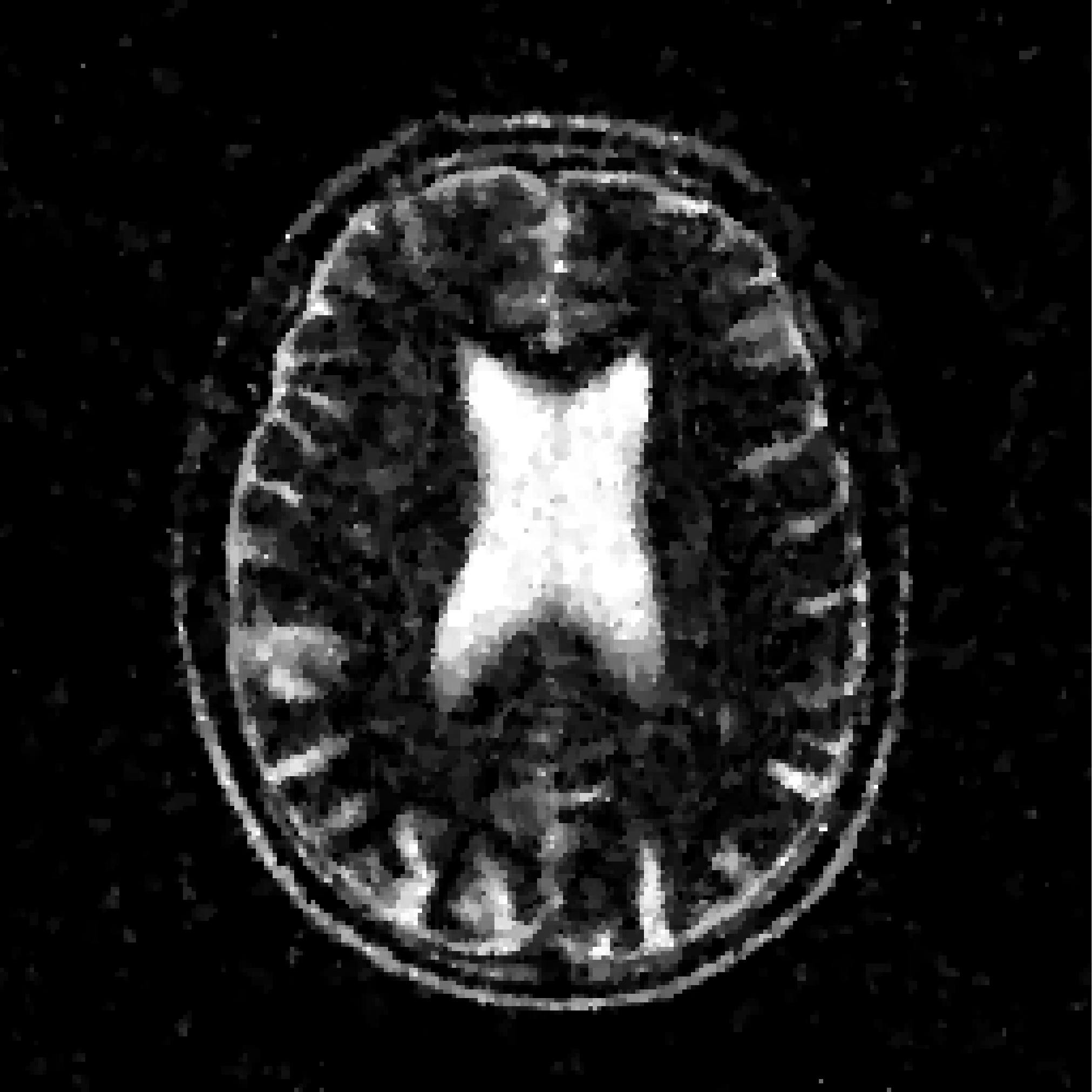}
		\end{minipage}&\hspace{-0.45cm}
		\begin{minipage}{3cm}
			\includegraphics[width=3cm]{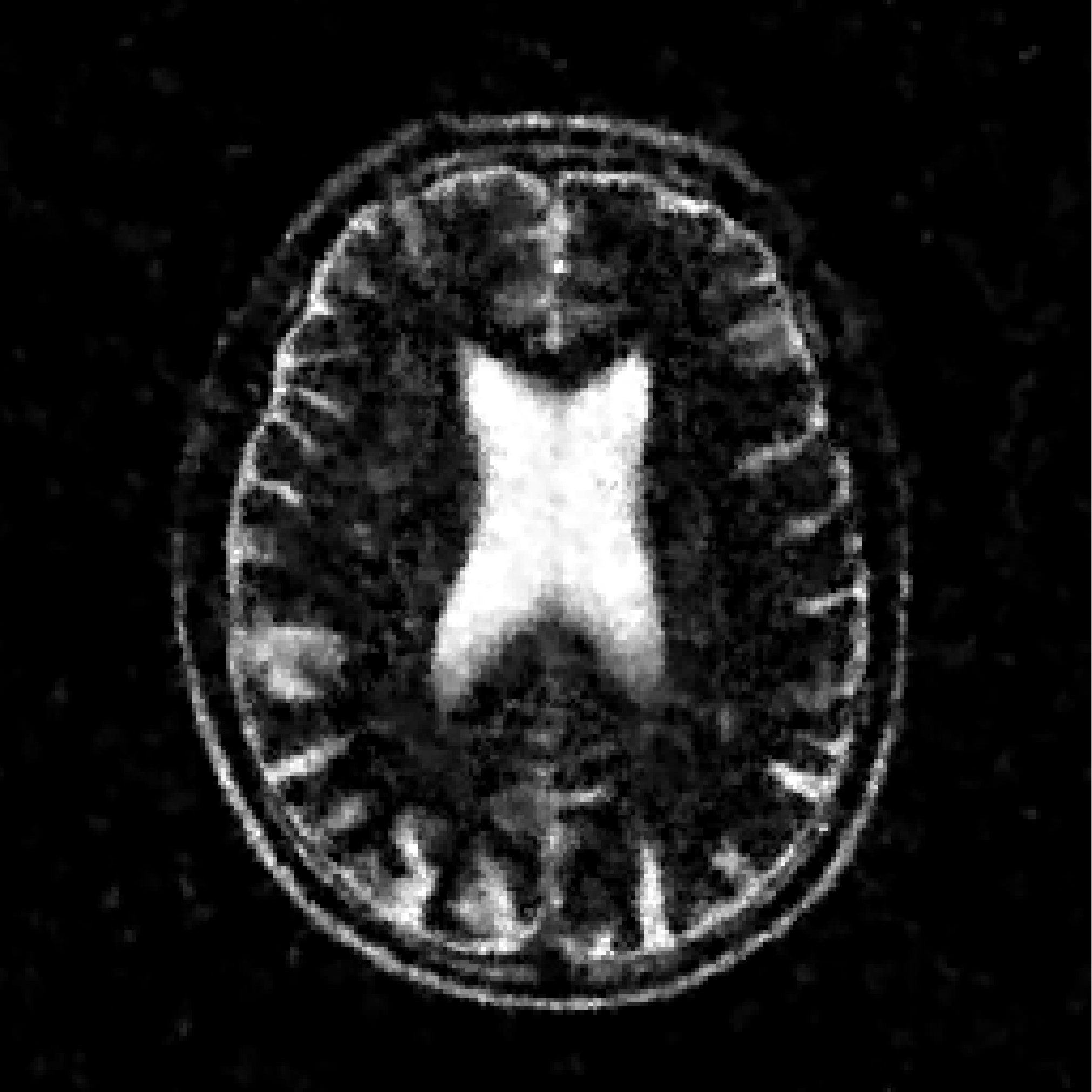}
		\end{minipage}&\hspace{-0.45cm}
		\begin{minipage}{3cm}
			\includegraphics[width=3cm]{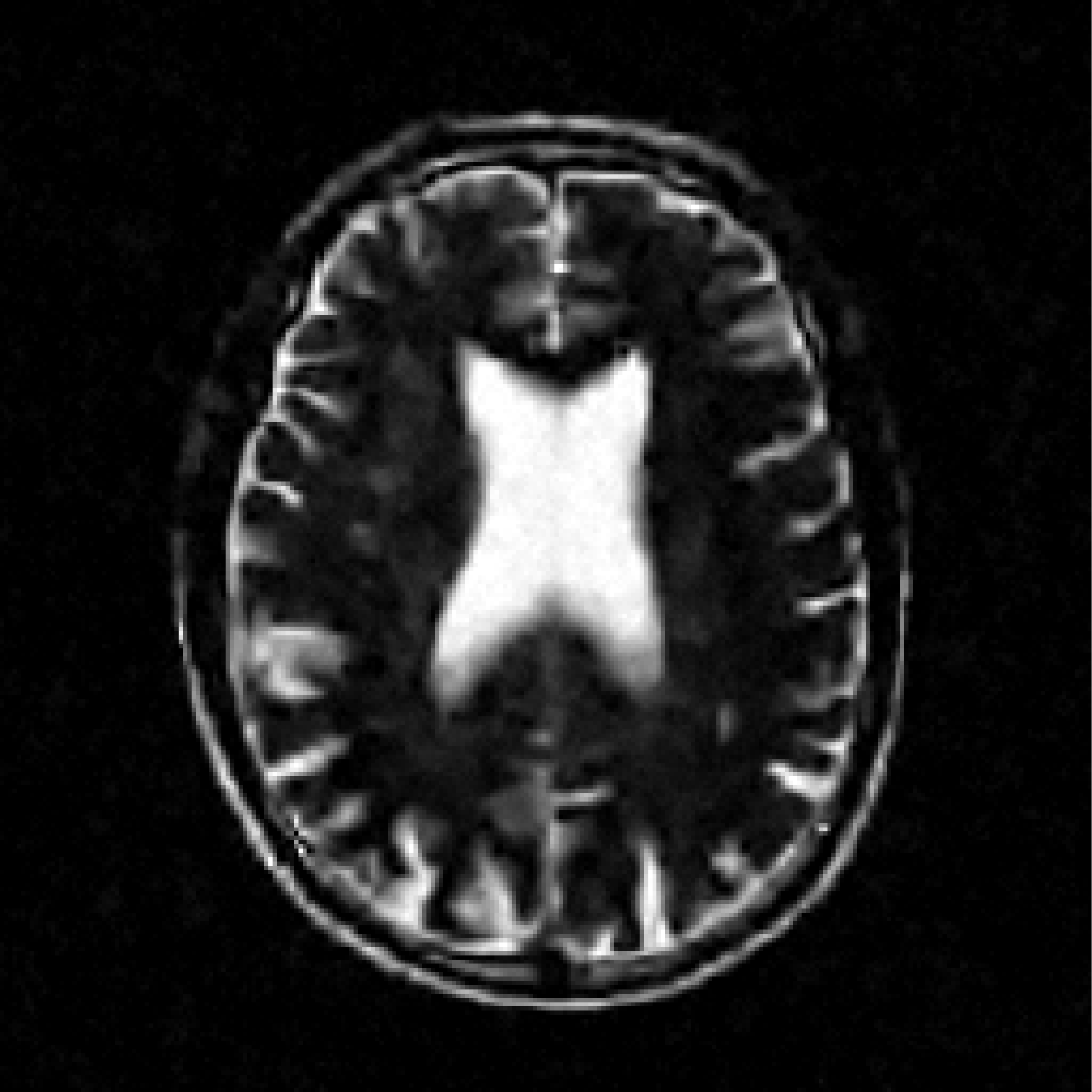}
		\end{minipage}&\hspace{-0.45cm}
		\begin{minipage}{3cm}
			\includegraphics[width=3cm]{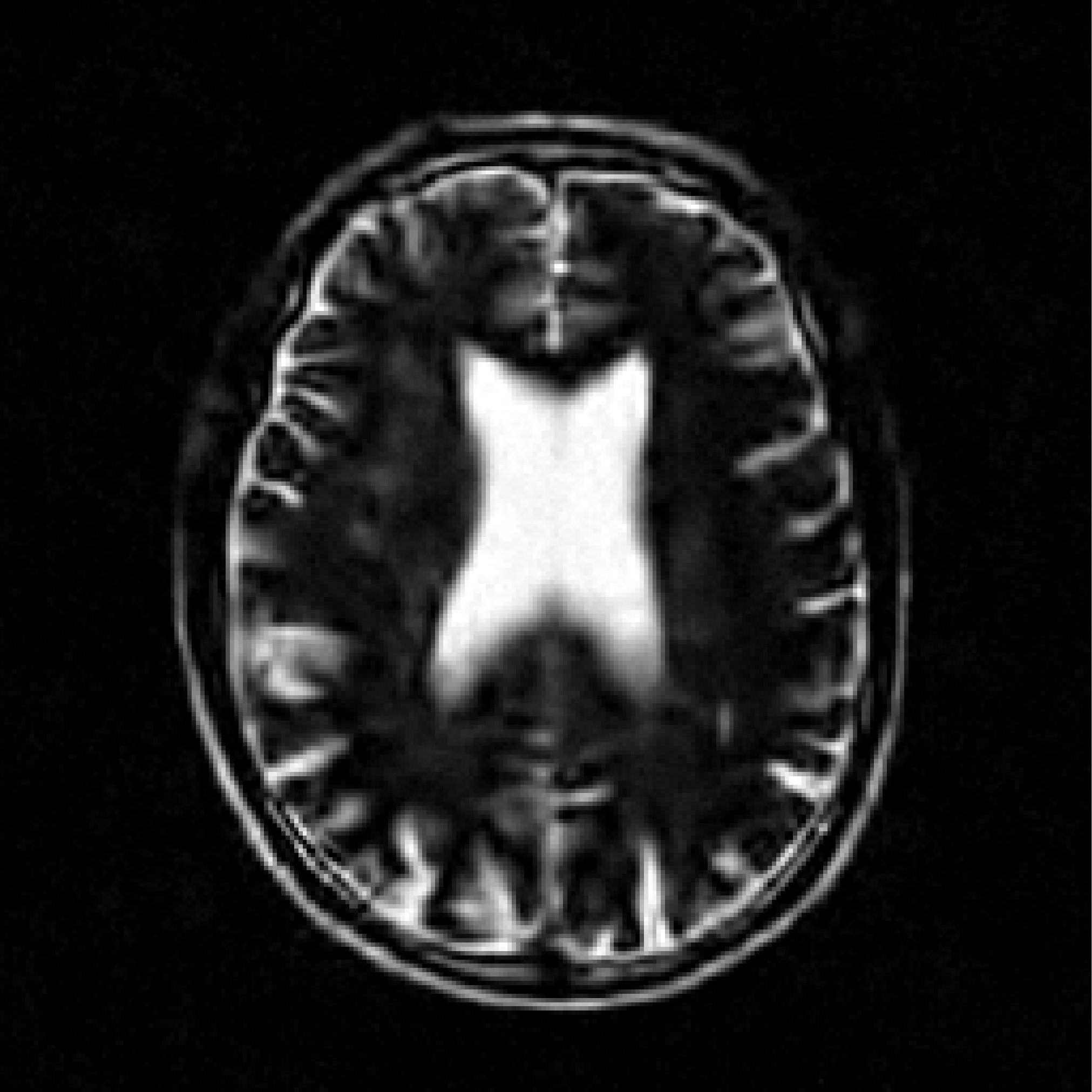}
		\end{minipage}\\
{\small{QPLS \cite{M.J.Ehrhardt2015}}}&\hspace{-0.45cm}
		{\small{JAnal \eqref{JAnalPETMRI}}}&\hspace{-0.45cm}
		{\small{JSTF \eqref{Proposed}}}&\hspace{-0.45cm}
		{\small{JSDDTF \eqref{DDTFPETMRI}}}
	\end{tabular}
	\caption{Visual comparison of PET-T2 Random joint reconstruction results. The first and second rows describe the PET images, and the third and fourth rows depict the MRI images.}\label{PETMRT215RandomResults}
\end{figure}

\begin{figure}[htp!]
	\centering
	\begin{tabular}{cccc}
		\begin{minipage}{3cm}
			\includegraphics[width=3cm]{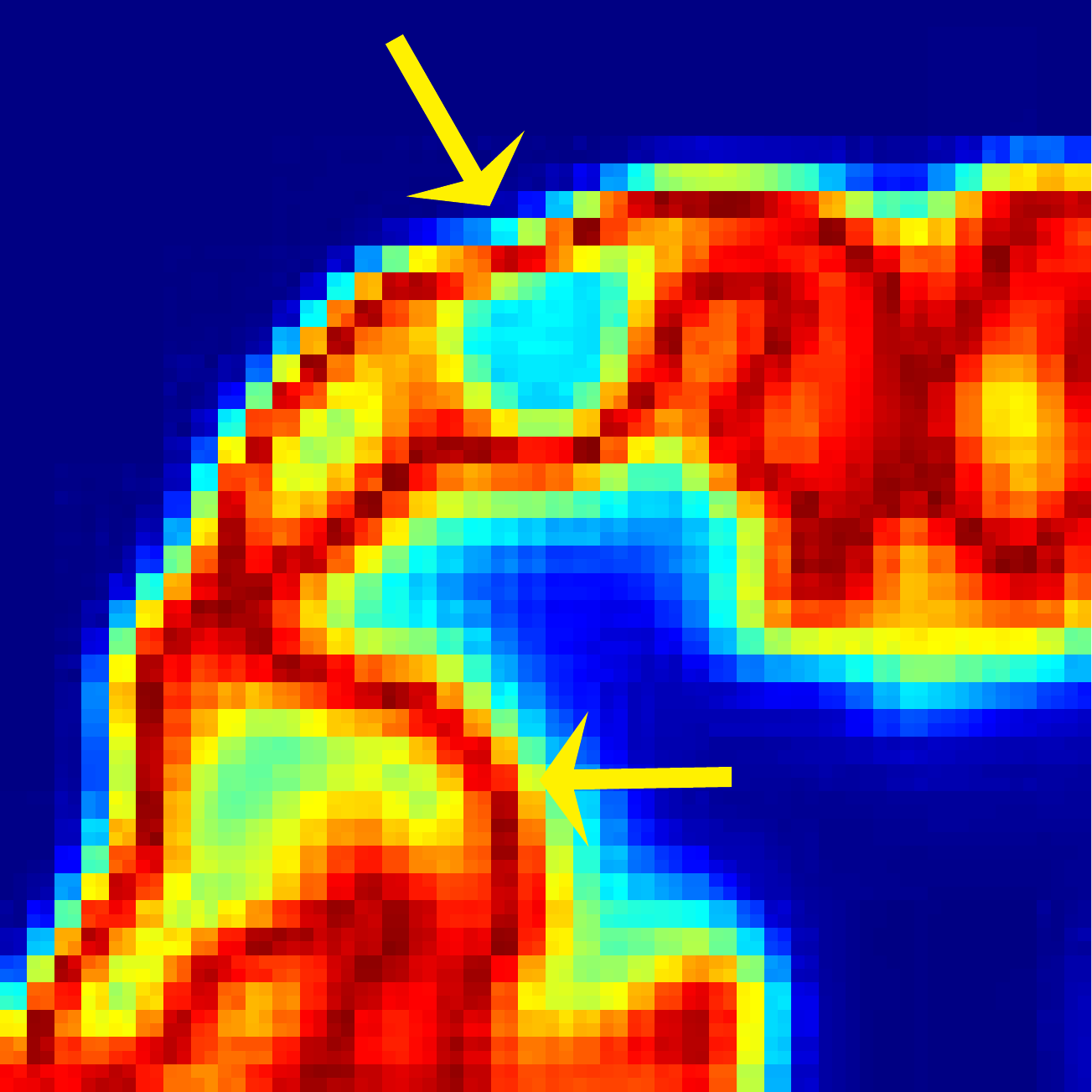}
		\end{minipage}&\hspace{-0.45cm}
\begin{minipage}{3cm}
			\includegraphics[width=3cm]{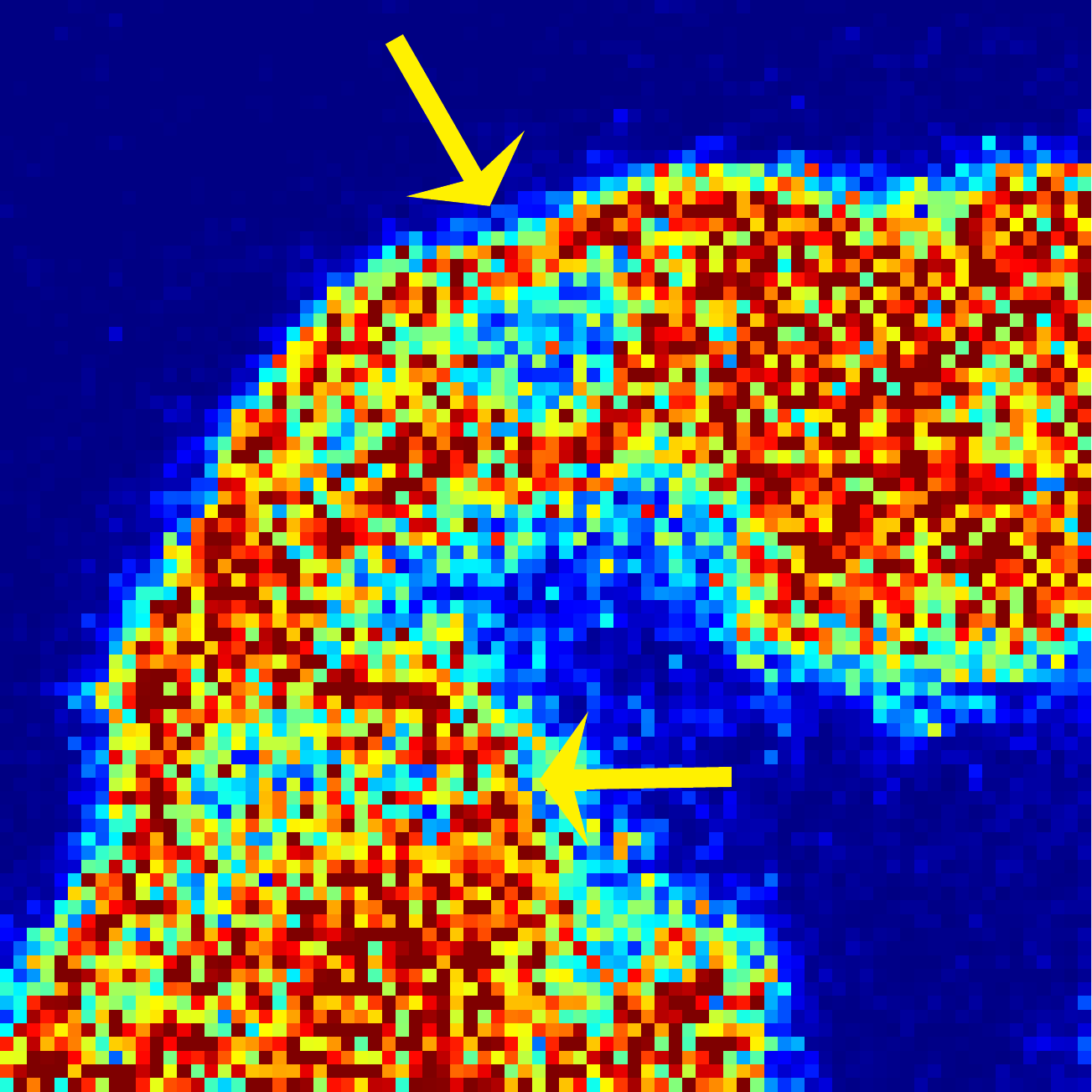}
		\end{minipage}&\hspace{-0.45cm}
		\begin{minipage}{3cm}
			\includegraphics[width=3cm]{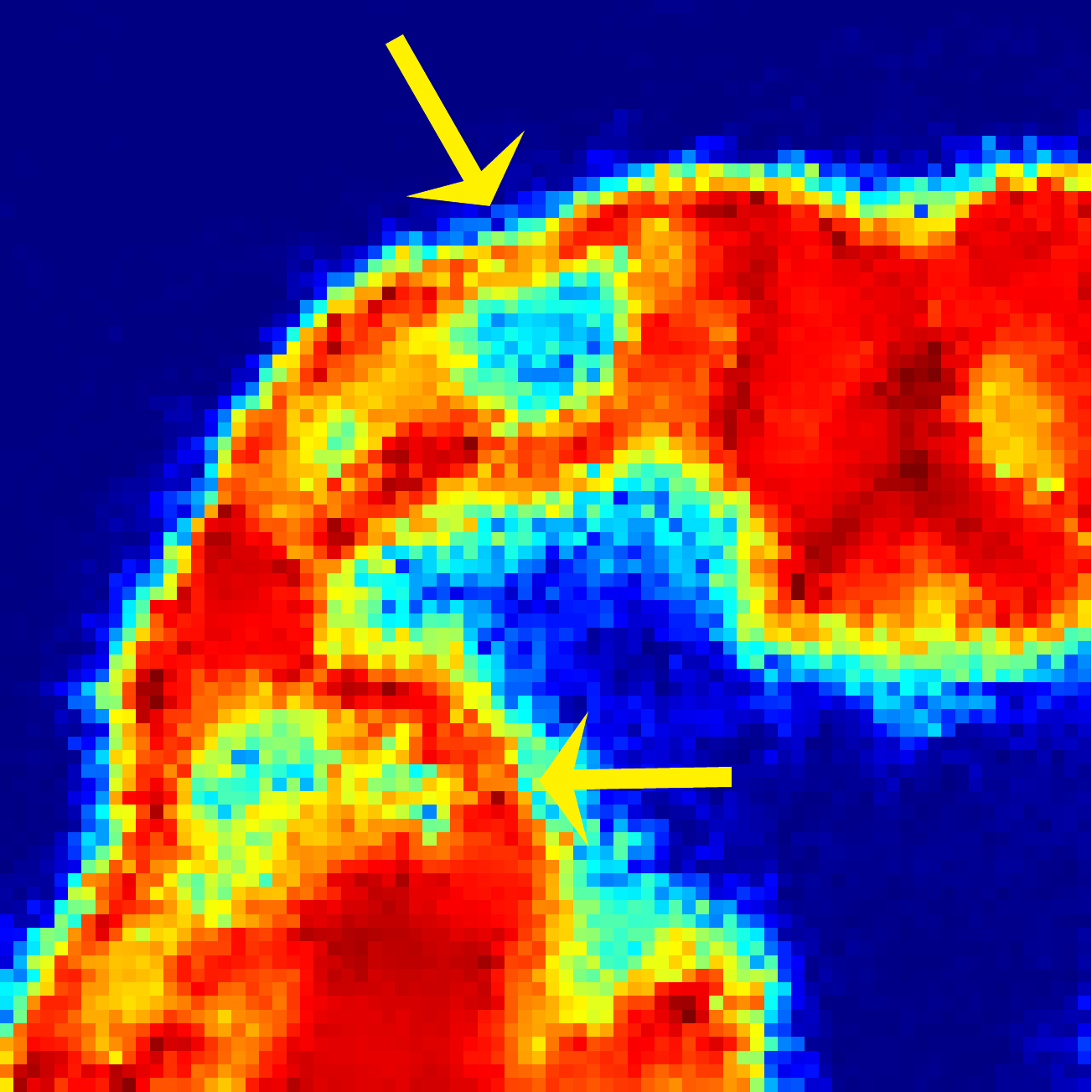}
		\end{minipage}&\hspace{-0.45cm}
		\begin{minipage}{3cm}
			\includegraphics[width=3cm]{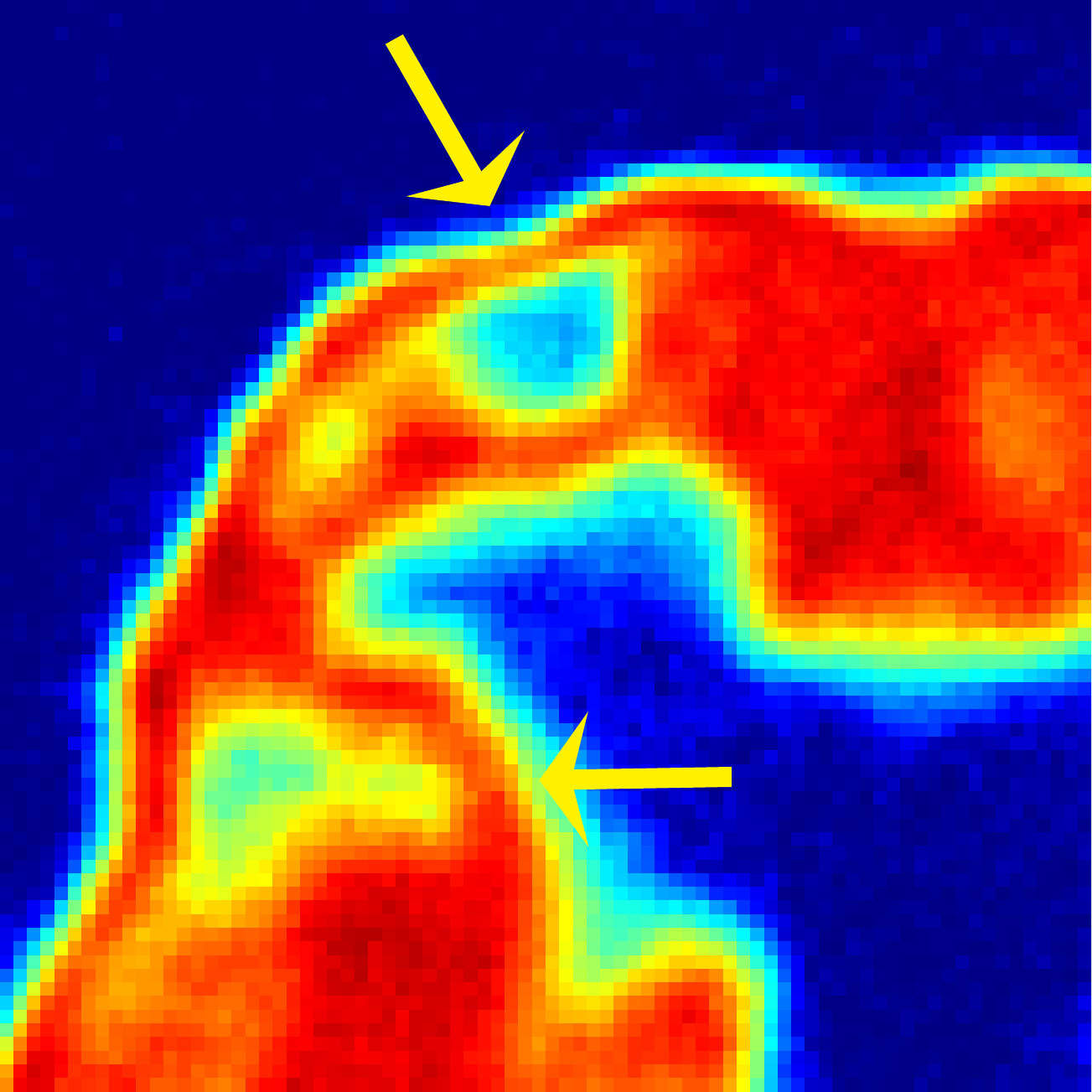}
		\end{minipage}\\
		{\small{Original}}&\hspace{-0.45cm}
{\small{Initial}}&\hspace{-0.45cm}
		{\small{Analysis \eqref{AnaPET}}}&\hspace{-0.45cm}
		{\small{DDTF \eqref{DDTFPET}}}\\
		\begin{minipage}{3cm}
			\includegraphics[width=3cm]{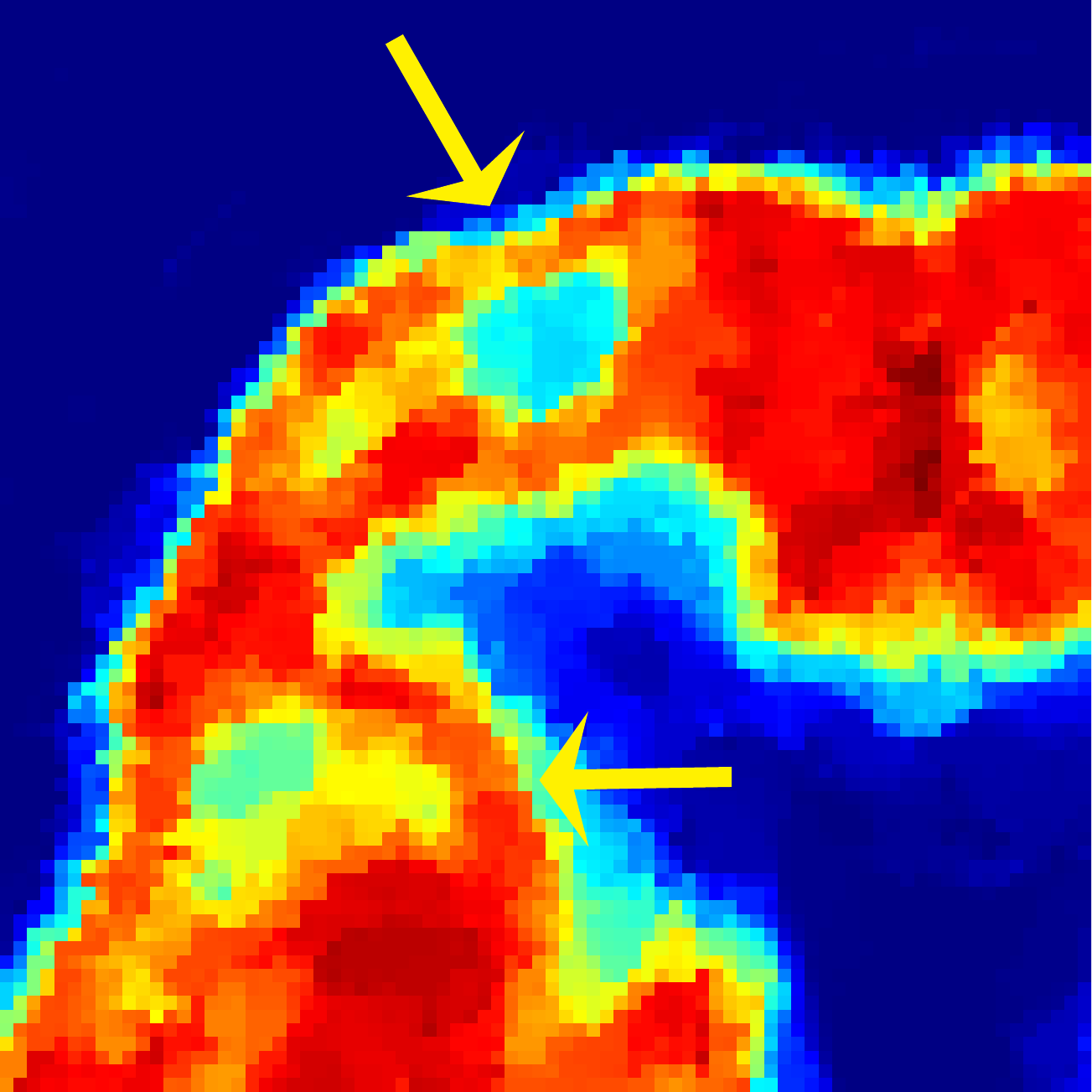}
		\end{minipage}&\hspace{-0.45cm}
		\begin{minipage}{3cm}
			\includegraphics[width=3cm]{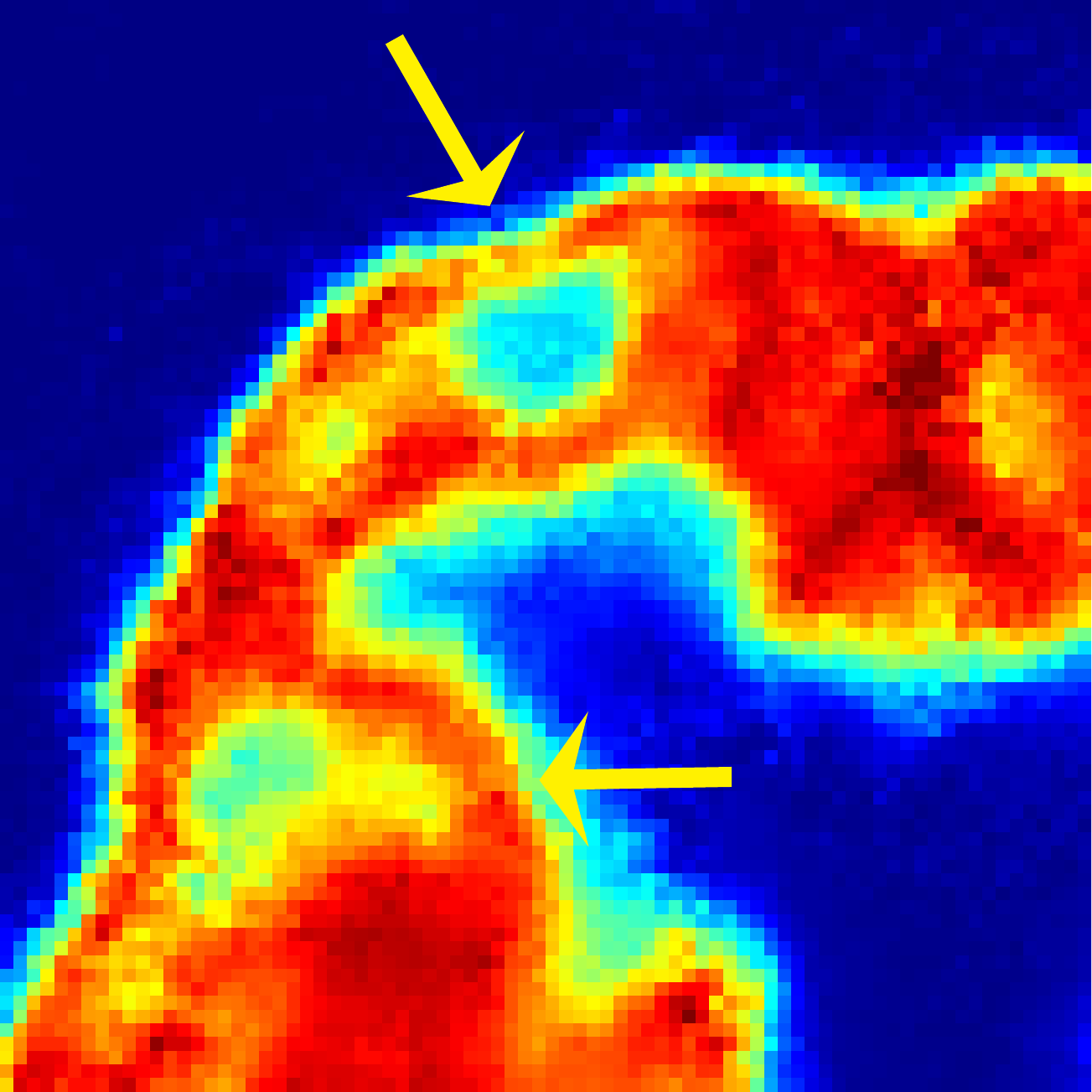}
		\end{minipage}&\hspace{-0.45cm}
		\begin{minipage}{3cm}
			\includegraphics[width=3cm]{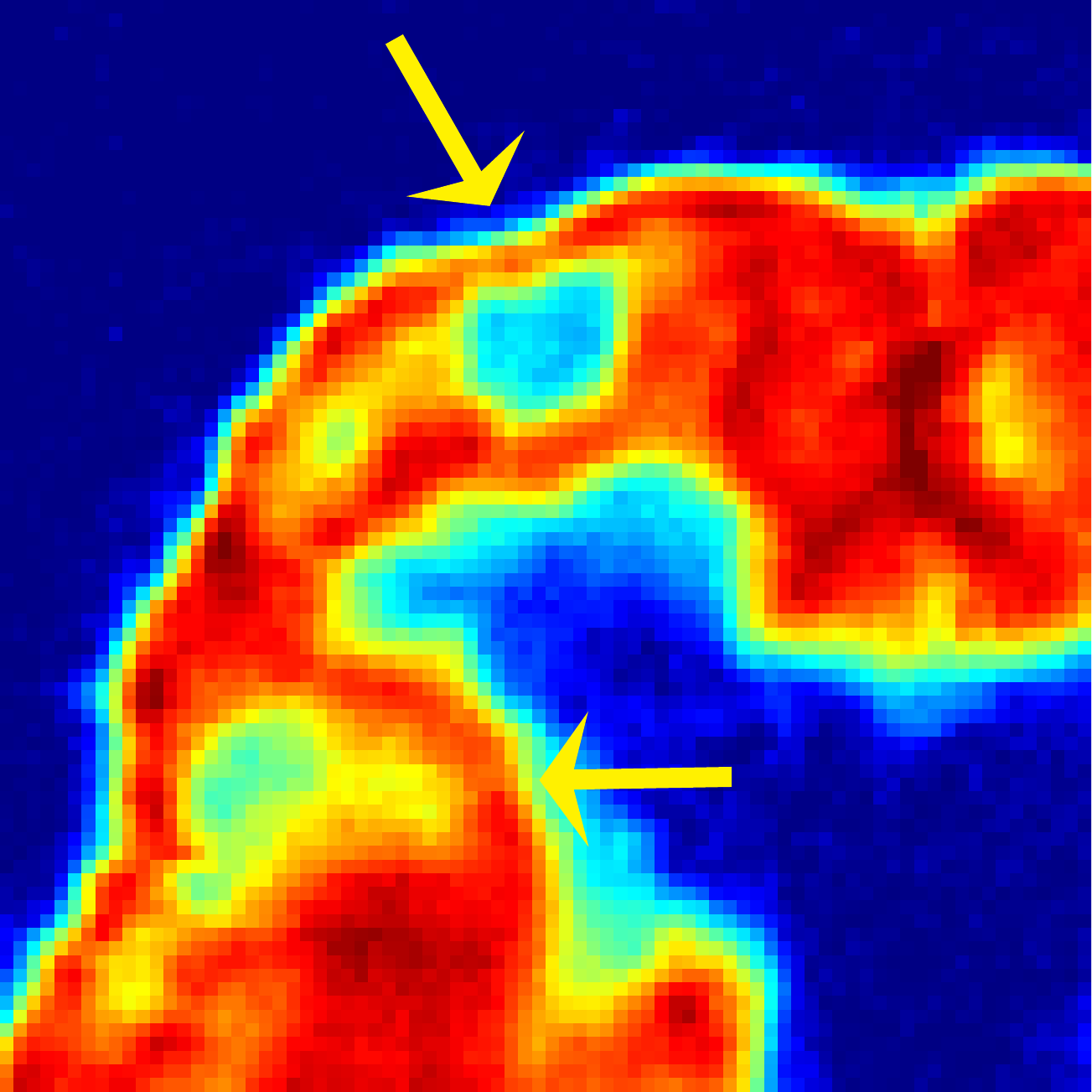}
		\end{minipage}&\hspace{-0.45cm}
		\begin{minipage}{3cm}
			\includegraphics[width=3cm]{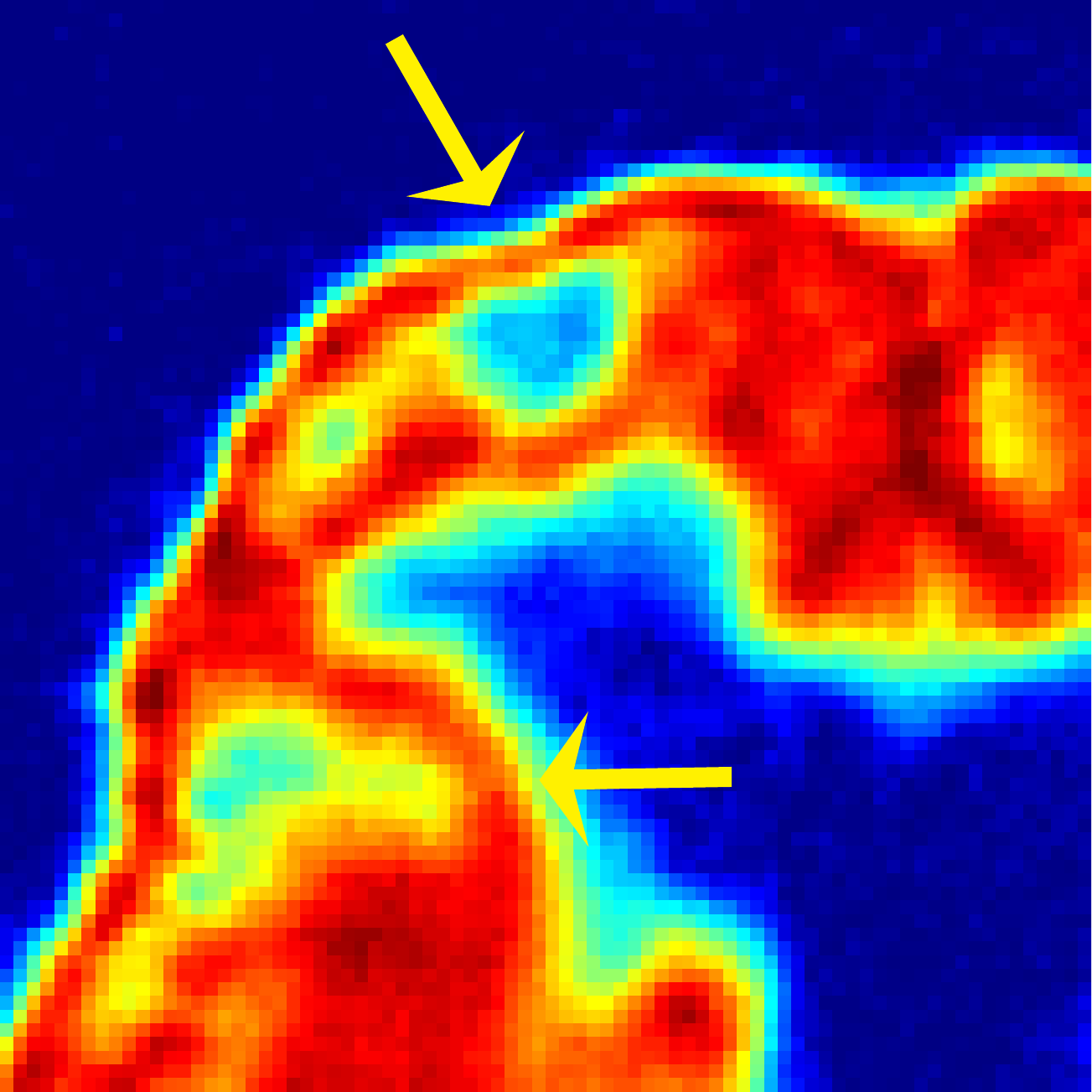}
		\end{minipage}\\
{\small{QPLS \cite{M.J.Ehrhardt2015}}}&\hspace{-0.45cm}
		{\small{JAnal \eqref{JAnalPETMRI}}}&\hspace{-0.45cm}
		{\small{JSTF \eqref{Proposed}}}&\hspace{-0.45cm}
		{\small{JSDDTF \eqref{DDTFPETMRI}}}\\
		\begin{minipage}{3cm}
			\includegraphics[width=3cm]{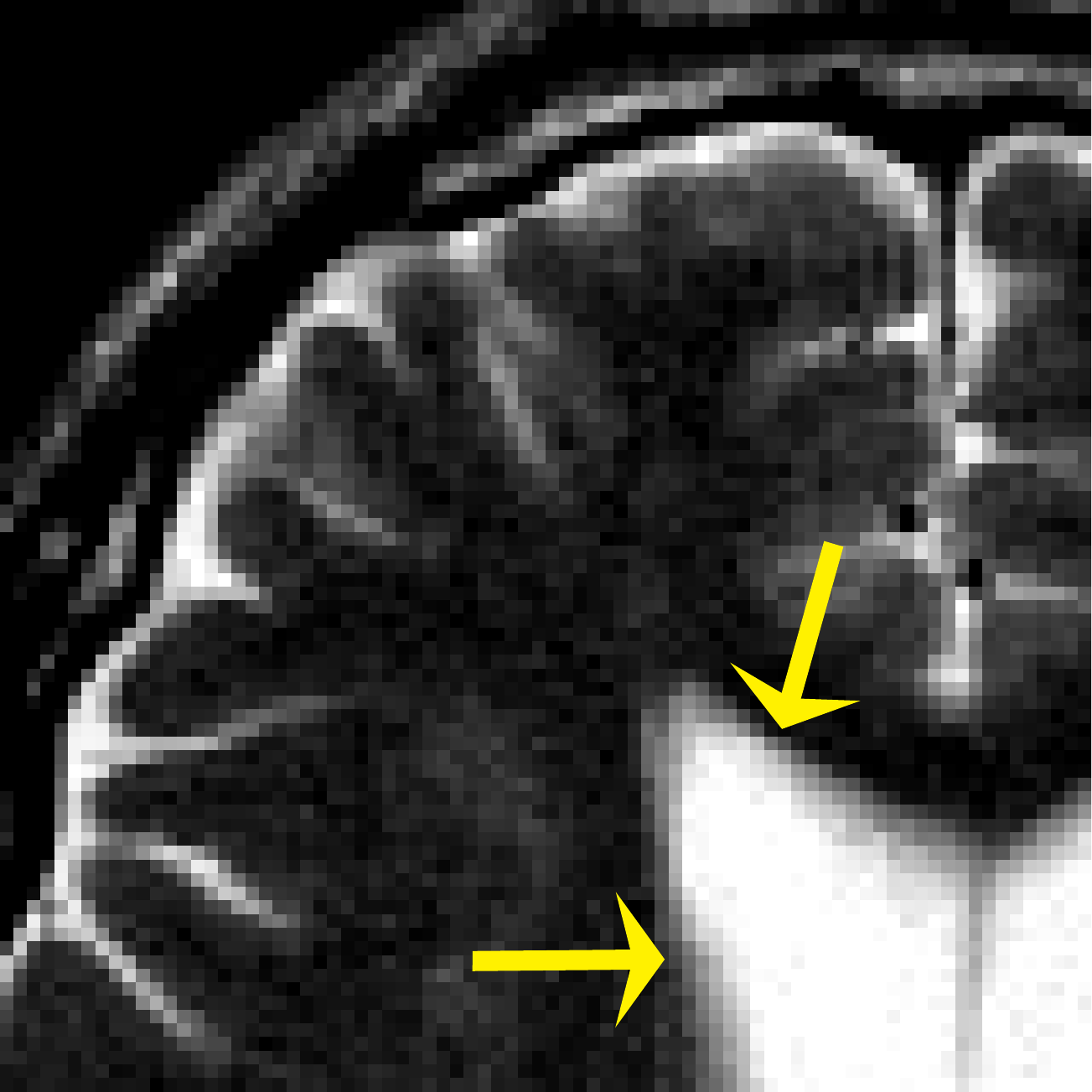}
		\end{minipage}&\hspace{-0.45cm}
		\begin{minipage}{3cm}
			\includegraphics[width=3cm]{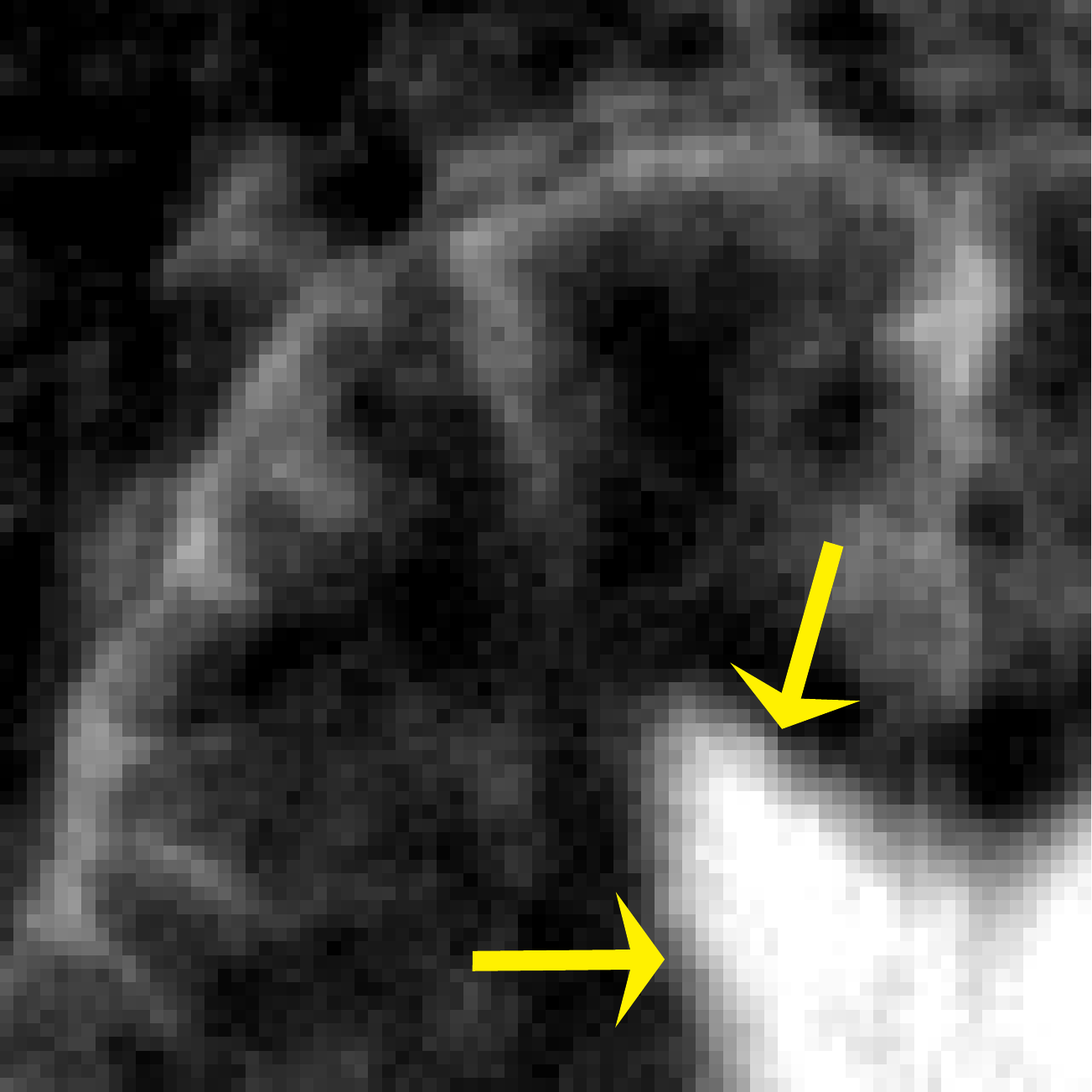}
		\end{minipage}&\hspace{-0.45cm}
		\begin{minipage}{3cm}
			\includegraphics[width=3cm]{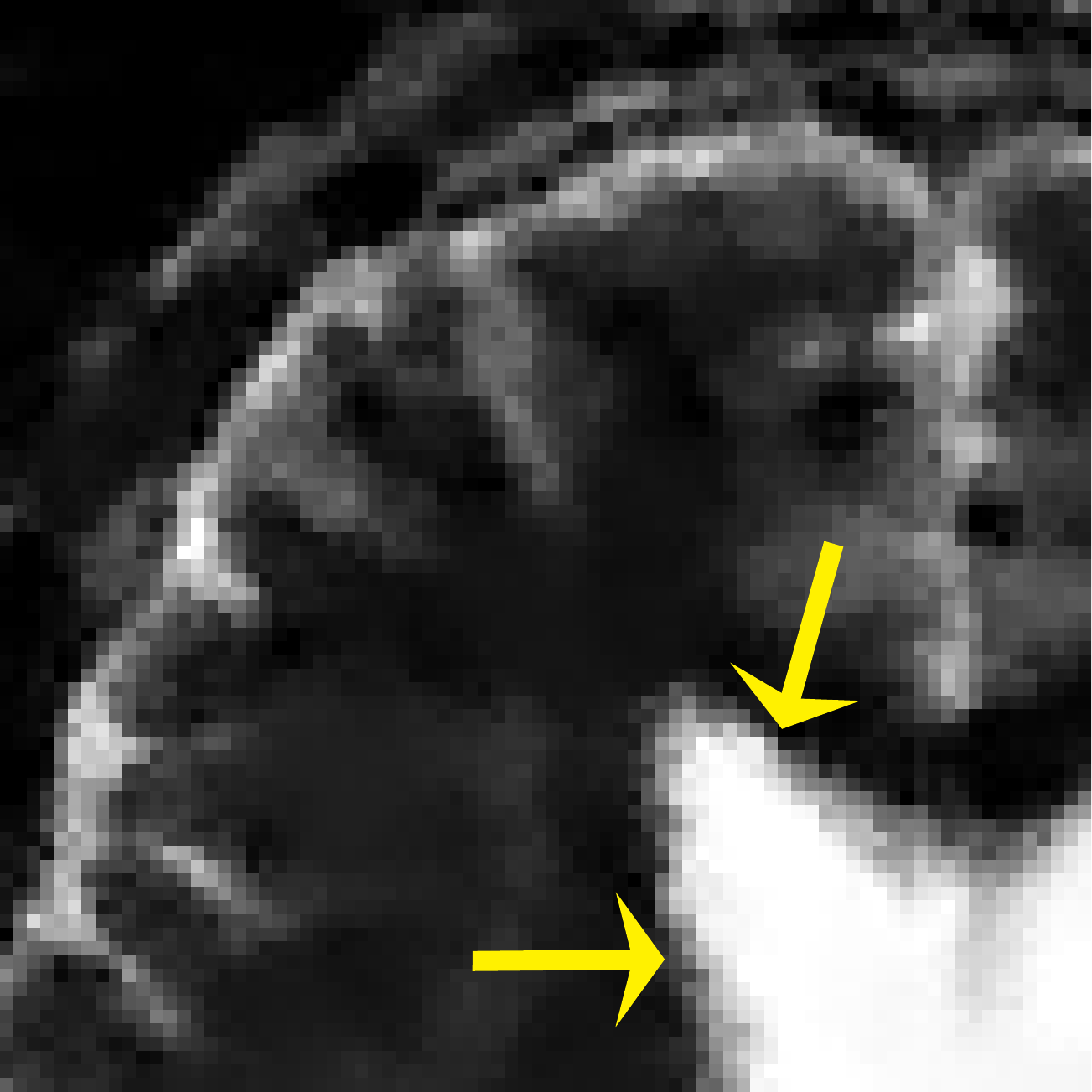}
		\end{minipage}&\hspace{-0.45cm}
\begin{minipage}{3cm}
\includegraphics[width=3cm]{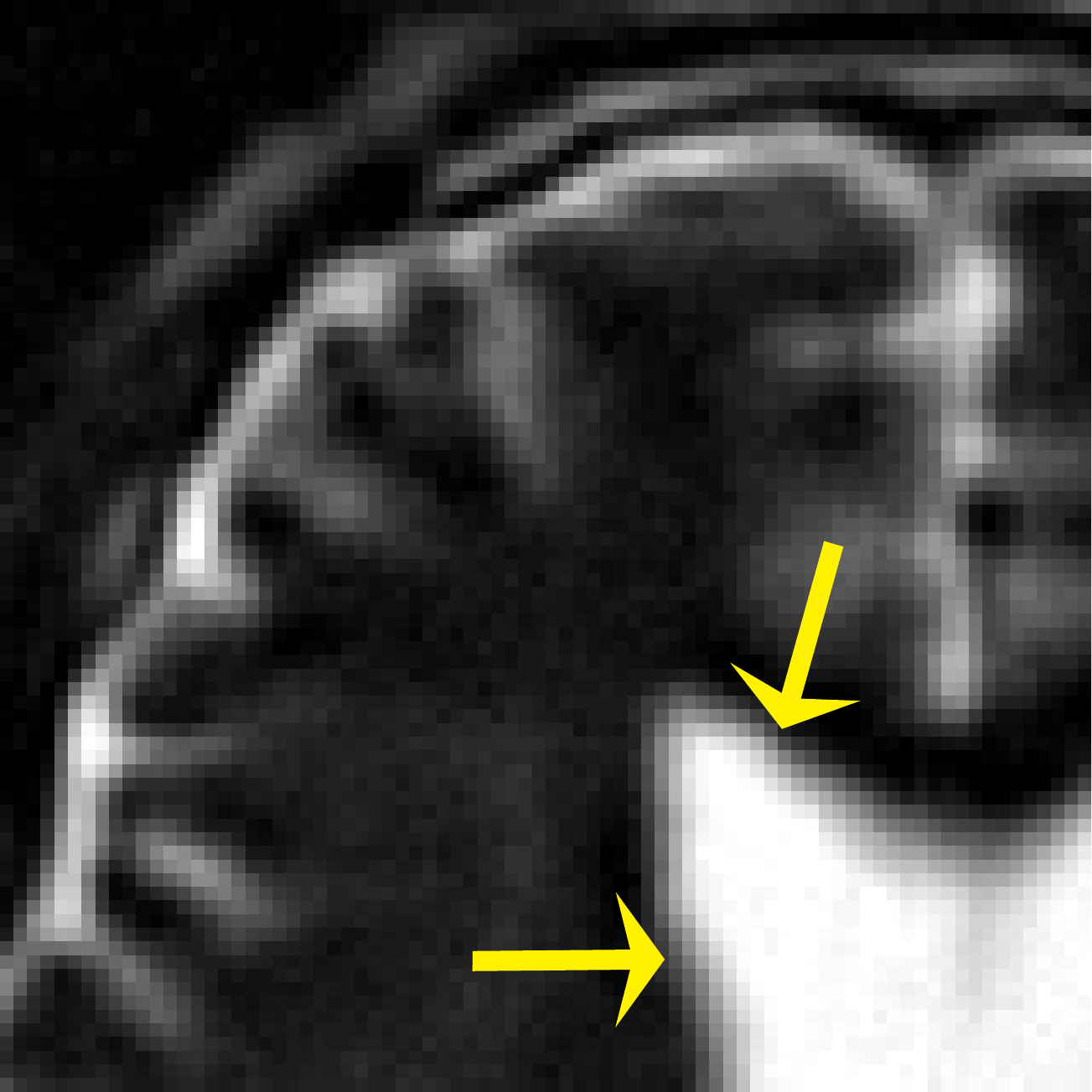}
\end{minipage}\\
		{\small{Original}}&\hspace{-0.45cm}
{\small{Initial}}&\hspace{-0.45cm}
{\small{Analysis \eqref{AnaMRI}}}&\hspace{-0.45cm}
		{\small{DDTF \eqref{DDTFMRI}}}\\
\begin{minipage}{3cm}
			\includegraphics[width=3cm]{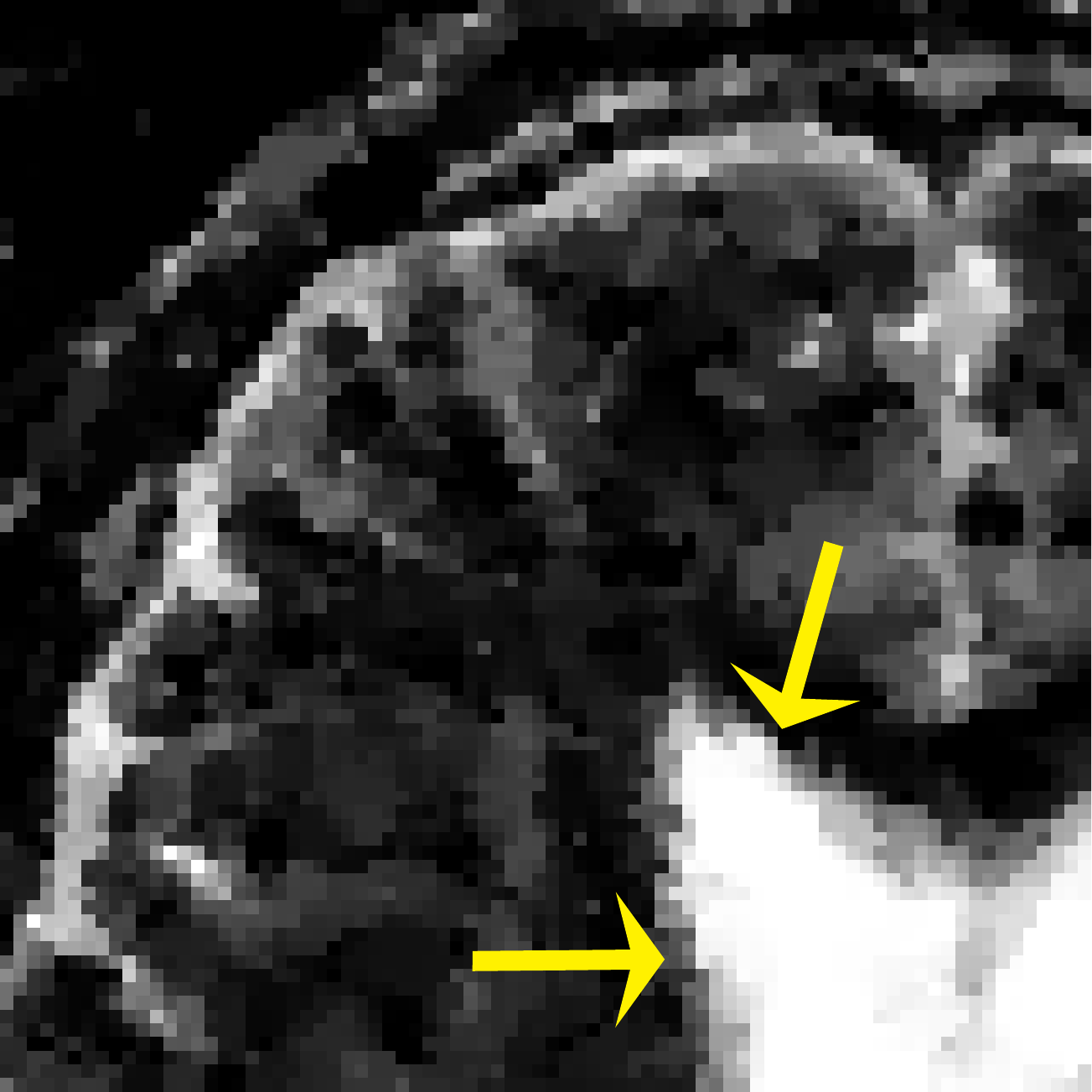}
		\end{minipage}&\hspace{-0.45cm}
		\begin{minipage}{3cm}
			\includegraphics[width=3cm]{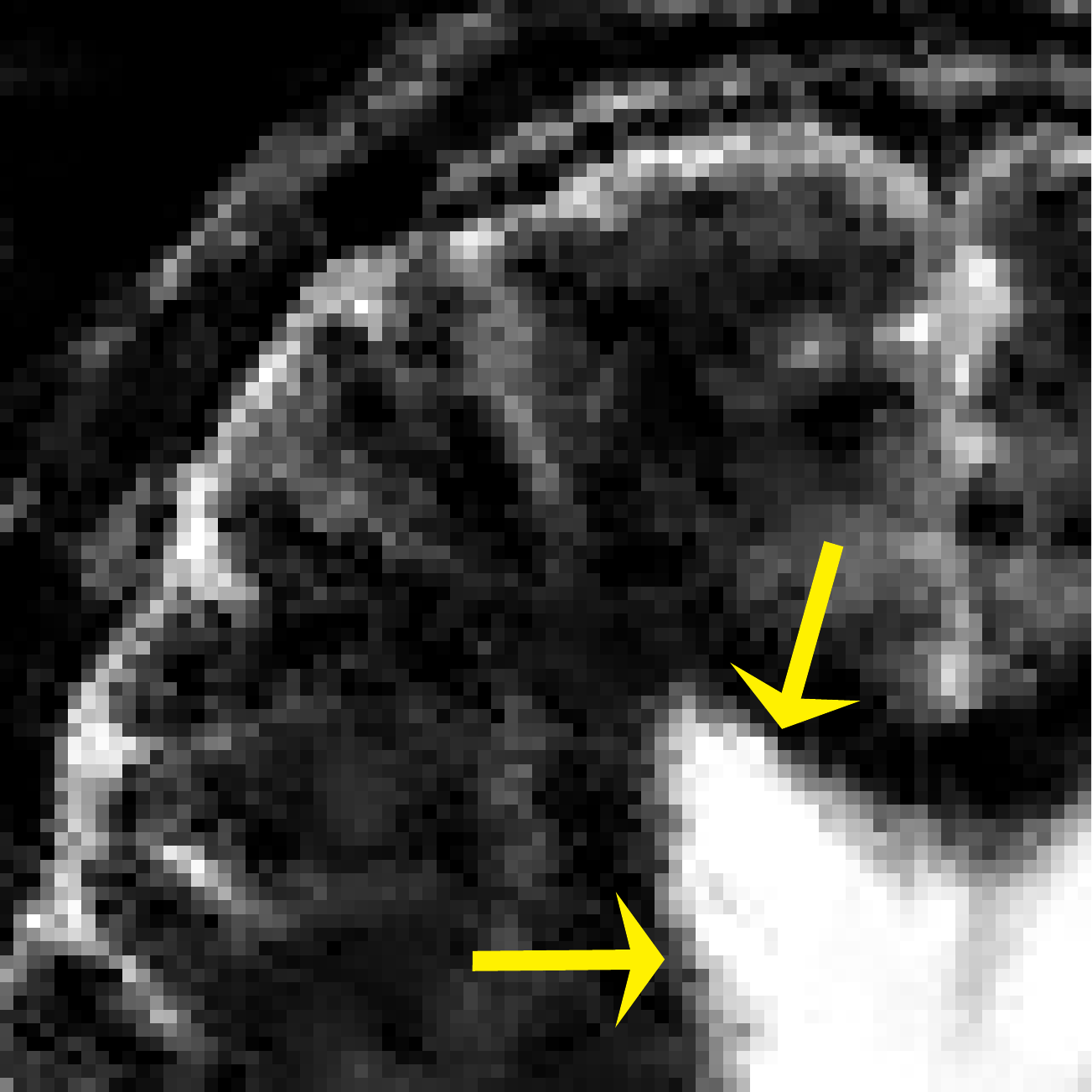}
		\end{minipage}&\hspace{-0.45cm}
		\begin{minipage}{3cm}
			\includegraphics[width=3cm]{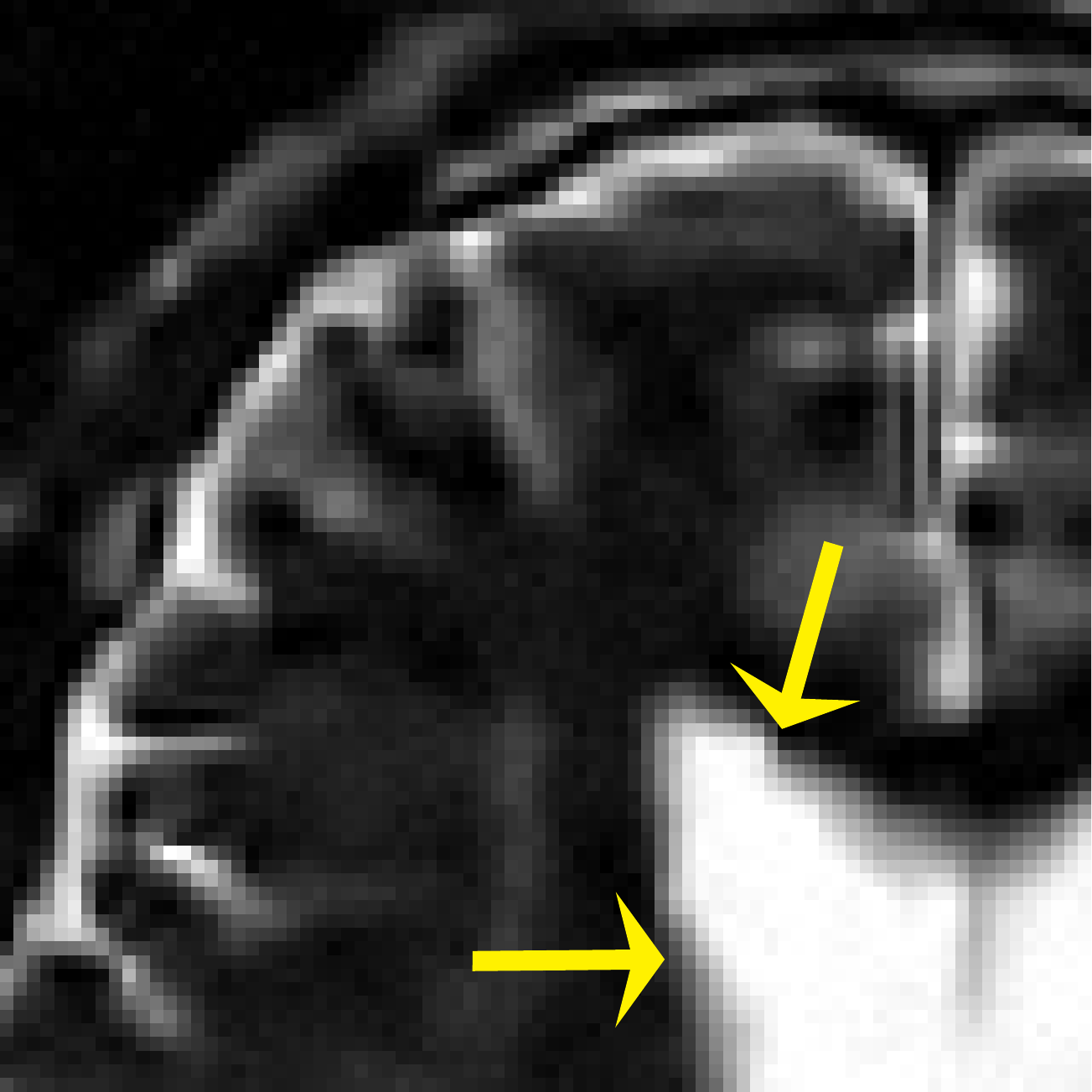}
		\end{minipage}&\hspace{-0.45cm}
		\begin{minipage}{3cm}
			\includegraphics[width=3cm]{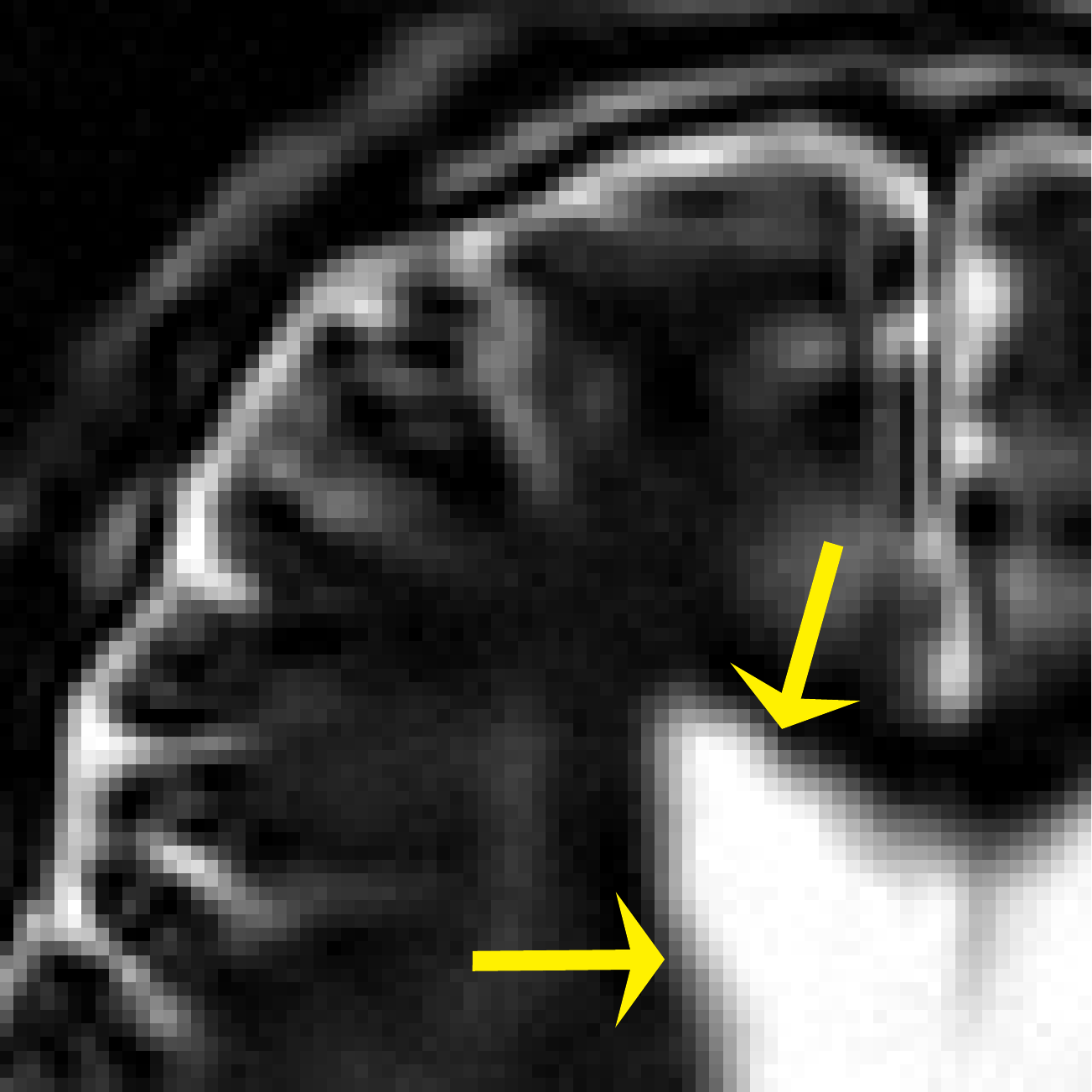}
		\end{minipage}\\
{\small{QPLS \cite{M.J.Ehrhardt2015}}}&\hspace{-0.45cm}
		{\small{JAnal \eqref{JAnalPETMRI}}}&\hspace{-0.45cm}
		{\small{JSTF \eqref{Proposed}}}&\hspace{-0.45cm}
		{\small{JSDDTF \eqref{DDTFPETMRI}}}
	\end{tabular}
	\caption{Zoom-in views of \cref{PETMRT215RadialResults} The first and second rows describe the PET images, and the third and fourth rows depict the MRI images. The yellow arrows indicate the regions worth noticing.}\label{PETMRT215RadialResultsZoom}
\end{figure}

\section{Conclusions}\label{Conclusion}

In this paper, we proposed a joint sparsity based tight frame regularization PET-MRI joint reconstruction model, together with a proximal alternating minimization algorithm. The numerical experiments show that our models \eqref{Proposed} and \eqref{DDTFPETMRI} both outperform the existing models in \cite{M.J.Ehrhardt2015}. This performance gain of our proposed models mainly comes from taking the different regularity of the different modality images into the consideration, as well as the structural correlation. Finally, our convergence analysis demonstrates that the sequence generated by our algorithm globally converges to a critical point of the proposed model.

\section*{Acknowledgments} We would like to thank the anonymous reviewers for their constructive suggestions and comments that helped tremendously with improving the presentations of this paper.

\bibliographystyle{siamplain}

\end{document}